\documentclass[a4paper,10pt]{amsart}
\usepackage{pb-diagram}
\usepackage{xypic}
\usepackage{amscd}
\usepackage{bm}
\usepackage{amsmath,amssymb,amscd,bbm,amsthm,mathrsfs,hyperref}
\usepackage{latexsym}
\usepackage{amsfonts}
\usepackage{amsthm}
\usepackage{indentfirst}
 \newtheorem{thm}{Theorem}[section]
 
 \newtheorem{lem}[thm]{Lemma}
 \newtheorem{prop}[thm]{Proposition}
 \newtheorem{defn}[thm]{Definition}
 \newtheorem{ex}[thm]{Example}
 \newtheorem{rem}[thm]{Remark}

 \newtheorem{conj}[thm]{Conjecture}
 \def\k{\mathbbm{k}}

 \newcommand{\Hom}{\mathrm{Hom}}

\title{Cohomology algebras of a family of cochain DG skew polynomial algebras}
\author{X.-F. Mao}
\address{Department of Mathematics, Shanghai University, Shanghai 200444, China}
\email{xuefengmao@shu.edu.cn}

\author{H. Wang}
\address{Department of Mathematics, Shanghai University, Shanghai 200444, China}
\email{happywang97@shu.edu.cn}

\author{G.Ren}
\address{School of Economics, Shanghai University, Shanghai 200444, China}
\email{rengui@shu.edu.cn}
\date{}

\begin{document}

\begin{abstract}
Let $\mathcal{A}$ be a connected cochain DG algebra such that its underlying graded algebra $\mathcal{A}^{\#}$ is the graded skew polynomial algebra
$$k\langle x_1,x_2, x_3\rangle/\left(\begin{array}{ccc}
x_1x_2+x_2x_1\\
x_2x_3+x_3x_2\\
x_3x_1+x_1x_3\\
                                 \end {array}\right), |x_1|=|x_2|=|x_3|=1.$$
From \cite{MWZ} or \cite{MWYZ}, one sees that the differential $\partial_{\mathcal{A}}$ is determined by
\begin{align*}
\left(
                         \begin{array}{c}
                           \partial_{\mathcal{A}}(x_1)\\
                           \partial_{\mathcal{A}}(x_2)\\
                           \partial_{\mathcal{A}}(x_3)
                         \end{array}
                       \right)=M\left(
                         \begin{array}{c}
                           x_1^2\\
                           x_2^2\\
                           x_3^2
                         \end{array}
                       \right),
\end{align*}
for some  $M\in M_3(k)$.
For the case $1\le r(M)\le 3$, we compute $H(\mathcal{A})$ case by case.
The computational results in this paper give substantial support for \cite{MWZ}, where
the various homological properties of such DG algebras are systematically studied.
We find some examples, which indicate that the cohomology graded algebras of such kind of DG algebras may be not left (right) Gorenstein.

\end{abstract}

\maketitle

\section{introduction}
In the literature, Koszul, homologically smooth, Gorenstein and Calabi-Yau properties of cochain DG algebras have been frequently studied.
 In general, these homological properties are difficult to detect. For a non-trivial DG algebra $\mathcal{A}$, the trivial DG algebra $H(\mathcal{A})$ is much simpler to study since it has
zero differential.
There have been some attempts to judge the various homological properties of
$\mathcal{A}$ from $H(\mathcal{A})$. It is shown in \cite{MH} that a connected cochain DG algebra $\mathcal{A}$ is a Kozul Calabi-Yau
DG algebra if $H(\mathcal{A})$ belongs to one of the following cases:
\begin{align*}
& (a) H(A)\cong \k;  \quad \quad (b) H(A)= \k[\lceil z\rceil], z\in \mathrm{ker}(\partial_A^1); \\
& (c) H(A)= \frac{\k\langle \lceil z_1\rceil, \lceil z_2\rceil\rangle}{(\lceil z_1\rceil\lceil z_2\rceil +\lceil z_2\rceil \lceil z_1\rceil)}, z_1,z_2\in \mathrm{ker}(\partial_A^1).
\end{align*}
A more general result is proved in \cite{MYY} that $\mathcal{A}$ is Calabi-Yau if the trivial DG algebra $(H(\mathcal{A}),0)$ is Calabi-Yau.
In particular, $\mathcal{A}$ is a Calabi-Yau DG algebra if $$H(\mathcal{A})=k\langle \lceil x\rceil,\lceil y\rceil,\lceil z\rceil\rangle/\left(\begin{array}{ccc}
                                 a\lceil y\rceil \lceil z\rceil +b\lceil z\rceil \lceil y\rceil + c\lceil x\rceil^{2}\\
                               a\lceil z\rceil \lceil x\rceil +b\lceil x\rceil \lceil z\rceil + c\lceil y\rceil^{2} \\
                                 a\lceil x\rceil \lceil y\rceil +b\lceil y\rceil \lceil x\rceil+ c\lceil z\rceil^{2}
                                 \end {array}\right),$$
where $(a,b,c)\in \Bbb{P}^2_k-\mathfrak{D}$ and $x,y,z\in \mathrm{ker}(\partial_{\mathcal{A}}^1)$.
By  \cite[Proposition $6.2$]{MXYA}, $\mathcal{A}$ is not a Gorenstein DG algebra but a Koszul and homologically smooth DG algebra if $H(\mathcal{A})=\k\langle \lceil y_1\rceil, \cdots, \lceil y_n\rceil \rangle,$ for some degree $1$ cocycle elements $y_1,\cdots, y_n$  in $\mathcal{A}$. And \cite[Proposition $6.5$]{MHLX} indicates that $\mathcal{A}$ is Calabi-Yau
if $H(\mathcal{A})=\k[\lceil z_1\rceil, \lceil z_2\rceil]$, where $z_1\in \mathrm{ker}(\partial_{\mathcal{A}}^1)$ and $z_2\in \mathrm{ker}(\partial_{\mathcal{A}}^2)$.
 In \cite{MGYC}, it is proved that $\mathcal{A}$ is a Koszul homologically smooth DG algebra if $H(\mathcal{A})=\k[\lceil y_1\rceil, \cdots, \lceil y_m\rceil ],$ for some central, cocycle and degree $1$ elements $y_1,\cdots, y_m$  in $\mathcal{A}$. Moreover, $\mathcal{A}$ is $0$-Calabi-Yau if and only if $m$ is an odd integer. It is proved in \cite[Proposition 4.3]{MWZ} that $\mathcal{A}$ is a Koszul and Calabi-Yau DG algebra
 if $$H(\mathcal{A})=k\langle \lceil y_1\rceil,\lceil y_2\rceil \rangle/(t_1\lceil y_1\rceil^2+t_2\lceil y_2\rceil^2+t_3(\lceil y_1\rceil \lceil y_2\rceil +\lceil y_2\rceil \lceil y_1\rceil))$$ with $y_1,y_2\in Z^1(\mathcal{A})$ and
 $(t_1,t_2,t_3)\in \Bbb{P}_k^2-\{(t_1,t_2,t_3)|t_1t_2-t_3^2\neq 0\}$.
 These results indicate that it is worthwhile to compute the cohomology algebra of a given DG algebra if one wants to study its homological properties.

Recently, the constructions and studies on some specific family of connected cochain DG algebras have attracted much attentions. In \cite{MHLX}, \cite{MGYC} and \cite{MXYA}, DG down-up algebras, DG polynomial algebras and DG free algebras are introduced and systematically studied, respectively. It is exciting to discover that non-trivial DG down-up algebras, non-trivial DG polynomial algebras and DG free algebras with $2$ degree $1$ variables are Calabi-Yau DG algebras. It seems to be a good way to construct some interesting homologically smooth DG algebras on AS-regular algebras.

Let $\mathfrak{D}$ the subset of the projective plane $\Bbb{P}_k^2$ consisting of the $12$ points:
$$\mathfrak{D}:=\{(1,0,0), (0,1,0),(0,0,1)\}\sqcup\{(a,b,c)|a^3=b^3=c^3\}.$$
Recall that the points $(a,b,c)\in \Bbb{P}_k^2-\mathfrak{D}$ parametrize the $3$-dimensional Sklyanin algebras,
$$S_{a,b,c}=\frac{k\langle x_1,x_2,x_3\rangle}{(f_1,f_2,f_3)},$$
where \begin{align*} f_1&=ax_2x_3+bx_3x_2+cx_1^2\\
                     f_2&=ax_3x_1+bx_1x_3+cx_2^2\\
                     f_3&=ax_1x_2+bx_2x_1+cx_3^2.
                     \end{align*}
The $3$-dimensional Sklyanin algebras form the most important class of Artin-Schelter regular algebras of global dimension $3$.
We say that a cochain DG algebra $\mathcal{A}$ is a $3$-dimensional Sklyanin DG algebra if its underlying graded algebra $\mathcal{A}^{\#}$ is a $3$-dimensional Sklyanin algebra $S_{a,b,c}$, for some $(a,b,c)\in \Bbb{P}_k^2-\mathfrak{D}$.
In \cite{MWYZ}, all possible differential structures on $3$-dimensional DG Sklyanin algebras are classified.
By \cite[Theorem A]{MWYZ}, $\partial_{\mathcal{A}}=0$ when $|a|\neq |b|$ or $c\neq 0$. Note that $\partial_{\mathcal{A}}\neq 0$ only if either $a=b, c=0$ or $a=-b,c=0$. When $a=-b,c=0$, the $3$-dimensional DG Sklyanin algebras $\mathcal{A}$ is just a DG Polynomial algebra, which is systematically studied in \cite{MGYC}. For the case $a=b,c=0$,  the differential $\partial_{\mathcal{A}}$ is defined by
\begin{align*}
\left(
                         \begin{array}{c}
                           \partial_{\mathcal{A}}(x_1)\\
                           \partial_{\mathcal{A}}(x_2)\\
                           \partial_{\mathcal{A}}(x_3)
                         \end{array}
                       \right)=M\left(
                         \begin{array}{c}
                           x_1^2\\
                           x_2^2\\
                           x_3^2
                         \end{array}
                       \right), \text{for some}\, M\in M_3(k).
\end{align*}
In this case, the $3$-dimensional DG Sklyanin algebra is just  $\mathcal{A}_{\mathcal{O}_{-1}(k^3)}(M)$ in \cite{MWZ}.
Note that such $3$-dimensional DG Sklyanin algebras are actually a family of cochain DG skew polynomial algebras.
The motivation of this paper is to compute $H(\mathcal{A})$ when $1\le r(M)\le 3$.

For any $M\in M_2(k)$, one sees that $H[\mathcal{A}_{\mathcal{O}_{-1}(k^2)}(M)]$ is always AS-Gorenstein
by \cite{Mao}.  And each DG algebra $\mathcal{A}_{\mathcal{O}_{-1}(k^2)}(M)$ is a Koszul Calabi-Yau DG algebra by
\cite[Theorem C]{MH}. It is natural for us to put forward the following conjecture.
\begin{conj}\label{biproduct}
For any  $M\in M_3(k)$, $H(\mathcal{A}_{\mathcal{O}_{-1}(k^3)}(M))$ is a left (right) Gorenstein graded algebra.
\end{conj}
Finally, we give a concrete counterexample to disprove Conjecture \ref{biproduct}(see Example \ref{countex}).
More generally, we have the following theorem (see Theorem \ref{nonasgoren}). \\
\begin{bfseries}
Theorem \ A.
\end{bfseries}
Let $\mathcal{A}$ be a connected cochain DG algebra such that
$$\mathcal{A}^{\#}=k\langle x_1,x_2, x_3\rangle/\left(\begin{array}{ccc}
x_1x_2+x_2x_1\\
x_2x_3+x_3x_2\\
x_3x_1+x_1x_3\\
                                 \end {array}\right), |x_1|=|x_2|=|x_3|=1,$$
and $\partial_A$ is determined by
\begin{align*}
\left(
                         \begin{array}{c}
                           \partial_{\mathcal{A}}(x_1)\\
                           \partial_{\mathcal{A}}(x_2)\\
                           \partial_{\mathcal{A}}(x_3)
                         \end{array}
                       \right)=N\left(
                         \begin{array}{c}
                           x_1^2\\
                           x_2^2\\
                           x_3^2
                         \end{array}
                       \right).
\end{align*}
Then the graded algebra $H(\mathcal{A})$ is not left (right) Gorenstein if and only if
there exists some  $C=(c_{ij})_{3\times 3}\in \mathrm{QPL}_3(k)$ satisfying $N=C^{-1}M(c_{ij}^2)_{3\times 3}$,
where
$$M=\left(
                                 \begin{array}{ccc}
                                   1 & 1 & 0 \\
                                   1 & 1 & 0 \\
                                   1 & 1 & 0 \\
                                 \end{array}
                               \right)
\,\,\text{or}\,\,M=\left(
                                 \begin{array}{ccc}
                                   m_{11} & m_{12} & m_{13} \\
                                   l_1m_{11} & l_1m_{12} & l_1m_{13} \\
                                   l_2m_{11} & l_2m_{12} & l_2m_{13} \\
                                 \end{array}
                               \right)$$ with $m_{12}l_1^2+m_{13}l_2^2\neq m_{11}, l_1l_2\neq 0$ and $4m_{12}m_{13}l_1^2l_2^2= (m_{12}l_1^2+m_{13}l_2^2-m_{11})^2$.

----------------------------------------------------------------------
\section{preliminaries}
\subsection{Notations and conventions}
Throughout this paper, $k$ is an algebraically closed field of characteristic $0$.
For any $k$-vector space $V$, we write $V'=\Hom_{k}(V,k)$. Let $\{e_i|i\in I\}$ be a basis of a finite dimensional $k$-vector space $V$.  We denote the dual basis of $V$ by $\{e_i^*|i\in I\}$, i.e., $\{e_i^*|i\in I\}$ is a basis of $V'$ such that $e_i^*(e_j)=\delta_{i,j}$. For any graded vector space $W$ and $j\in\Bbb{Z}$, the $j$-th suspension $\Sigma^j W$ of $W$ is a graded vector space defined by $(\Sigma^j W)^i=W^{i+j}$.

 A cochain DG algebra is
a graded
$k$-algebra $\mathcal{A}$ together with a differential $\partial_{\mathcal{A}}: \mathcal{A}\to \mathcal{A}$  of
degree $1$ such that
\begin{align*}
\partial_{\mathcal{A}}(ab) = (\partial_{\mathcal{A}} a)b + (-1)^{|a|}a(\partial_{\mathcal{A}} b)
\end{align*}
for all graded elements $a, b\in \mathcal{A}$.  We write $\mathcal{A}\!^{op}$ for its opposite DG
algebra, whose multiplication is defined as
 $a \cdot b = (-1)^{|a|\cdot|b|}ba$ for all
graded elements $a$ and $b$ in $\mathcal{A}$.

Let $\mathcal{A}$ be
a cochain DG algebra.  We denote by $\mathcal{A}^i$ its $i$-th homogeneous component.  The differential $\partial_{\mathcal{A}}$ is a sequence of linear maps $\partial_{\mathcal{A}}^i: \mathcal{A}^i\to \mathcal{A}^{i+1}$ such that $\partial_{\mathcal{A}}^{i+1}\circ \partial_{\mathcal{A}}^i=0$, for all $i\in \Bbb{Z}$.  If $\partial_{\mathcal{A}}\neq 0$, $\mathcal{A}$ is called
non-trivial. The cohomology graded algebra of $\mathcal{A}$ is the graded algebra $$H(\mathcal{A})=\bigoplus_{i\in \Bbb{Z}}\frac{\mathrm{ker}(\partial_{\mathcal{A}}^i)}{\mathrm{im}(\partial_{\mathcal{A}}^{i-1})}.$$
 Let $z\in \mathrm{ker}(\partial_{\mathcal{A}}^i)$ be a cocycle element of degree $i$. We write $\lceil z \rceil$ for the cohomology class in $H(\mathcal{A})$ represented by $z$. If $\mathcal{A}^0=k$ and $\mathcal{A}^i=0, \forall i<0$, then we say that $\mathcal{A}$ is connected.  One sees that $H(\mathcal{A})$ is a connected graded algebra if $\mathcal{A}$ is a connected cochain DG algebra.
Let $\mathcal{A}$ be a connected cochain DG $k$-algebra.  We write $\frak{m}$ as the maximal DG ideal $\mathcal{A}^{>0}$ of $\mathcal{A}$.
Via the canonical surjection $\varepsilon: \mathcal{A}\to k$, $k$ is both a DG
$\mathcal{A}$-module and a DG $\mathcal{A}\!^{op}$-module. It is easy to check that the enveloping DG algebra $\mathcal{A}^e = \mathcal{A}\otimes \mathcal{A}\!^{op}$ of $\mathcal{A}$
is also a connected cochain DG algebra with $H(\mathcal{A}^e)\cong H(\mathcal{A})^e$, and $$\frak{m}_{\mathcal{A}^e}=\frak{m}_{\mathcal{A}}\otimes \mathcal{A}^{op} + \mathcal{A}\otimes
\frak{m}_{\mathcal{A}^{op}}.$$

 The derived category of left DG modules over $\mathcal{A}$ (DG $\mathcal{A}$-modules for short) is denoted by $\mathrm{D}(\mathcal{A})$.  A DG $\mathcal{A}$-module  $M$ is compact if the functor $\Hom_{\mathrm{D}(A)}(M,-)$ preserves
all coproducts in $\mathrm{D}(\mathcal{A})$.
 By \cite[Proposition 3.3]{MW1},
a DG $\mathcal{A}$-module  is compact if and only if it admits a minimal semi-free resolution with a finite semi-basis. The full subcategory of $\mathrm{D}(\mathcal{A})$ consisting of compact DG $\mathcal{A}$-modules is denoted by $\mathrm{D^c}(\mathcal{A})$.

In the rest of this subsection, we review some important homological properties for DG algebras.
\begin{defn}\label{basicdef}
{\rm Let $\mathcal{A}$ be a connected cochain DG algebra.
\begin{enumerate}
\item  If $\dim_{k}H(R\Hom_{\mathcal{A}}(\k,\mathcal{A}))=1$, then $\mathcal{A}$ is called Gorenstein (cf. \cite{FHT1});
\item  If ${}_{\mathcal{A}}k$, or equivalently ${}_{\mathcal{A}^e}\mathcal{A}$, has a minimal semi-free resolution with a semi-basis concentrated in degree $0$, then $\mathcal{A}$ is called Koszul (cf. \cite{HW});
\item If ${}_{\mathcal{A}}k$, or equivalently the DG $\mathcal{A}^e$-module $\mathcal{A}$ is compact, then $\mathcal{A}$ is called homologically smooth (cf. \cite[Corollary 2.7]{MW3});
\item If $\mathcal{A}$ is homologically smooth and $$R\Hom_{\mathcal{A}^e}(\mathcal{A}, \mathcal{A}^e)\cong
\Sigma^{-n}\mathcal{A}$$ in  the derived category $\mathrm{D}((\mathcal{A}^e)^{op})$ of right DG $\mathcal{A}^e$-modules, then $\mathcal{A}$ is called an $n$-Calabi-Yau DG algebra  (cf. \cite{Gin,VdB}).
\end{enumerate}}
 \end{defn}

\subsection{AS-Gorenstein (AS-regular) graded algebras}
In this subsection, we let $A$ be a connected graded algebra. We have the following definitions on AS-Gorenstein graded algebras and AS-regular graded algebras.
\begin{defn}
  We say that $A$ is
 left (resp. right)
Gorenstein if $\dim_k\mathrm{Ext}_A^*(k,A) = 1$ (resp.
$\dim_k\mathrm{Ext}_{A\!^{op}}^*(k,A) = 1$), where
$\mathrm{Ext}_A^*(k,A) = \oplus_{i \in
\mathbb{Z}}\mathrm{Ext}_A^i(k,A).$ For a
 left Gorenstein graded algebra $A$,  there is some integer $l$ such that
\begin{equation}\label{gorenstein}
\mathrm{Ext}_A^i(k,A)=
\begin{cases}0, &  i \neq \mathrm{depth}_AA,
  \\
 k(l), & i=\mathrm{depth}_AA.
\end{cases}
\end{equation}
A left (resp. right) Gorenstein graded algebra $A$ is called left
(resp. right) AS-Gorenstein (AS stands for Artin-Schelter) if
$\mathrm{id}_AA<\infty$ (resp. $\mathrm{id}_{A^{op}}A<\infty$). If
further, $\mathrm{gl.dim}A<\infty$, then we say $A$ is left (resp.
right) AS-regular.

\end{defn}

\begin{lem}\label{extaslem}
Let $A$ be a Noetherian and AS-Gorenstein graded algebra. Then the graded algebra $B=A[x]$ with $|x|=2$ is also a Noetherian and AS-Gorenstein graded algebra.
\end{lem}
\begin{proof}
By the well-known `Hilbert basis Theorem', one sees that $B$ is Noetherian. We have $B=A\otimes k[x]$.
Let $P$ and $Q$ be the  finitely generated
minimal free resolutions of ${}_{\mathcal{A}}k$ and ${}_{k[x]}k$ respectively. Then
$P\otimes Q$ is a finitely generated minimal free resolution of ${}_{\mathcal{B}}k$. We have
\begin{align*}
H(\Hom_{B}(P\otimes Q, B))&=H(\Hom_{A\otimes k[x]}(P\otimes Q, A\otimes k[x])) \\
& \cong H(\Hom_{A}(P, \Hom_{k[x]}(Q,A\otimes k[x])))\\
& \cong H(\Hom_{A}(P, A\otimes \Hom_{k[x]}(Q,k[x])) \\
& \cong H(\Hom_{A}(P,A)\otimes \Hom_{k[x]}(Q,k[x]))\\
& \cong H(\Hom_{A}(P,A))\otimes H(\Hom_{k[x]}(Q,k[x])).
\end{align*}
Since $A$ and $k[x]$ are both AS-Gorenstein, we have $$\dim_k\mathrm{Ext}_B^*(k,B)=\dim_kH(\Hom_{B}(P\otimes Q, B))=1.$$
Thus $B=A[x]$ is left AS-Gorenstein. We can similarly show that $B=A[x]$ is right AS-Gorenstein.
\end{proof}

\begin{lem}\label{nongorone}
Let $A$ be a connected graded algebra such that $$A=\frac{k\langle x,y\rangle}{(ax^2+\sqrt{ab}(xy+yx)+by^2)}, ab>0, |x|=|y|=1.$$
Then $A$ is not left (right) Gorenstein.
\end{lem}
\begin{proof}
The trivial module ${}_Ak$ admits a finitely generated  minimal free resolution
$$\cdots \xrightarrow{d_{n+1}} F_{n} \xrightarrow{d_n} F_{n-1}\xrightarrow{d_{n-1}}\cdots \xrightarrow{d_3}  F_2\xrightarrow{d_2}F_1=Ae_x\oplus Ae_y
\xrightarrow{d_1} A\xrightarrow{\varepsilon} {}_Ak\rightarrow 0,$$
where \begin{align*}
&F_{n-1}=Ae_{n-1}, d_n(e_n)=(ax+\sqrt{ab}y)e_{n-1}, n\ge 3;\\
&d_2(e_2)=(ax+\sqrt{ab}y)e_x+(\sqrt{ab}x+by)e_y, d_1(e_x)=x, d_1(e_y)=y.\\
\end{align*}
Acting the functor $\Hom_{A}(-,A)$ on the deleted complex of the minimal free resolution above, we obtain the complex
$$0\to 1^*A \xrightarrow{d_1^*}e_x^*A\oplus e_y^*A  \xrightarrow{d_2^*} e_2^*A\xrightarrow{d_3^*}e_3^*A \xrightarrow{d_4^*}\cdots \xrightarrow{d_n^*}e_n^*A \xrightarrow{d_{n+1}^*}\cdots, $$
where
\begin{align*}
&d_1^*(1^*)=e_x^*x+e_y^*y; d_2^*(e_x^*)=e_r^*(ax+\sqrt{ab}y), d_2^*(e_y^*)=e_r^*(\sqrt{ab}x+by);\\
&d_{i+1}^*(e_i^*)=e_{i+1}^*(ax+\sqrt{ab}y), i\ge 2.
\end{align*}
 We have
 \begin{align*}
 &\mathrm{Ext}_A^0(k,A)=\mathrm{ker}(d_1^*)=0;\\
 &\mathrm{Ext}_A^1(k,A)=\frac{\mathrm{ker}(d_2^*)}{\mathrm{im}(d_1^*)}=\frac{(\sqrt{\frac{b}{a}}e_x^*-e_y^*)A\oplus (e_x^*x+e_y^*y)A}{(e_x^*x+e_y^*y)A}\cong(\sqrt{\frac{b}{a}}e_x^*-e_y^*)A;\\
 &\mathrm{Ext}_A^i(k,A)=\frac{\mathrm{ker}(d_{i+1}^*)}{\mathrm{im}(d_i^*)}=\frac{e_i^*(ax+\sqrt{ab}y)A}{e_i^*(ax+\sqrt{ab}y)A}=0, i\ge 2.
 \end{align*}
Obviously, $\dim_k \mathrm{Ext}_A^*(k,A)\neq 1$ and hence $A$ is not left Gorenstein. Similarly, we can show that $A$ is not right Gorenstein.

\end{proof}

\begin{lem}\label{nongorsec}
Let $A$ be a connected graded algebra such that $$A=\frac{k\langle x,y\rangle}{(ax^2+by^2)}, ab=0,(a,b)\neq (0,0), |x|=|y|=1.$$
Then $A$ is not left (right) Gorenstein.
\end{lem}
\begin{proof}
Without the loss of generality, we assume that $a=0,b\neq 0$.
The trivial module ${}_Ak$ admits a finitely generated  minimal free resolution
$$\cdots \xrightarrow{d_{n+1}} F_{n} \xrightarrow{d_n} F_{n-1}\xrightarrow{d_{n-1}}\cdots \xrightarrow{d_3}  F_2\xrightarrow{d_2}F_1=Ae_x\oplus Ae_y
\xrightarrow{d_1} A\xrightarrow{\varepsilon} {}_Ak\rightarrow 0,$$
where \begin{align*}
&F_n=Ae_n, d_n(e_n)=(by)e_{n-1}, n\ge 3;\\
&d_2(e_2)=(by)e_y, d_1(e_x)=x, d_1(e_y)=y.\\
\end{align*}
Acting the functor $\Hom_{A}(-,A)$ on the deleted complex of the minimal free resolution above, we obtain the complex
$$0\to 1^*A \xrightarrow{d_1^*}e_x^*A\oplus e_y^*A  \xrightarrow{d_2^*} e_2^*A\xrightarrow{d_3^*}e_3^*A \xrightarrow{d_4^*}\cdots \xrightarrow{d_n^*}e_n^*A \xrightarrow{d_{n+1}^*}\cdots, $$
where
\begin{align*}
&d_1^*(1^*)=e_x^*x+e_y^*y; d_2^*(e_x^*)=0, d_2^*(e_y^*)=e_r^*(by);\\
&d_{i+1}^*(e_i^*)=e_{i+1}^*(by), i\ge 2.\\
\end{align*}
 \begin{align*}
 &\mathrm{Ext}_A^0(k,A)=\mathrm{ker}(d_1^*)=0;\\
 &\mathrm{Ext}_A^1(k,A)=\frac{\mathrm{ker}(d_2^*)}{\mathrm{im}(d_1^*)}=\frac{e_x^*A\oplus (e_x^*x+e_y^*y)A}{(e_x^*x+e_y^*y)A}\cong e_x^*A;\\
 &\mathrm{Ext}_A^i(k,A)=\frac{\mathrm{ker}(d_{i+1}^*)}{\mathrm{im}(d_i^*)}=\frac{e_i^*(by)A}{e_i^*(by)A}=0, i\ge 2.
 \end{align*}
Since $\dim_k \mathrm{Ext}_A^*(k,A)\neq 1$,  $A$ is not left Gorenstein. Similarly, we can show that $A$ is not right Gorenstein.

\end{proof}

\section{some basic lemmas}
In this section, we give some simple lemmas, which will be used in the subsequent computations. If there is no special assumption is emphasized, we
let $\mathcal{A}$ be a DG Sklyanin algebra with $\mathcal{A}^{\#}=S_{a,a,0}$, and $\partial_{\mathcal{A}}$ is determined by a matrix $M$ in $M_3(k)$.
\begin{lem}\label{centcocy}
For any $t\in \Bbb{N}$, $x_1^{2t}, x_2^{2t}, x_3^{2t}$ are cocycle central elements of $\mathcal{A}$.
\end{lem}
\begin{proof}
One sees that $x_i^2$ is a central element of $\mathcal{A}$ since $$x_i^2x_j=x_ix_ix_j=-x_ix_jx_i=x_jx_i^2, $$ when $i\neq j$. This implies that each
$x_i^{2t}$ is a central element of  $\mathcal{A}$.
We have \begin{align*}
\partial_{\mathcal{A}}(x_i^2)&=\partial_{\mathcal{A}}(x_i)x_i-x_i\partial_{\mathcal{A}}(x_i)\\
                             &=\sum\limits_{j=1}^nm_{ij}x_j^2x_i-x_i\sum\limits_{j=1}^nm_{ij}x_j^2\\
                             &=\sum\limits_{j=1}^nm_{ij}(x_j^2x_i-x_ix_j^2)=0.
\end{align*}
Using this, we can inductively prove $\partial_{\mathcal{A}}(x_i^{2t})=0$.
\end{proof}

\begin{lem}\label{coboundary}
Let $\Omega$ be a coboundary element in $\mathcal{A}$ of degree $d\ge 3$.

$(1)$ If $d=2l+1$ is odd, then $\Omega=\partial_{\mathcal{A}}[x_1x_2f+x_1x_3g+x_2x_3h]$, where $f, g$ and $h$ are all linear combinations of monomials with non-negative even exponents.

$(2)$ If $d=2l$ is even, then $\Omega=\partial_{\mathcal{A}}[x_1f+x_2g+x_3h+x_1x_2x_3u]$, where $f, g$, $h$ and $u$ are all linear combinations of monomials with non-negative even exponents.
\end{lem}
\begin{proof}
By the assumption, we have $$\Omega=\partial_{\mathcal{A}}[\sum_{\substack{
l_1+l_2+l_3=d-1\\
l_1,l_2,l_3\ge 0
}}
C_{l_1,l_2,l_3}x_1^{l_1}x_2^{l_2}x_3^{l_3}].$$
If $d=2l+1$ is odd, then $d=2l$ is even.  Since \begin{align*}
&\quad\quad \sum_{\substack{
l_1+l_2+l_3=d-1\\
l_1,l_2,l_3\ge 0
}}
C_{l_1,l_2,l_3}x_1^{l_1}x_2^{l_2}x_3^{l_3}\\
&=\sum_{\substack{
l_1+l_2+l_3=d-1\\
l_1,l_2,l_3\ge 0\\
l_1,l_2 \,\text{are odd},\, l_3\, \text{is even}
}}
C_{l_1,l_2,l_3}x_1^{l_1}x_2^{l_2}x_3^{l_3}+\sum_{\substack{
l_1+l_2+l_3=d-1\\
l_1,l_2,l_3\ge 0\\
l_1,l_3 \,\text{are odd},\, l_2\, \text{is even}
}}
C_{l_1,l_2,l_3}x_1^{l_1}x_2^{l_2}x_3^{l_3}\\
&+\sum_{\substack{
l_1+l_2+l_3=d-1\\
l_1,l_2,l_3\ge 0\\
l_2,l_3 \,\text{are odd},\, l_1\, \text{is even}
}}
C_{l_1,l_2,l_3}x_1^{l_1}x_2^{l_2}x_3^{l_3}
+\sum_{\substack{
l_1+l_2+l_3=d-1\\
l_1,l_2,l_3\ge 0\\
l_1,l_2,l_3\, \text{are even}
}}
C_{l_1,l_2,l_3}x_1^{l_1}x_2^{l_2}x_3^{l_3},
\end{align*}
 we have
\begin{align*}
\Omega &=\partial_{\mathcal{A}}[\sum_{\substack{
l_1+l_2+l_3=d-1\\
l_1,l_2,l_3\ge 0
}}
C_{l_1,l_2,l_3}x_1^{l_1}x_2^{l_2}x_3^{l_3}]\\
&=\partial_{\mathcal{A}}[x_1x_2\sum_{\substack{
l_1+l_2+l_3=d-1\\
l_1,l_2,l_3\ge 0\\
l_1,l_2 \,\text{are odd},\, l_3 \text{is even}
}}
C_{l_1,l_2,l_3}x_1^{l_1-1}x_2^{l_2-1}x_3^{l_3}] \\
&+\partial_{\mathcal{A}}[x_1x_3\sum_{\substack{
l_1+l_2+l_3=d-1\\
l_1,l_2,l_3\ge 0\\
l_1,l_3 \,\text{are odd},\, l_2 \text{is even}
}}
C_{l_1,l_2,l_3}x_1^{l_1-1}x_2^{l_2}x_3^{l_3-1}] \\
&+\partial_{\mathcal{A}}[x_2x_3\sum_{\substack{
l_1+l_2+l_3=d-1\\
l_1,l_2,l_3\ge 0\\
l_2,l_3 \,\text{are odd},\, l_1 \text{is even}
}}
C_{l_1,l_2,l_3}x_1^{l_1}x_2^{l_2-1}x_3^{l_3-1}]
\end{align*}
by Lemma \ref{centcocy}.
Let \begin{align*}f&=\sum\limits_{\substack{
l_1+l_2+l_3=d-1\\
l_1,l_2,l_3\ge 0\\
l_1,l_2 \,\text{are odd},\, l_3 \text{is even}
}}
C_{l_1,l_2,l_3}x_1^{l_1-1}x_2^{l_2-1}x_3^{l_3},\\
 g&=\sum_{\substack{
l_1+l_2+l_3=d-1\\
l_1,l_2,l_3\ge 0\\
l_1,l_3 \,\text{are odd},\, l_2 \text{is even}
}}
C_{l_1,l_2,l_3}x_1^{l_1-1}x_2^{l_2}x_3^{l_3-1},\\
h& =\sum_{\substack{
l_1+l_2+l_3=d-1\\
l_1,l_2,l_3\ge 0\\
l_2,l_3 \,\text{are odd},\, l_1 \text{is even}
}}
C_{l_1,l_2,l_3}x_1^{l_1}x_2^{l_2-1}x_3^{l_3-1}.
\end{align*}
Then we prove $(1)$.

If $d=2l$ is even, then $d-1=2l-1$ is odd.  Since \begin{align*}
&\quad\quad \sum_{\substack{
l_1+l_2+l_3=d-1\\
l_1,l_2,l_3\ge 0
}}
C_{l_1,l_2,l_3}x_1^{l_1}x_2^{l_2}x_3^{l_3}\\
&=\sum_{\substack{
l_1+l_2+l_3=d-1\\
l_1,l_2,l_3\ge 0\\
l_1,l_2 \,\text{are even},\, l_3\, \text{is odd}
}}
C_{l_1,l_2,l_3}x_1^{l_1}x_2^{l_2}x_3^{l_3}+\sum_{\substack{
l_1+l_2+l_3=d-1\\
l_1,l_2,l_3\ge 0\\
l_1,l_3 \,\text{are even},\, l_2\, \text{is odd}
}}
C_{l_1,l_2,l_3}x_1^{l_1}x_2^{l_2}x_3^{l_3}\\
&+\sum_{\substack{
l_1+l_2+l_3=d-1\\
l_1,l_2,l_3\ge 0\\
l_2,l_3 \,\text{are even},\, l_1\, \text{is odd}
}}
C_{l_1,l_2,l_3}x_1^{l_1}x_2^{l_2}x_3^{l_3}
+\sum_{\substack{
l_1+l_2+l_3=d-1\\
l_1,l_2,l_3\ge 0\\
l_1,l_2,l_3\, \text{are odd}
}}
C_{l_1,l_2,l_3}x_1^{l_1}x_2^{l_2}x_3^{l_3},
\end{align*}
we have
\begin{align*}
\Omega &=\partial_{\mathcal{A}}[\sum_{\substack{
l_1+l_2+l_3=d-1\\
l_1,l_2,l_3\ge 0
}}
C_{l_1,l_2,l_3}x_1^{l_1}x_2^{l_2}x_3^{l_3}]\\
&=\partial_{\mathcal{A}}[x_3\sum_{\substack{
l_1+l_2+l_3=d-1\\
l_1,l_2,l_3\ge 0\\
l_1,l_2 \,\text{are even},\, l_3\, \text{is odd}
}}
C_{l_1,l_2,l_3}x_1^{l_1}x_2^{l_2}x_3^{l_3-1}]\\
&+\partial_{\mathcal{A}}[x_2\sum_{\substack{
l_1+l_2+l_3=d-1\\
l_1,l_2,l_3\ge 0\\
l_1,l_3 \,\text{are even},\, l_2\, \text{is odd}
}}
C_{l_1,l_2,l_3}x_1^{l_1}x_2^{l_2-1}x_3^{l_3}]\\
&+\partial_{\mathcal{A}}[x_1\sum_{\substack{
l_1+l_2+l_3=d-1\\
l_1,l_2,l_3\ge 0\\
l_2,l_3 \,\text{are even},\, l_1\, \text{is odd}
}}
C_{l_1,l_2,l_3}x_1^{l_1-1}x_2^{l_2}x_3^{l_3}]\\
&+\partial_{\mathcal{A}}[x_1x_2x_3\sum_{\substack{
l_1+l_2+l_3=d-1\\
l_1,l_2,l_3\ge 0\\
l_1,l_2,l_3\, \text{are odd}
}}
C_{l_1,l_2,l_3}x_1^{l_1-1}x_2^{l_2-1}x_3^{l_3-1}].
\end{align*}
Let \begin{align*}
f=\sum_{\substack{
l_1+l_2+l_3=d-1\\
l_1,l_2,l_3\ge 0\\
l_2,l_3 \,\text{are even},\, l_1\, \text{is odd}
}}
C_{l_1,l_2,l_3}x_1^{l_1-1}x_2^{l_2}x_3^{l_3},\\
g=\sum_{\substack{
l_1+l_2+l_3=d-1\\
l_1,l_2,l_3\ge 0\\
l_1,l_3 \,\text{are even},\, l_2\, \text{is odd}
}}
C_{l_1,l_2,l_3}x_1^{l_1}x_2^{l_2-1}x_3^{l_3},\\
h=\sum_{\substack{
l_1+l_2+l_3=d-1\\
l_1,l_2,l_3\ge 0\\
l_1,l_2 \,\text{are even},\, l_3\, \text{is odd}
}}
C_{l_1,l_2,l_3}x_1^{l_1}x_2^{l_2}x_3^{l_3-1},\\
u=\sum_{\substack{
l_1+l_2+l_3=d-1\\
l_1,l_2,l_3\ge 0\\
l_1,l_2,l_3\, \text{are odd}
}}
C_{l_1,l_2,l_3}x_1^{l_1-1}x_2^{l_2-1}x_3^{l_3-1}.
\end{align*}
Then we prove $(2)$.
\end{proof}

\begin{lem}\label{boundary}
Let $M=(m_{ij})_{3\times 3}$ be a matrix in $\mathrm{GL}_3(k)$. Then $x_1^2,x_2^2,x_3^2$ are coboundary elements in $\mathcal{A}$.
\end{lem}

\begin{proof}
For $\forall a_1,a_2,a_3\in k$, we have \begin{align*}
&\quad\partial_{\mathcal{A}}(c_1x_1+c_2x_2+c_3x_3)\\
&=a_1(m_{11}x_1^2+m_{12}x_2^2+m_{13}x_3^2)+a_2(m_{21}x_1^2+m_{22}x_2^2+m_{23}x_3^2)\\
                                            &+a_3(m_{31}x_1^2+m_{32}x_2^2+m_{33}x_3^2)\\
                                            &=(a_1m_{11}+a_2m_{21}+a_3m_{31})x_1^2+(a_1m_{12}+a_2m_{22}+a_3m_{32})x_2^2\\
                                            &+(a_1m_{13}+a_2m_{23}+a_3m_{33})x_3^2.
\end{align*}
So $\partial_{\mathcal{A}}(a_1x_1+a_2x_2+a_3x_3)=x_1^2$ if and only if
\begin{align*}
\begin{cases}
a_1m_{11}+a_2m_{21}+a_3m_{31}=1\\
a_1m_{12}+a_2m_{22}+a_3m_{32}=0\\
a_1m_{13}+a_2m_{23}+a_3m_{33}=0
\end{cases}\Leftrightarrow M^T\left(
                                \begin{array}{c}
                                  a_1 \\
                                  a_2 \\
                                  a_3 \\
                                \end{array}
                              \right)=\left(
                                        \begin{array}{c}
                                          1 \\
                                          0 \\
                                          0 \\
                                        \end{array}
                                      \right).
\end{align*}
Since $r(M)=3$, there exists $$
\begin{cases}
                                                                 a_1=\frac{m_{22}m_{33}-m_{23}m_{32}}{|M|} \\
                                                                 a_2=\frac{m_{13}m_{32}-m_{12}m_{33}}{|M|} \\
                                                                 a_3=\frac{m_{12}m_{23}-m_{13}m_{22}}{|M|}
\end{cases}
$$ such that $\partial_{\mathcal{A}}(a_1x_1+a_2x_2+a_3x_3)=x_1^2$.  Similarly, we can show there exist
$$
\begin{cases}
                                                                 b_1=\frac{m_{23}m_{31}-m_{21}m_{33}}{|M|} \\
                                                                 b_2=\frac{m_{11}m_{33}-m_{13}m_{31}}{|M|} \\
                                                                 b_3=\frac{m_{13}m_{21}-m_{11}m_{23}}{|M|}
\end{cases}\quad \text{and}\quad \begin{cases}
                                                                 c_1=\frac{m_{21}m_{32}-m_{22}m_{31}}{|M|} \\
                                                                 c_2=\frac{m_{12}m_{31}-m_{11}m_{32}}{|M|} \\
                                                                 c_3=\frac{m_{11}m_{22}-m_{12}m_{21}}{|M|}
\end{cases}
$$
such that $\partial_{\mathcal{A}}(b_1x_1+b_2x_2+b_3x_3)=x_2^2$ and $\partial_{\mathcal{A}}(c_1x_1+c_2x_2+c_3x_3)=x_3^2$, respectively.

\end{proof}

\begin{lem}\label{ithree}
Let $M=(m_{ij})_{3\times 3}$ be a matrix in $\mathrm{GL}_3(k)$ and $m_{22}m_{33}-m_{23}m_{32}\neq 0$. If
$g(\bar{x_2},\bar{x_3})\in Z^{2l+1}[\mathcal{A}/(x_1^2)]$ and  $h(\bar{x_2},\bar{x_3})\in Z^{2l}[\mathcal{A}/(x_1^2)]$ are sum of monomials in variables $\bar{x_2}$ and $\bar{x_3}$ with $l\ge 1$.
Then $$h(\bar{x_2},\bar{x_3})=\sum\limits_{i=0}^{l}r_{2i}\bar{x_2}^{2l-2i}\bar{x_3}^{2i}\quad \text{with}\quad
r_{2i}\in k, 0\le i\le l.$$ Furthermore, there exist $u(x_2,x_3)$ and $v(x_2,x_3)$, which are sums of monomials  in variables $x_2$ and $x_3$, such that
\begin{align*}
\begin{cases}
g(\bar{x_2},\bar{x_3})=\overline{\partial_{\mathcal{A}}[u(x_2,x_3)]},\\
h(\bar{x_2},\bar{x_3})=\overline{\partial_{\mathcal{A}}[v(x_2,x_3)]}.
\end{cases}
\end{align*}
\end{lem}
\begin{proof}
 Let $g(\bar{x_2}.\bar{x_3})=\sum\limits_{j=0}^{2l+1}t_j\bar{x_2}^{2l+1-j}\bar{x_3}^j$ and
$h(\bar{x_2},\bar{x_3})=\sum\limits_{j=0}^{2l}r_j\bar{x_2}^{2l-j}\bar{x_3}^j$, where each $t_j, r_j\in k$. Then
\begin{align*}
&\quad 0=\overline{\partial_{\mathcal{A}}(\sum\limits_{j=0}^{2l+1}t_jx_2^{2l+1-j}x_3^j)}\\
&=\overline{\partial_{\mathcal{A}}(\sum\limits_{i=0}^{l}t_{2i}x_2^{2l-1-2i}x_3^{2i}+\sum\limits_{i=1}^{l+1}t_{2i-1}x_2^{2l-2i}x_3^{2i-1})}\\
&=\sum\limits_{i=0}^{l}[t_{2i}(m_{22}\bar{x_2}^2+m_{23}\bar{x_3}^2)\bar{x_2}^{2l-2i-2}\bar{x_3}^{2i}+t_{2i+1}\bar{x_2}^{2l-2i-2}\bar{x_3}^{2i}(m_{32}\bar{x_2}^2+m_{33}\bar{x_3}^2)] \\
&=\sum\limits_{i=0}^{l}[(t_{2i}m_{22}+t_{2i+1}m_{32})\bar{x_2}^{2l-2i}\bar{x_3}^{2i}+(t_{2i}m_{23}+t_{2i+1}m_{33})\bar{x_2}^{2l-2i-2}\bar{x_3}^{2i+2}]
\\
\end{align*}
and
\begin{align*}
&\quad 0=\overline{\partial_{\mathcal{A}}(\sum\limits_{j=0}^{2l}r_jx_2^{2l-j}x_3^j)}\\
 &=\sum\limits_{i=1}^{l}r_{2i-1}[(m_{22}\bar{x_2}^2+m_{23}\bar{x_3}^2)\bar{x_2}^{2l-2i}\bar{x_3}^{2i-1}-\bar{x_2}^{2l-2i+1}\bar{x_3}^{2i-2}(m_{32}\bar{x_2}^2+m_{33}\bar{x_3}^2)].
\end{align*}
They imply
\begin{align}\label{teqs}
\begin{cases}
t_0m_{22}+t_1m_{32}=0\\
t_2m_{22}+t_3m_{32}+t_0m_{23}+t_1m_{33}=0\\
t_4m_{22}+t_5m_{32}+t_2m_{23}+t_3m_{33}=0\\
........  \\
t_{2l-2}m_{22}+t_{2l-1}m_{32}+t_{2l-4}m_{23}+t_{2l-3}m_{33}=0\\
t_{2l}m_{22}+t_{2l+1}m_{32}+t_{2l-2}m_{23}+t_{2l-1}m_{33}=0\\
t_{2l}m_{23}+t_{2l+1}m_{33}=0
\end{cases}
\end{align}
and
\begin{align}\label{reqs}
\begin{cases}
r_1m_{32}=0\\
r_1m_{22}=0\\
r_1m_{33}+r_3m_{32}=0\\
r_1m_{23}+r_{3}m_{22}=0\\
........  \\
r_{2l-3}m_{33}+r_{2l-1}m_{32}=0\\
r_{2l-3}m_{23}+r_{2l-1}m_{22}=0\\
r_{2l-1}m_{33}=0 \\
r_{2l-1}m_{23}=0.
\end{cases}
\end{align}

Since $m_{22}m_{33}-m_{23}m_{32}\neq 0$,  the rank of the system matrix
\begin{tiny}$$\left(
  \begin{array}{ccccccccccccccc}
    m_{22} & m_{32} & 0 & 0 & 0 & 0 & 0 & \cdots & 0 & 0 & 0 & 0 & 0 & 0 & 0 \\
    m_{23} & m_{33} & m_{22} & m_{32} & 0 & 0 & 0 & \cdots & 0 & 0 & 0 & 0 & 0 & 0 & 0 \\
    0 & 0 & m_{23} & m_{33} & m_{22} & m_{32} & 0 & \cdots & 0 & 0 & 0 & 0 & 0 & 0 & 0 \\
    \cdots & \cdots & \cdots & \cdots & \cdots & \cdots & \cdots & \cdots & \cdots & \cdots & \cdots & \cdots & \cdots & \cdots & \cdots \\
    0& 0 & 0 & 0 & 0 & 0 & 0 & \cdots & 0 & m_{23} & m_{33} & m_{22} & m_{32} & 0 & 0 \\
    0 & 0 & 0 & 0& 0& 0 & 0 & \cdots& 0 & 0 & 0 & m_{23} & m_{33}& m_{22} & m_{32} \\
    0 & 0 & 0 & 0 & 0 & 0 & 0 & \cdots & 0 & 0 & 0 & 0 & 0 & m_{23} & m_{33} \\
  \end{array}
\right)$$ \end{tiny} of (\ref{teqs}) is $l+2$. Hence the space of
the solutions of (\ref{teqs}) is of dimension $l$.
On the other
hand, for any $1\le i\le l$,
$\overline{\partial_{\mathcal{A}}(x_2^{2l-2i+1}x_3^{2i-1})}$ is
$$-m_{32}\bar{x_2}^{2l-2i+3}\bar{x_3}^{2i-2} +m_{22}\bar{x_2}^{2l-2i+2}\bar{x_3}^{2i-1} -m_{33}\bar{x_2}^{2l-2i+1}\bar{x_3}^{2i}+
m_{23}\bar{x_2}^{2l-2i}\bar{x_3}^{2i+1}.$$ Therefore, $\{\left(
                                                   \begin{array}{c}
                                                     -m_{32}\\
                                                     m_{22}\\
                                                     -m_{33}\\
                                                     m_{23}\\
                                                     0\\
                                                     0 \\
                                                     0 \\
                                                     \vdots \\
                                                     0 \\
                                                     0 \\
                                                     0 \\
                                                     0 \\
                                                     0 \\
                                                     0 \\
                                                     0 \\
                                                   \end{array}
                                                 \right),
                                                 \left(
                                                   \begin{array}{c}
                                                   0\\
                                                   0\\
                                                     -m_{32}\\
                                                     m_{22}\\
                                                     -m_{33}\\
                                                     m_{23}\\
                                                     0\\
                                                     \vdots \\
                                                     0 \\
                                                     0 \\
                                                     0 \\
                                                     0 \\
                                                     0 \\
                                                     0 \\
                                                     0 \\
                                                   \end{array}
                                                 \right),\cdots,\left(
                                                   \begin{array}{c}
                                                     0\\
                                                     0\\
                                                     0\\
                                                     0 \\
                                                     0\\
                                                     0 \\
                                                     0 \\
                                                     \vdots \\
                                                     0 \\
                                                     -m_{32} \\
                                                     m_{22} \\
                                                     -m_{33} \\
                                                     m_{23} \\
                                                     0 \\
                                                     0 \\
                                                   \end{array}
                                                 \right) , \left(
                                                   \begin{array}{c}
                                                     0 \\
                                                     0 \\
                                                     0 \\
                                                     0 \\
                                                     0\\
                                                     0 \\
                                                     0 \\
                                                     \vdots \\
                                                     0 \\
                                                     0 \\
                                                     0 \\
                                                     -m_{32}\\
                                                     m_{22}\\
                                                     -m_{33}\\
                                                     m_{23} \\
                                                   \end{array}
                                                 \right)\}$ is a
                                                 $k$-basis of the space of the
                                                 solutions of system
                                                 (\ref{teqs}). So there exists $\{s_{2i-1}\in k| 1\le i\le l\}$  such that $\overline{\partial_{\mathcal{A}}(\sum\limits_{i=1}^ls_{2i-1}x_2^{2l-2i+1}x_3^{2i-1})}=g(\bar{x_2},\bar{x_3})$.
                                                 Take $u(x_2,x_3)=\sum\limits_{i=1}^ls_{2i-1}x_2^{2l-2i+1}x_3^{2i-1}$.
                                                .

Since $\left|
                                                                               \begin{array}{cc}
                                                                                 m_{22} & m_{23} \\
                                                                                 m_{32} & m_{33} \\
                                                                               \end{array}
                                                                          \right|\neq 0$,
we can conclude $r_1=r_3=\cdots =r_{2l-1}=0$ from the system of equations (\ref{reqs}). So $h(\bar{x_2},\bar{x_3})=\sum\limits_{i=0}^{l}r_{2i}\bar{x_2}^{2l-2i}\bar{x_3}^{2i}$. Since
$$\begin{cases}
\overline{\partial_{\mathcal{A}}[\frac{m_{33}}{m_{22}m_{33}-m_{23}m_{32}}x_2-\frac{m_{23}}{m_{22}m_{33}-m_{23}m_{32}}x_3]}=\bar{x_2}^2\\
\overline{\partial_{\mathcal{A}}[\frac{-m_{32}}{m_{22}m_{33}-m_{23}m_{32}}x_2+\frac{m_{22}}{m_{22}m_{33}-m_{23}m_{32}}x_3]}=\bar{x_3}^2,
\end{cases}
$$
we have
\begin{align*}
&\quad h(\bar{x_2},\bar{x_3})=\sum\limits_{i=0}^{l}r_{2i}\bar{x_2}^{2l-2i}\bar{x_3}^{2i}\\
&=\overline{ \partial_{\mathcal{A}}[\sum\limits_{i=0}^{l-1}r_{2i}(\frac{m_{33}x_2}{m_{22}m_{33}-m_{23}m_{32}}-\frac{m_{23}x_3}{m_{22}m_{33}-m_{23}m_{32}})x_2^{2l-2i-2}x_3^{2i}]}\\
&+ \overline{\partial_{\mathcal{A}}[r_{2l}(\frac{-m_{32}x_2}{m_{22}m_{33}-m_{23}m_{32}}+\frac{m_{22}x_3}{m_{22}m_{33}-m_{23}m_{32}})x_3^{2l-2}].                                                                              }
\end{align*}
Take
\begin{align*}
v(x_2,x_3)&=\sum\limits_{i=0}^{l-1}r_{2i}(\frac{m_{33}x_2}{m_{22}m_{33}-m_{23}m_{32}}-\frac{m_{23}x_3}{m_{22}m_{33}-m_{23}m_{32}})x_2^{2l-2i-2}x_3^{2i}\\
      &+r_{2l}(\frac{-m_{32}x_2}{m_{22}m_{33}-m_{23}m_{32}}+\frac{m_{22}x_3}{m_{22}m_{33}-m_{23}m_{32}})x_3^{2l-2}.
\end{align*}
Then we are done.
\end{proof}
\begin{rem}\label{oddcase}
Since $x_2^2$ and $x_3^2$ are cocycle elements in $\mathcal{A}$, one sees that $u(x_2,x_3)$ in Lemma \ref{ithree} can be chosen as
$
u(x_2,x_3)=\sum\limits_{i=1}^{l}s_{2i-1}x_2^{2l-2i+1}x_3^{2i-1}
$ with $s_{2i-1}\in k$, $1\le i\le l$.
\end{rem}

\begin{lem}\label{twohg}
Let $M=(m_{ij})_{3\times 3}$ be a matrix in $\mathrm{GL}_3(k)$ with $m_{22}m_{33}-m_{23}m_{32}\neq 0$ and $m_{33}\neq 0$.
Assume that  $I_1=(x_1^2),I_2=(x_1^2,x_2^2)$ and $I_3=(x_1^2,x_2^2,x_3^2)$ are the three DG ideals generated by the subsets $\{x_1^2\}, \{x_1^2,x_2^2\}$ and $\{x_1^2,x_2^2,x_3^2\}$ of the DG algebra $\mathcal{A}$, respectively. Then
                                                                       $$H^i(I_2/I_1)=\begin{cases}
 k\lceil \bar{x_2}^2\rceil,\,\, \text{if}\,\, i=2\\
 k\lceil \bar{x_1}\bar{x_2}^2+\bar{x_2}^2(\frac{m_{13}m_{32}-m_{12}m_{33}}{m_{22}m_{33}-m_{23}m_{32}}\bar{x_2}+\frac{m_{12}m_{23}-m_{13}m_{22}}{m_{22}m_{33}-m_{23}m_{32}}\bar{x_3})\rceil, \,\,\text{if}\,\,i=3 \\
 0,\,\, \text{if}\,\, i\ge 4
 \end{cases}$$
 and
 $$
H^i(I_3/I_2)=\begin{cases}
 k\lceil \bar{x_3}^2\rceil,\,\, \text{if}\,\, i=2\\
 k\lceil -m_{33}\bar{x_1}\bar{x_3}^2+m_{13}\bar{x_3}^3 \rceil \oplus k\lceil -m_{33}\bar{x_2}\bar{x_3}^2+m_{23}\bar{x_3}^3\rceil,\,\, \text{if}\,\, i=3\\
 k\lceil m_{23}\bar{x_1}\bar{x_3}^3-m_{13}\bar{x_2}\bar{x_3}^3-m_{33}\bar{x_1}\bar{x_2}\bar{x_3}^2\rceil,  \,\,\text{if}\,\, i=4 \\
 0,\,\,\text{if}\,\, i\ge 5.
\end{cases}$$
\end{lem}

\begin{proof}
By Lemma \ref{centcocy}, each $x_i^2$ is a central cocycle element of $\mathcal{A}$.  So $I_1,I_2$ and $I_3$ are indeed DG ideals of $\mathcal{A}$. Then $H^2(I_2/I_1)=k\lceil x_2^2\rceil $ and $H^2(I_3/I_2)=k\lceil x_3^2\rceil$ since $I_2/I_1$ and $I_3/I_2$ are concentrated in degrees $\ge 2$, $(I_2/I_1)^2=kx_2^2$ and $(I_3/I_2)^2=kx_3^2$.

Any graded cocycle element $\Omega$ of degree $d$ in $I_2/I_1$ can be written as $$\Omega=\bar{x_1}\bar{x_2}^2f(\bar{x_2},\bar{x_3})+\bar{x_2}^2g(\bar{x_2},\bar{x_3}),$$ where $f(\bar{x_2},\bar{x_3})$ and $g(\bar{x_2},\bar{x_3})$ are sums of monomials in variables $\bar{x_2}$ and $\bar{x_3}$. We have
\begin{align*}
                 0=&\partial_{I_2/I_1}(z)\\
=&(m_{12}\bar{x_2}^2+m_{13}\bar{x_3}^2)\bar{x_2}^2f(\bar{x_2},\bar{x_3})-\bar{x_1}\bar{x_2}^2\overline{\partial_{\mathcal{A}}[f(x_2,x_3)]}+\bar{x_2}^2\overline{\partial_{\mathcal{A}}[g(x_2,x_3)]}\\
=&\bar{x_2}^2\{(m_{12}\bar{x_2}^2+m_{13}\bar{x_3}^2)f(\bar{x_2},\bar{x_3})+\overline{\partial_{\mathcal{A}}[g(x_2,x_3)]}\}-\bar{x_1}\bar{x_2}^2\overline{\partial_{\mathcal{A}}[f(x_2,x_3)]}.
\end{align*}
Thus \begin{align}\label{eqone}
\begin{cases}
\overline{\partial_{\mathcal{A}}[f(x_2,x_3)]}=0 \\ \overline{\partial_{\mathcal{A}}[g(x_2,x_3)]}=-(m_{12}\bar{x_2}^2+m_{13}\bar{x_3}^2)f(\bar{x_2},\bar{x_3}).
\end{cases}
\end{align}

When $d=3$, we have $|f(\overline{x_2},\overline{x_3})|=0$ and $|g(\overline{x_2},\overline{x_3})|=1$. Let $f(\overline{x_2},\overline{x_3})=c\in k$ and $g(\overline{x_2},\overline{x_3})=c_1\overline{x_2} +c_2\overline{x_3}$. Then
\begin{align*}
-(m_{12}\bar{x_2}^2+m_{13}\bar{x_3}^2)&c=\overline{\partial_{\mathcal{A}}[g(x_2,x_3)]}\\
                                       &=\overline{\partial_{\mathcal{A}}(c_1x_2+c_2x_3)}\\
                                       &=\overline{c_1(m_{21}x_1^2+m_{22}x_2^2+m_{23}x_3^2)+c_2(m_{31}x_1^2+m_{32}x_2^2+m_{33}x_3^2)}\\                                    &=(c_1m_{22}+c_2m_{32})\bar{x_2}^2+(c_1m_{23}+c_2m_{33})\bar{x_3}^2.
\end{align*}
This implies that
\begin{align*}
\begin{cases}
c_1m_{22}+c_2m_{32}=-cm_{12}\\
c_1m_{23}+c_2m_{33}=-cm_{13}.
\end{cases}
\end{align*}
And hence \begin{align*}
\begin{cases}
c_1=\frac{\left|
                       \begin{array}{cc}
                         -m_{12} & m_{32} \\
                         -m_{13} & m_{33} \\
                       \end{array}
                     \right|}{\left|
                                \begin{array}{cc}
                                  m_{22} & m_{32} \\
                                  m_{23} & m_{33} \\
                                \end{array}
                              \right|}c=\frac{c(m_{13}m_{32}-m_{22}m_{33})}{m_{22}m_{33}-m_{23}m_{32}}\\
c_2==\frac{\left|
                       \begin{array}{cc}
                         m_{22} & -m_{12} \\
                         m_{23} & -m_{13} \\
                       \end{array}
                     \right|}{\left|
                                \begin{array}{cc}
                                  m_{22} & m_{32} \\
                                  m_{23} & m_{33} \\
                                \end{array}
                              \right|}c=\frac{c(m_{12}m_{23}-m_{13}m_{22})}{m_{22}m_{33}-m_{23}m_{32}}
\end{cases}
\end{align*}
Then $$\Omega=\bar{x_1}\bar{x_2}^2c+\bar{x_2}^2[\frac{c(m_{13}m_{32}-m_{22}m_{33})}{m_{22}m_{33}-m_{23}m_{32}}\bar{x_2}+\frac{c(m_{12}m_{23}-m_{13}m_{22})}{m_{22}m_{33}-m_{23}m_{32}}\bar{x_3}]$$
and $$H^3(I_2/I_1)=k\lceil \bar{x_1}\bar{x_2}^2+\bar{x_2}^2(\frac{m_{13}m_{32}-m_{22}m_{33}}{m_{22}m_{33}-m_{23}m_{32}}\bar{x_2}+\frac{m_{12}m_{23}-m_{13}m_{22}}{m_{22}m_{33}-m_{23}m_{32}}\bar{x_3})\rceil$$
since $B^3(I_2/I_1)=0$.

When $d=4$, we have $|f(\bar{x_2},\bar{x_3})|=1$ and $|g(\bar{x_2},\bar{x_3})|=2$. Let $f(\bar{x_2},\bar{x_3})=l_1\bar{x_2}+l_2\bar{x_3}$ and $g(\bar{x_2},\bar{x_3})=t_1\bar{x_2}^2+t_2\bar{x_2}\bar{x_3}+t_3\bar{x_3}^2$. Then by (\ref{eqone}), we have
\begin{align*}
0=&\overline{\partial_{\mathcal{A}}[f(x_2,x_3)]}\\
 =&\overline{\partial_{\mathcal{A}}(l_1x_2+l_2x_3)}\\
 =&\overline{l_1(m_{21}x_1^2+m_{22}x_2^2+m_{23}x_3^2)+l_2(m_{31}x_1^2+m_{32}x_2^2+m_{33}x_3^2)}\\
 =&(l_1m_{22}+l_2m_{32})\bar{x_2}^2+(l_1m_{23}+l_2m_{33})\bar{x_3}^2,\\
\end{align*}
which implies that
\begin{align*}
\begin{cases}
l_1m_{22}+l_2m_{32}=0\\
l_1m_{23}+l_2m_{33}=0.
\end{cases}
\end{align*}
Since $m_{22}m_{33}-m_{23}m_{32}\neq 0$, we get $l_1=l_2=0$ and hence $f(\bar{x_2},\bar{x_3})=0$. Then by (\ref{eqone}), we have
\begin{align*}
0=&\overline{\partial_{\mathcal{A}}[g(x_2,x_3)]}\\
 =&\overline{\partial_{\mathcal{A}}[t_1x_2^2+t_2x_2x_3+t_3x_3^2]}\\
 =&\overline{t_2(m_{21}x_1^2+m_{22}x_2^2+m_{23}x_3^2)x_3-t_2x_2(m_{31}x_1^2+m_{32}x_2^2+m_{33}x_3^2)}\\
 =&t_2m_{22}\bar{x_2}^2\bar{x_3}+t_2m_{23}\bar{x_3}^3-t_2m_{32}\bar{x_2}^3-t_2m_{33}\bar{x_2}\bar{x_3}^2.
\end{align*}
Thus $t_2m_{22}=t_2m_{23}=t_2m_{32}=t_2m_{33}=0$. Since $\left|
                                                           \begin{array}{cc}
                                                             m_{22} & m_{23} \\
                                                             m_{32} & m_{33} \\
                                                           \end{array}
                                                         \right|\neq 0$, we get $t_2=0$.
                                                         So                                                $\Omega=\bar{x_1}\bar{x_2}^2f(\bar{x_2},\bar{x_3})+\bar{x_2}^2g(\bar{x_2},\bar{x_3})=\bar{x_2}^2(t_1\bar{x_2}^2+t_3\bar{x_3}^2)$. By the proof of Lemma \ref{boundary}, there exist
                                                         $$
\begin{cases}
                                                                 b_1=\frac{m_{23}m_{31}-m_{21}m_{33}}{|M|} \\
                                                                 b_2=\frac{m_{11}m_{33}-m_{13}m_{31}}{|M|} \\
                                                                 b_3=\frac{m_{13}m_{21}-m_{11}m_{23}}{|M|}
\end{cases}\quad \text{and}\quad \begin{cases}
                                                                 c_1=\frac{m_{21}m_{32}-m_{22}m_{31}}{|M|} \\
                                                                 c_2=\frac{m_{12}m_{31}-m_{11}m_{32}}{|M|} \\
                                                                 c_3=\frac{m_{11}m_{22}-m_{12}m_{21}}{|M|}
\end{cases}
$$
such that $\partial_{\mathcal{A}}(b_1x_1+b_2x_2+b_3x_3)=x_2^2$ and $\partial_{\mathcal{A}}(c_1x_1+c_2x_2+c_3x_3)=x_3^2$, respectively.
Then \begin{align*}
z&=\bar{x_2}^2(t_1\bar{x_2}^2+t_3\bar{x_3}^2)\\
 &=\bar{x_2}^2[t_1\overline{\partial_{\mathcal{A}}(b_1x_1+b_2x_2+b_3x_3)}+t_3\overline{\partial_{\mathcal{A}}(c_1x_1+c_2x_2+c_3x_3)}]\\
 &=\partial_{I_2/I_1}\{\bar{x_2}^2[t_1(b_1\bar{x_1}+b_2\bar{x_2}+b_3\bar{x_3})+t_3(c_1\bar{x_1}+c_2\bar{x_2}+c_3\bar{x_3})]\}.
\end{align*}
Hence $H^4(I_2/I_1)=0$.

When $d=2l+3,l\ge 1$, we have $|f(\bar{x_2},\bar{x_3})|=2l$ and $|g(\bar{x_2},\bar{x_3})|=2l+1$. Since $\overline{\partial_{\mathcal{A}}[f(x_2,x_3)]}=0$ by (\ref{eqone}), we get
$f(\bar{x_2},\bar{x_3})=\sum\limits_{i=0}^{l}r_{2i}\bar{x_2}^{2l-2i}\bar{x_3}^{2i}$ by Lemma \ref{ithree}, where $r_{2i}\in k, 0\le i\le l$.   Then by (\ref{eqone}), we have
\begin{align*}
& \overline{\partial_{\mathcal{A}}[g(x_2,x_3)]}=-(m_{12}\bar{x_2}^2+m_{13}\bar{x_3}^2)f(\bar{x_2},\bar{x_3})\\
                                             &=-(m_{12}\bar{x_2}^2+m_{13}\bar{x_3}^2)(\sum\limits_{i=0}^{l}r_{2i}\bar{x_2}^{2l-2i}\bar{x_3}^{2i})\\
&=\overline{\partial_{\mathcal{A}}[(\frac{-m_{12}m_{33}}{m_{22}m_{33}-m_{23}m_{32}}x_2+\frac{m_{12}m_{23}}{m_{22}m_{33}-m_{23}m_{32}}x_3)(\sum\limits_{i=0}^{l}r_{2i}\bar{x_2}^{2l-2i}\bar{x_3}^{2i})] }\\
&+\overline{\partial_{\mathcal{A}}[(\frac{m_{13}m_{32}}{m_{22}m_{33}-m_{23}m_{32}}x_2-\frac{m_{13}m_{22}}{m_{22}m_{33}-m_{23}m_{32}}x_3)(\sum\limits_{i=0}^{l}r_{2i}\bar{x_2}^{2l-2i}\bar{x_3}^{2i})] }\\
&=\overline{\partial_{\mathcal{A}}\{\sum\limits_{i=0}^{l}r_{2i}[\frac{m_{13}m_{32}-m_{12}m_{33}}{m_{22}m_{33}-m_{23}m_{32}}x_2^{2l-2i+1}x_3^{2i}+\frac{m_{12}m_{23}-m_{13}m_{22}}{m_{22}m_{33}-m_{23}m_{32}}x_2^{2l-2i}x_3^{2i+1}]\}}
\end{align*}
Then, by Lemma \ref{ithree}, we may let
\begin{align*}
&\quad g(\bar{x_2},\bar{x_3})\\
&=\sum\limits_{i=0}^{l}r_{2i}[\frac{m_{13}m_{32}-m_{12}m_{33}}{m_{22}m_{33}-m_{23}m_{32}}\bar{x_2}^{2l-2i+1}\bar{x_3}^{2i}+\frac{m_{12}m_{23}-m_{13}m_{22}}{m_{22}m_{33}-m_{23}m_{32}}\bar{x_2}^{2l-2i}\bar{x_3}^{2i+1}]\\
&+\overline{\partial_{\mathcal{A}}[u(x_2,x_3)]},
\end{align*}
where $u(x_2,x_3)$ is a sum of monomials  in variables $x_2$ and $x_3$. Then
\begin{align*}
\Omega &=\bar{x_1}\bar{x_2}^2f(\bar{x_2},\bar{x_3})+\bar{x_2}^2g(\bar{x_2},\bar{x_3})\\
&=\sum\limits_{i=0}^{l}r_{2i}\bar{x_1}\bar{x_2}^{2l-2i+2}\bar{x_3}^{2i}+\bar{x_2}^2\overline{\partial_{\mathcal{A}}[u(x_2,x_3)]}\\
&+\sum\limits_{i=0}^{l}r_{2i}[\frac{m_{13}m_{32}-m_{12}m_{33}}{m_{22}m_{33}-m_{23}m_{32}}\bar{x_2}^{2l-2i+3}\bar{x_3}^{2i}+\frac{m_{12}m_{23}-m_{13}m_{22}}{m_{22}m_{33}-m_{23}m_{32}}\bar{x_2}^{2l-2i+2}\bar{x_3}^{2i+1}]\\
&=\sum\limits_{i=0}^{l}r_{2i}[\bar{x_1}+\frac{(m_{13}m_{32}-m_{12}m_{33})\bar{x_2}+(m_{12}m_{23}-m_{13}m_{22})\bar{x_3}}{m_{22}m_{33}-m_{23}m_{32}}]\bar{x_2}^{2l-2i+2}\bar{x_3}^{2i}\\
&+\bar{x_2}^2\overline{\partial_{\mathcal{A}}[u(x_2,x_3)]}.\\
\end{align*}
 One sees that $\omega=x_1+\frac{(m_{13}m_{32}-m_{12}m_{33})x_2+(m_{12}m_{23}-m_{13}m_{22})x_3}{m_{22}m_{33}-m_{23}m_{32}}$ is a cocycle element in $\mathcal{A}$. Hence
 \begin{align*}
 z&=\overline{\partial_{\mathcal{A}}[-\sum\limits_{i=0}^{l-1}r_{2i}\omega(b_1x_1+b_2x_2+b_3x_3)x_2^{2l-2i}x_3^{2i}-r_{2l}\omega x_2^2x_3^{2l-2}(c_1x_1+c_2x_2+c_3x_3)]}\\
 &+\bar{x_2}^2\overline{\partial_{\mathcal{A}}[u(x_2,x_3)]}\\
 &=\partial_{I_2/I_1}\{[-\sum\limits_{i=0}^{l-1}r_{2i}\omega(b_1\bar{x_1}+b_2\bar{x_2}+b_3\bar{x_3})\bar{x_2}^{2l-2i-2}\bar{x_3}^{2i}]\bar{x_2}^2\}\\
 &+\partial_{I_2/I_1}\{[-r_{2l}\omega(c_1\bar{x_1}+c_2\bar{x_2}+c_3\bar{x_3})\bar{x_3}^{2l-2}+u(\bar{x_2},\bar{x_3})]\bar{x_2}^2\}.
 \end{align*}
Thus $H^{2l+3}(I_2/I_1)=0$.

When $d=2l+4$, we have $|f(\bar{x_2},\bar{x_3})|=2l+1$ and $|g(\bar{x_2},\bar{x_3})|=2l+2$. Since $\overline{\partial_{\mathcal{A}}[f(x_2,x_3)]}=0$ by (\ref{eqone}), we have $$f(\bar{x_2},\bar{x_3})=\overline{\partial_{\mathcal{A}}[\sum\limits_{i=1}^{l}s_{2i-1}x_2^{2l-2i+1}x_3^{2i-1}]}$$ by Lemma \ref{ithree} and Remark \ref{oddcase}, where $s_{2i-1}\in k$, $1\le i\le l$.  Then by (\ref{eqone}), we have
\begin{align*}
\overline{\partial_{\mathcal{A}}[g(x_2,x_3)]} &=-(m_{12}\bar{x_2}^2+m_{13}\bar{x_3}^2)f(\bar{x_2},\bar{x_3})\\                                           &=-(m_{12}\bar{x_2}^2+m_{13}\bar{x_3}^2)\overline{\partial_{\mathcal{A}}[\sum\limits_{i=1}^{l}s_{2i-1}x_2^{2l-2i+1}x_3^{2i-1}]}.
\end{align*}
Then, by Lemma \ref{ithree}, we may let
\begin{align*}
 g(\bar{x_2},\bar{x_3})=-(m_{12}\bar{x_2}^2+m_{13}\bar{x_3}^2)[\sum\limits_{i=1}^{l}s_{2i-1}\bar{x_2}^{2l-2i+1}\bar{x_3}^{2i-1}]+\overline{\partial_{\mathcal{A}}[v(x_2,x_3)]}.
\end{align*}
where $v(x_2,x_3)$ is a sum of monomials  in variables $x_2$ and $x_3$. Then
\begin{align*}
\Omega &=\bar{x_1}\bar{x_2}^2f(\bar{x_2},\bar{x_3})+\bar{x_2}^2g(\bar{x_2},\bar{x_3})\\
 &=\bar{x_1}\bar{x_2}^2\overline{\partial_{\mathcal{A}}[\sum\limits_{i=1}^{l}s_{2i-1}x_2^{2l-2i+1}x_3^{2i-1}]}-(m_{12}\bar{x_2}^2+m_{13}\bar{x_3}^2)[\sum\limits_{i=1}^{l}s_{2i-1}\bar{x_2}^{2l-2i+3}\bar{x_3}^{2i-1}]\\
 &+\bar{x_2}^2\overline{\partial_{\mathcal{A}}[v(x_2,x_3)]}\\
 &=-\overline{\partial_{\mathcal{A}}[x_1\sum\limits_{i=1}^{l}s_{2i-1}x_2^{2l-2i+1}x_3^{2i-1}-v(x_2,x_3)]}\bar{x_2}^2\\
 &=\partial_{I_2/I_1}[(v(\bar{x_2},\bar{x_3})-\bar{x_1}\sum\limits_{i=1}^{l}s_{2i-1}\bar{x_2}^{2l-2i+1}\bar{x_3}^{2i-1})\bar{x_2}^2]
\end{align*}
and hence $H^{2l+4}(I_2/I_1)=0$.

Since $(I_3/I_2)^3=k\bar{x_1}\bar{x_3}^2\oplus k\bar{x_2}\bar{x_3}^2\oplus k\bar{x_3}^3$, any cocycle element in $(I_3/I_2)^3$ can be denoted by $c_1\bar{x_1}\bar{x_3}^2+c_2\bar{x_2}\bar{x_3}^2+c_3\bar{x_3}^3$ where $c_1,c_2,c_3\in k$. Then
\begin{align*}
0=&\partial_{I_3/I_2}[c_1\bar{x_1}\bar{x_3}^2+c_2\bar{x_2}\bar{x_3}^2+c_3\bar{x_3}^3]\\
=&c_1m_{13}\bar{x_3}^4+c_2m_{23}\bar{x_3}^4+c_3m_{33}\bar{x_3}^4\\
=&(c_1m_{13}+c_2m_{23}+c_3m_{33})\bar{x_3}^4
\end{align*}
and hence $c_1m_{13}+c_2m_{23}+c_3m_{33}=0$, which has a basic solution system
$$\left(
    \begin{array}{c}
      -m_{33} \\
      0 \\
      m_{13} \\
    \end{array}
  \right),\left(
            \begin{array}{c}
              0 \\
              -m_{33} \\
              m_{23} \\
            \end{array}
          \right)
  $$ So $Z^3(I_3/I_2)=k(-m_{33}\bar{x_1}\bar{x_3}^2+m_{13}\bar{x_3}^3)\oplus k(-m_{33}\bar{x_2}\bar{x_3}^2+m_{23}\bar{x_3}^3)$. Then $$H^3(I_3/I_2)=k\lceil -m_{33}\bar{x_1}\bar{x_3}^2+m_{13}\bar{x_3}^3\rceil\oplus k\lceil -m_{33}\bar{x_2}\bar{x_3}^2+m_{23}\bar{x_3}^3 \rceil$$ since one sees easily that $B^3(I_3/I_2)=0$. Any graded cocycle element $z$ of degree $d, d\ge 4$ in $I_3/I_2$ can be written as $$\chi=\bar{x_1}\bar{x_3}^2\phi(\bar{x_3})+\bar{x_2}\bar{x_3}^2\varphi(\bar{x_3})+\bar{x_1}\bar{x_2}\bar{x_3}^2\psi(\bar{x_3})+\bar{x_3}^2\lambda(\bar{x_3}).$$
We have
\begin{align*}
0=&\partial_{I_3/I_2}(\chi)=\partial_{I_3/I_2}[\bar{x_1}\bar{x_3}^2\phi(\bar{x_3})+\bar{x_2}\bar{x_3}^2\varphi(\bar{x_3})+\bar{x_1}\bar{x_2}\bar{x_3}^2\psi(\bar{x_3})+\bar{x_3}^2\lambda(\bar{x_3})]\\
 =&m_{13}\bar{x_3}^4\phi(\bar{x_3})-\bar{x_1}\bar{x_3}^2\overline{\partial_{\mathcal{A}}[\phi(x_3)]}+m_{23}\bar{x_3}^4\varphi(\bar{x_3})-\bar{x_2}\bar{x_3}^2\overline{\partial_{\mathcal{A}}[\varphi(x_3)]}\\
 &+m_{13}\bar{x_2}\bar{x_3}^4\psi(\bar{x_3})-m_{23}\bar{x_1}\bar{x_3}^4\psi(\bar{x_3})+\bar{x_1}\bar{x_2}\bar{x_3}^2\overline{\partial_{\mathcal{A}}[\psi(\bar{x_3})]}+\bar{x_3}^2\overline{\partial_{\mathcal{A}}[\lambda(x_3)]}\\
 =&\bar{x_3}^2[m_{13}\bar{x_3}^2\phi(\bar{x_3})+m_{23}\bar{x_3}^2\varphi(\bar{x_3})+\overline{\partial_{\mathcal{A}}[\lambda(x_3)]}]+\bar{x_1}\bar{x_2}\bar{x_3}^2\overline{\partial_{\mathcal{A}}[\psi(\bar{x_3})]}\\
 &-\bar{x_1}[\bar{x_3}^2\overline{\partial_{\mathcal{A}}[\phi(x_3)]}+m_{23}\bar{x_3}^4\psi(\bar{x_3})]+\bar{x_2}[m_{13}\bar{x_3}^4\psi(\bar{x_3})-\bar{x_3}^2\overline{\partial_{\mathcal{A}}[\varphi(x_3)]}].
\end{align*}
Hence
\begin{align}\label{ieqtwo}
\begin{cases}
m_{13}\bar{x_3}^2\phi(\bar{x_3})+m_{23}\bar{x_3}^2\varphi(\bar{x_3})+\overline{\partial_{\mathcal{A}}[\lambda(x_3)]}=0\\
\bar{x_3}^2\overline{\partial_{\mathcal{A}}[\phi(x_3)]}+m_{23}\bar{x_3}^4\psi(\bar{x_3})=0\\
m_{13}\bar{x_3}^4\psi(\bar{x_3})-\bar{x_3}^2\overline{\partial_{\mathcal{A}}[\varphi(x_3)]}=0\\
\overline{\partial_{\mathcal{A}}[\psi(\bar{x_3})]}=0.
\end{cases}
\end{align}
When $d=4$, we have $|\psi(\bar{x_3})|=0$, $|\varphi(\bar{x_3})|=|\phi(\bar{x_3})|=1$ and $|\lambda(\bar{x_3})|=2$. Let $\psi(\bar{x_3})=c\in k$. Then by (\ref{ieqtwo}), we get $\varphi(x_3)=\frac{cm_{13}}{m_{33}}x_3, \phi(x_3)=\frac{-cm_{23}}{m_{33}}x_3$ and $\lambda(\bar{x_3})=c'\bar{x_3}^2$, for some $c'\in k$.  So $$Z^4(I_3/I_2)=k(-m_{23}\bar{x_1}\bar{x_3}^3+m_{13}\bar{x_2}\bar{x_3}^3+m_{33}\bar{x_1}\bar{x_2}\bar{x_2}^2)\oplus k\bar{x_3}^4.$$
Then $H^4(I_3/I_2)=k\lceil m_{23}\bar{x_1}\bar{x_3}^3-m_{13}\bar{x_2}\bar{x_3}^3-m_{33}\bar{x_1}\bar{x_2}\bar{x_2}^2\rceil$
since $B^4(I_3/I_2)=k\bar{x_3}^4$.
When $d=2l-1\ge 5$, we have $|\phi(\bar{x_3})|=2l-4$, $|\varphi(\bar{x_3})|=2l-4$, $|\psi(\bar{x_3})|=2l-5$ and $|\lambda(\bar{x_3})|=2l-3$. Let $\psi(\bar{x_3})=q\bar{x_3}^{2l-5}$ for some $q\in k$. Then
  $0=\overline{\partial_{\mathcal{A}}[\psi(\bar{x_3})]}=qm_{33}\bar{x_3}^{2l-4}$ by (\ref{ieqtwo}). So $q=0$ and $\psi(\bar{x_3})=0$. Then we get
  $\overline{\partial_{\mathcal{A}}[\phi(x_3)]}=\overline{\partial_{\mathcal{A}}[\varphi(x_3)]}=0$ by (\ref{ieqtwo}).
Let $\phi(x_3)=px_3^{2l-4}$ and $\varphi(x_3)=rx_3^{2l-4}$, $p,r\in k$. Then $$\overline{\partial_{\mathcal{A}}[\lambda(x_3)]}=-m_{13}p\bar{x_3}^{2l-2}-m_{23}r\bar{x_3}^{2l-2}.$$
So $\lambda(\bar{x_3})=\frac{-(m_{13}p+m_{23}r)\bar{x_3}^{2l-3}}{m_{33}}$.
Then
\begin{align*}
\chi &=\bar{x_1}\bar{x_3}^2\phi(\bar{x_3})+\bar{x_2}\bar{x_3}^2\varphi(\bar{x_3})+\bar{x_1}\bar{x_2}\bar{x_3}^2\psi(\bar{x_3})+\bar{x_3}^2\lambda(\bar{x_3})\\
&=p\bar{x_1}\bar{x_3}^{2l-2}+r\bar{x_2}\bar{x_3}^{2l-2}-\frac{(m_{13}p+m_{23}r)\bar{x_3}^{2l-1}}{m_{33}}\\
&=\frac{[m_{33}(p\bar{x_1}+r\bar{x_2})-(pm_{13}+rm_{23})\bar{x_3}]}{m_{33}}\bar{x_3}^{2l-2}\\
&=\partial_{I_3/I_2}\{\frac{[-m_{33}(p\bar{x_1}+r\bar{x_2})+(pm_{13}+rm_{23})\bar{x_3}]}{m_{33}^2}\bar{x_3}^{2l-3}\}.
\end{align*}
Thus $H^{2l-1}(I_3/I_2)=0$, for any $l\ge 3$.
When $d=2l\ge 6$, we have $|\phi(\bar{x_3})|=2l-3$, $|\varphi(\bar{x_3})|=2l-3$, $|\psi(\bar{x_3})|=2l-4$ and $|\lambda(\bar{x_3})|=2l-2$.
So $\overline{\partial_{\mathcal{A}}[\lambda(x_3)]}=0$ and $\overline{\partial_{\mathcal{A}}[\psi(x_3)]}=0$. Then (\ref{ieqtwo}) is equivalent to
\begin{align*}
\begin{cases}
m_{13}\bar{x_3}^2\phi(\bar{x_3})+m_{23}\bar{x_3}^2\varphi(\bar{x_3})=0\\
\bar{x_3}^2\overline{\partial_{\mathcal{A}}[\phi(x_3)]}+m_{23}\bar{x_3}^4\psi(\bar{x_3})=0\\
m_{13}\bar{x_3}^4\psi(\bar{x_3})-\bar{x_3}^2\overline{\partial_{\mathcal{A}}[\varphi(x_3)]}=0.
\end{cases}
\end{align*}
Let $\lambda(\bar{x_3})=s\bar{x_3}^{2l-2}$ and $\psi(x_3)=t\bar{x_3}^{2l-4}$. Then by the system of equations above, we get
$\phi(\bar{x_3})=\frac{-m_{23}t}{m_{33}}\bar{x_3}^{2l-3}$ and $\varphi(\bar{x_3})=\frac{m_{13}t}{m_{33}}\bar{x_3}^{2l-3}$.
Then
\begin{align*}
\chi &=\bar{x_1}\bar{x_3}^2\phi(\bar{x_3})+\bar{x_2}\bar{x_3}^2\varphi(\bar{x_3})+\bar{x_1}\bar{x_2}\bar{x_3}^2\psi(\bar{x_3})+\bar{x_3}^2\lambda(\bar{x_3})\\
&=\frac{-m_{23}t}{m_{33}}\bar{x_1}\bar{x_3}^{2l-1}+\frac{m_{13}t}{m_{33}}\bar{x_2}\bar{x_3}^{2l-1}+t\bar{x_1}\bar{x_2}\bar{x_3}^{2l-2}+s\bar{x_2}^{2l}\\
&=[\frac{-m_{23}\bar{x_1}\bar{x_3}+m_{13}\bar{x_2}\bar{x_3}+m_{33}\bar{x_1}\bar{x_2}}{m_{33}}]t\bar{x_3}^{2l-2}+s\bar{x_3}^{2l}\\
&=\partial_{I_3/I_2}\{[\frac{-m_{23}\bar{x_1}\bar{x_3}+m_{13}\bar{x_2}\bar{x_3}+m_{33}\bar{x_1}\bar{x_2}}{m_{33}^2}]t\bar{x_3}^{2l-1}+\frac{s}{m_{33}}\bar{x_3}^{2l-1}\}
\end{align*}
Hence $H^{2l}(I_3/I_2)=0$ for any $l\ge 3$.
\end{proof}

\begin{lem}\label{rank5}
Let $M=(m_{ij})_{3\times 3}$ and  $r(M)=2$. Then $r(X)=5$, where  $$X=\left(
      \begin{array}{ccccccccc}
        m_{11} & m_{21} & m_{31} & 0 & 0 & 0 & 0 & 0 & 0 \\
        m_{12} & m_{22} & m_{32} & m_{11} & m_{21} & m_{31} & 0 & 0 & 0 \\
        m_{13} & m_{23} & m_{33} & 0 & 0 & 0 & m_{11} & m_{21} & m_{31} \\
        0 & 0 & 0 & m_{13} & m_{23} & m_{33} & m_{12} & m_{22} & m_{32} \\
        0 & 0 & 0 & m_{12} & m_{22} & m_{32} & 0 & 0 & 0 \\
        0 & 0 & 0 & 0 & 0 & 0 & m_{13} & m_{23} & m_{33} \\
      \end{array}
    \right).$$
\end{lem}
\begin{proof}
Since $r(M)=2$, there exists $(s_1,s_2,s_3)^T\neq 0$ such that
$M(s_1,s_2,s_3)^T=0$,  which is equivalent to $$s_1\left(
                                  \begin{array}{c}
                                    m_{11}\\
                                    m_{21}\\
                                    m_{31} \\
                                  \end{array}
                                \right)+s_2\left(
                                  \begin{array}{c}
                                    m_{12}\\
                                    m_{22}\\
                                    m_{32} \\
                                  \end{array}
                                \right)+s_3\left(
                                  \begin{array}{c}
                                    m_{13}\\
                                    m_{23}\\
                                    m_{33} \\
                                  \end{array}
                                \right)=0.$$
 Without the loss of generality, let $s_1\neq 0$. Then $ \left(
                                  \begin{array}{c}
                                    m_{12}\\
                                    m_{22}\\
                                    m_{32} \\
                                  \end{array}
                                \right), \left(
                                  \begin{array}{c}
                                    m_{13}\\
                                    m_{23}\\
                                    m_{33} \\
                                  \end{array}
                                \right)$ are linearly independent and
 $$(m_{11},m_{21},m_{31})+\frac{s_2}{s_1}(m_{12},m_{22},m_{32})+\frac{s_3}{s_1}(m_{13},m_{23},m_{33})=0.$$
For $X$, we can perform the following elementary row transformations

\begin{align*}
X&\xrightarrow[r_1+\frac{s_3}{s_1}\times r_3]{r_1+\frac{s_2}{s_1}\times r_2} \left(
      \begin{array}{ccccccccc}
             0 & 0 & 0 & \frac{s_2}{s_1}m_{11} & \frac{s_2}{s_1}m_{21} & \frac{s_2}{s_1}m_{31} & \frac{s_3}{s_1}m_{11} & \frac{s_3}{s_1}m_{21} & \frac{s_3}{s_1}m_{31} \\
        m_{12} & m_{22} & m_{32} & m_{11} & m_{21} & m_{31} & 0 & 0 & 0 \\
        m_{13} & m_{23} & m_{33} & 0 & 0 & 0 & m_{11} & m_{21} & m_{31} \\
        0 & 0 & 0 & m_{13} & m_{23} & m_{33} & m_{12} & m_{22} & m_{32} \\
        0 & 0 & 0 & m_{12} & m_{22} & m_{32} & 0 & 0 & 0 \\
        0 & 0 & 0 & 0 & 0 & 0 & m_{13} & m_{23} & m_{33} \\
      \end{array}
    \right)\\
&\xrightarrow[r_1+\frac{s_2s_3}{s_1^2}\times r_4]{r_1+\frac{s_2^2}{s_1^2}\times r_5} \left(
      \begin{array}{ccccccccc}
             0 & 0 & 0 & 0 & 0 & 0 & \frac{-s_3^2}{s_1^2}m_{13} & \frac{-s_3^2}{s_1^2}m_{23} & \frac{-s_3^2}{s_1^2}m_{33} \\
        m_{12} & m_{22} & m_{32} & m_{11} & m_{21} & m_{31} & 0 & 0 & 0 \\
        m_{13} & m_{23} & m_{33} & 0 & 0 & 0 & m_{11} & m_{21} & m_{31} \\
        0 & 0 & 0 & m_{13} & m_{23} & m_{33} & m_{12} & m_{22} & m_{32} \\
        0 & 0 & 0 & m_{12} & m_{22} & m_{32} & 0 & 0 & 0 \\
        0 & 0 & 0 & 0 & 0 & 0 & m_{13} & m_{23} & m_{33} \\
      \end{array}
    \right)\\
&\xrightarrow[]{r_1+\frac{s_3^2}{s_1^2}\times r_6} \left(
      \begin{array}{ccccccccc}
             0 & 0 & 0 & 0 & 0 & 0 & 0 & 0 & 0 \\
        m_{12} & m_{22} & m_{32} & m_{11} & m_{21} & m_{31} & 0 & 0 & 0 \\
        m_{13} & m_{23} & m_{33} & 0 & 0 & 0 & m_{11} & m_{21} & m_{31} \\
        0 & 0 & 0 & m_{13} & m_{23} & m_{33} & m_{12} & m_{22} & m_{32} \\
        0 & 0 & 0 & m_{12} & m_{22} & m_{32} & 0 & 0 & 0 \\
        0 & 0 & 0 & 0 & 0 & 0 & m_{13} & m_{23} & m_{33} \\
      \end{array}
    \right).
\end{align*}
This indicates $r(X)\le 5$ and $$r(X)=r\left(
      \begin{array}{ccccccccc}
        m_{12} & m_{22} & m_{32} & m_{11} & m_{21} & m_{31} & 0 & 0 & 0 \\
        m_{13} & m_{23} & m_{33} & 0 & 0 & 0 & m_{11} & m_{21} & m_{31} \\
        0 & 0 & 0 & m_{13} & m_{23} & m_{33} & m_{12} & m_{22} & m_{32} \\
        0 & 0 & 0 & m_{12} & m_{22} & m_{32} & 0 & 0 & 0 \\
        0 & 0 & 0 & 0 & 0 & 0 & m_{13} & m_{23} & m_{33} \\
      \end{array}
    \right).$$ Let
\begin{align*}
l_1\left(
  \begin{array}{c}
    m_{12} \\
    m_{22} \\
    m_{32} \\
    m_{11} \\
    m_{21} \\
    m_{31} \\
    0 \\
    0 \\
    0 \\
  \end{array}
\right)+l_2\left(
  \begin{array}{c}
    m_{13} \\
    m_{23} \\
    m_{33} \\
    0 \\
    0 \\
    0 \\
    m_{11} \\
    m_{21} \\
    m_{31} \\
  \end{array}
\right)+l_3\left(
             \begin{array}{c}
               0 \\
               0 \\
               0 \\
               m_{13}\\
               m_{23} \\
               m_{33} \\
               m_{12} \\
               m_{22} \\
               m_{32}\\
             \end{array}
           \right)+l_4\left(
                        \begin{array}{c}
                          0\\
                          0 \\
                          0 \\
                          m_{12} \\
                          m_{22} \\
                          m_{32} \\
                          0\\
                          0 \\
                          0\\
                        \end{array}
                      \right)+l_5\left(
                                   \begin{array}{c}
                                     0 \\
                                     0 \\
                                     0 \\
                                     0 \\
                                     0 \\
                                     0 \\
                                     m_{13} \\
                                     m_{23} \\
                                     m_{33} \\
                                   \end{array}
                                 \right)=0.
\end{align*}
Then
\begin{align*}
\begin{cases}
l_1m_{12}+l_2m_{13}=0\\
l_1m_{22}+l_2m_{23}=0\\
l_1m_{32}+l_2m_{33}=0\\
l_1m_{11}+l_3m_{13}+l_4m_{12}=0\\
l_1m_{21}+l_3m_{23}+l_4m_{22}=0\\
l_1m_{31}+l_3m_{33}+l_4m_{32}=0\\
l_2m_{11}+l_3m_{12}+l_5m_{13}=0\\
l_2m_{21}+l_3m_{22}+l_5m_{23}=0\\
l_2m_{31}+l_3m_{32}+l_5m_{33}=0,
\end{cases}
\end{align*}
which implies $l_1=l_2=l_3=l_4=l_5=0$ since $ \left(
                                  \begin{array}{c}
                                    m_{12}\\
                                    m_{22}\\
                                    m_{32} \\
                                  \end{array}
                                \right), \left(
                                  \begin{array}{c}
                                    m_{13}\\
                                    m_{23}\\
                                    m_{33} \\
                                  \end{array}
                                \right)$ are linearly independent. Thus  $$r(X)=r\left(
      \begin{array}{ccccccccc}
        m_{12} & m_{22} & m_{32} & m_{11} & m_{21} & m_{31} & 0 & 0 & 0 \\
        m_{13} & m_{23} & m_{33} & 0 & 0 & 0 & m_{11} & m_{21} & m_{31} \\
        0 & 0 & 0 & m_{13} & m_{23} & m_{33} & m_{12} & m_{22} & m_{32} \\
        0 & 0 & 0 & m_{12} & m_{22} & m_{32} & 0 & 0 & 0 \\
        0 & 0 & 0 & 0 & 0 & 0 & m_{13} & m_{23} & m_{33} \\
      \end{array}
    \right)=5.$$
    Similarly, we can show $r(X)=5$ when $s_2\neq 0$ or $s_3\neq 0$.
\end{proof}

\begin{lem}\label{prime}
Let $M=(m_{ij})_{3\times 3}$ be a matrix in $M_3(k)$ with $r(M)=2$. If \begin{align*}
r_1 &=m_{11}x_1^2+m_{12}x_2^2+m_{13}x_3^2,\\
r_2 &=m_{21}x_1^2+m_{22}x_2^2+m_{23}x_3^2,\\
r_3 &=m_{31}x_1^2+m_{32}x_2^2+m_{33}x_3^2,
\end{align*}
then the graded ideal $(r_1,r_2,r_3)$ is a prime graded ideal of the polynomial graded algebra $k[x_1^2,x_2^2,x_3^2]$.
\end{lem}
\begin{proof}
Since $r(M)=2$, there exist a non-zero solution vector $(t_1,t_2,t_3)^T$ of the homogeneous linear equations $M^TX=0$. We have
\begin{align*}
t_1r_1+t_2r_2+t_3r_3=(t_1,t_2,t_3)\left(
                                                           \begin{array}{c}
                                                             r_1 \\
                                                             r_2 \\
                                                             r_3 \\
                                                           \end{array}
                                                         \right) =(t_1,t_2,t_3)M\left(
                                                           \begin{array}{c}
                                                             x_1^2 \\
                                                             x_2^2 \\
                                                             x_3^2 \\
                                                           \end{array}
                                                         \right)=0.
\end{align*}
Since $(t_1,t_2,t_3)^T\neq 0$, we may as well let $t_3\neq 0$. Then
$r_3=-\frac{t_1}{t_3}r_1-\frac{t_2}{t_3}r_2$ and hence $(r_1,r_2,r_3)=(r_1,r_2)$. Since $$
\left(
  \begin{array}{c}
    m_{31} \\
    m_{32} \\
    m_{33}\\
  \end{array}
\right)=
-\frac{t_1}{t_3}\left(
                  \begin{array}{c}
                    m_{11} \\
                    m_{12}\\
                    m_{13} \\
                  \end{array}
                \right)
-\frac{t_2}{t_3}\left(
                  \begin{array}{c}
                    m_{21} \\
                    m_{22}\\
                    m_{23}\\
                  \end{array}
                \right),$$ we have
$$r\left(
     \begin{array}{ccc}
       m_{11} & m_{12} & m_{13} \\
       m_{21} & m_{22} & m_{23} \\
     \end{array}
   \right)
=2,$$  this indicates that there at least one non-zero minor among
\begin{align*}
\left|
  \begin{array}{cc}
    m_{11} & m_{12} \\
    m_{21} & m_{22} \\
  \end{array}
\right|, \left|
  \begin{array}{cc}
    m_{11} & m_{13} \\
    m_{21} & m_{23} \\
  \end{array}
\right|, \left|
           \begin{array}{cc}
             m_{12} & m_{13} \\
             m_{22} & m_{23} \\
           \end{array}
         \right|.
\end{align*}
We may as well let $\left|
  \begin{array}{cc}
    m_{11} & m_{12} \\
    m_{21} & m_{22} \\
  \end{array}
\right|\neq 0$. Then one sees that
\begin{align*}
k[x_1^2,x_2^2,x_3^2]/(r_1,r_2)\cong k[x_3^2]
\end{align*}
is a domain. So $(r_1,r_2,r_3)=(r_1,r_2)$ is a graded prime ideal of $k[x_1^2,x_2^2,x_3^2]$.
\end{proof}

\begin{lem}\label{second}
  Assume that $M=(m_{ij})_{3\times 3}\in M_3(k)$ with $r(M)=2$,  $k(s_1,s_2,s_3)^T$ and $k(t_1,t_2,t_3)^T$ are the solution spaces of  homogeneous linear equations $MX=0$ and $M^TX=0$, respectively.  We have the following statements.

(1) If $s_1t_1^2+s_2t_2^2+s_3t_3^2\neq 0$,
then $k[\lceil t_1x_1 +t_2x_2+t_3x_3 \rceil]$ is a subalgebra of $H(\mathcal{A})$;

(2) If $s_1t_1^2+s_2t_2^2+s_3t_3^2=0$, then
$$k[\lceil t_1x_1 +t_2x_2+t_3x_3\rceil,\lceil s_1x_1^2+s_2x_2^2+s_3x_3^2\rceil ]/(\lceil t_1x_1 +t_2x_2+t_3x_3\rceil^2)$$ is a subalgebra of $H(\mathcal{A})$.
\end{lem}

\begin{proof}
Clearly, we have $H^0(\mathcal{A})=k$. Since $r(M^T)=2<3$, there is a non-zero solution vector $(t_1,t_2,t_3)^T$ of the homogeneous linear equations $M^TX=0$. For any $c_1 x_1+c_2x_2+c_3x_3\in Z^1(\mathcal{A})$, we have
\begin{align*}
0&=\partial_{\mathcal{A}}(c_1 x_1+c_2x_2+c_3x_3)\\
&= (c_1,c_2,c_3)\left(
                                                           \begin{array}{c}
                                                             \partial_{\mathcal{A}}(x_1) \\
                                                              \partial_{\mathcal{A}}(x_2) \\
                                                              \partial_{\mathcal{A}}(x_3) \\
                                                           \end{array}
                                                         \right) \\
                         &=(c_1,c_2,c_3)M \left(
                                            \begin{array}{c}
                                              x_1^2 \\
                                              x_2^2 \\
                                              x_3^2 \\
                                            \end{array}
                                          \right),
\end{align*}
which implies that $(c_1,c_2,c_3)M =0$ or equivalently $M^T\left(
                                                             \begin{array}{c}
                                                               c_1 \\
                                                               c_2 \\
                                                               c_3 \\
                                                             \end{array}
                                                           \right)=0$. Thus $H^1(\mathcal{A})=k\lceil t_1x_1 +t_2x_2+t_3x_3 \rceil$.

                                                           For any $l_{11}x_1^2+l_{12}x_1x_2+l_{13}x_1x_3+l_{22}x_2^2+l_{23}x_2x_3+l_{33}x_3^2\in  \mathrm{ker}(\partial_{\mathcal{A}}^2)$, we have
\begin{align*}
0&=\partial_{\mathcal{A}}[l_{11}x_1^2+l_{12}x_1x_2+l_{13}x_1x_3+l_{22}x_2^2+l_{23}x_2x_3+l_{33}x_3^2]\\
 &=l_{12}(m_{11}x_1^2+m_{12}x_2^2+m_{13}x_3^2)x_2-l_{12}x_1(m_{21}x_1^2+m_{22}x_2^2+m_{23}x_3^2)\\
 &\quad +l_{13}(m_{11}x_1^2+m_{12}x_2^2+m_{13}x_3^2)x_3-l_{13}x_1(m_{31}x_1^2+m_{32}x_2^2+m_{33}x_3^2)\\
 &\quad +l_{23}(m_{21}x_1^2+m_{22}x_2^2+m_{23}x_3^2)x_3-l_{23}x_2(m_{31}x_1^2+m_{32}x_2^2+m_{33}x_3^2)\\
 &=-(l_{12}m_{21}+l_{13}m_{31})x_1^3+(l_{12}m_{11}-l_{23}m_{31})x_1^2x_2+(l_{13}m_{11}+l_{23}m_{21})x_1^2x_3\\
 &\quad -(l_{12}m_{22}+l_{13}m_{32})x_1x_2^2 +(l_{13}m_{12}+l_{23}m_{22})x_2^2x_3+(l_{12}m_{12}-l_{23}m_{32})x_2^3\\
 &\quad -(l_{12}m_{23}+l_{13}m_{33})x_1x_3^2 +(l_{12}m_{13}-l_{23}m_{33})x_2x_3^2 +(l_{13}m_{13}+l_{23}m_{23})x_3^3.
\end{align*}
Hence
\begin{align*}
\begin{cases}
l_{12}m_{21}+l_{13}m_{31}=0\\
l_{12}m_{11}-l_{23}m_{31}=0\\
l_{13}m_{11}+l_{23}m_{21}=0\\
l_{12}m_{22}+l_{13}m_{32}=0\\
l_{13}m_{12}+l_{23}m_{22}=0\\
l_{12}m_{12}-l_{23}m_{32}=0\\
l_{12}m_{23}+l_{13}m_{33}=0\\
l_{12}m_{13}-l_{23}m_{33}=0\\
l_{13}m_{13}+l_{23}m_{23}=0
\end{cases} \Leftrightarrow  \begin{cases}
l_{12}m_{21}+l_{13}m_{31}=0\\
l_{12}m_{22}+l_{13}m_{32}=0\\
l_{12}m_{23}+l_{13}m_{33}=0\\
l_{12}m_{11}-l_{23}m_{31}=0\\
l_{12}m_{12}-l_{23}m_{32}=0\\
l_{12}m_{13}-l_{23}m_{33}=0\\
l_{13}m_{11}+l_{23}m_{21}=0\\
l_{13}m_{12}+l_{23}m_{22}=0\\
l_{13}m_{13}+l_{23}m_{23}=0,
\end{cases}
\end{align*}
which is equivalent to $$ \left(
                            \begin{array}{ccc}
                              m_{11} & m_{21} & m_{31} \\
                              m_{12} & m_{22} & m_{32} \\
                              m_{13} & m_{23} & m_{33} \\
                            \end{array}
                          \right)
 \left(
                                  \begin{array}{ccc}
                                    0 & l_{12} & l_{13} \\
                                    l_{12} & 0 & l_{23} \\
                                    l_{13} & -l_{23} & 0 \\
                                  \end{array}
                                \right)=0_{3\times 3}.$$
                                We claim that $l_{12}=l_{23}=l_{13}=0$. Indeed, if any one of $l_{12}, l_{23}, l_{13}$ is nonzero, then
there are at least two non-zero linear independent vectors among
$$ \left(
     \begin{array}{c}
       0 \\
       l_{12} \\
       l_{13} \\
     \end{array}
   \right),\left(
             \begin{array}{c}
               l_{12} \\
               0 \\
                -l_{23} \\
             \end{array}
           \right),\left(
                     \begin{array}{c}
                       l_{13} \\
                        l_{23}\\
                       0\\
                     \end{array}
                   \right),
 $$
 which are all solutions of $MX=0$. This contradicts with $r(M)=2$. Hence $\mathrm{ker}(\partial_{\mathcal{A}}^2)=kx_1^2\oplus kx_2^2\oplus kx_3^2$.
                                                           In $\mathcal{A}$, we have
                                                    $$(t_1x_1 +t_2x_2+t_3x_3)^2 =t_1^2x_1^2 + t_2^2x_2^2+t_3^2x_3^2.$$
 (1)If $s_1t_1^2+s_2t_2^2+s_3t_3^2\neq 0$, we claim that $t_1^2x_1^2 + t_2^2x_2^2+t_3^2x_3^2\not\in B^2(\mathcal{A})$.
 Indeed, if there exist $q_1x_1+q_2x_2+q_3x_3\in \mathcal{A}^1$ such that $\partial_{\mathcal{A}}(q_1x_1+q_2x_2+q_3x_3)=t_1^2x_1^2 + t_2^2x_2^2+t_3^2x_3^2$, then
\begin{align*}
   (q_1,q_2,q_3)M\left(
                   \begin{array}{c}
                     x_1^2\\
                     x_2^2 \\
                     x_3^2 \\
                   \end{array}
                 \right)
          &= \partial_{\mathcal{A}}(q_1x_1+q_2x_2+q_3x_3) \\
          &= t_1^2x_1^2 + t_2^2x_2^2+t_3^2x_3^2 \\
          &=(t_1^2,t_2^2,t_3^2)\left(
                                 \begin{array}{c}
                                   x_1^2 \\
                                   x_2^2 \\
                                   x_3^2 \\
                                 \end{array}
                               \right),
\end{align*}
which implies that $(q_1,q_2,q_3)M=(t_1^2,t_2^2,t_3^2)$ and hence
$$0=(q_1,q_2,q_3)M \left(
                     \begin{array}{c}
                       s_1 \\
                       s_2 \\
                       s_3 \\
                     \end{array}
                   \right)=(t_1^2,t_2^2,t_3^2)\left(
                     \begin{array}{c}
                       s_1 \\
                       s_2 \\
                       s_3 \\
                     \end{array}
                   \right) =s_1t_1^2+s_2t_2^2+s_3t_3^2. $$
                   This contradicts with  the assumption that $s_1t_1^2+s_2t_2^2+s_3t_3^2\neq 0$. Then we get that
                    $t_1^2x_1^2 + t_2^2x_2^2+t_3^2x_3^2\not\in B^2(\mathcal{A})$ if $s_1t_1^2+s_2t_2^2+s_3t_3^2\neq 0$.  On the  other hand, we have $\dim_k B^2(\mathcal{A})=2$ since $r(M)=2$. Therefore $\dim_k H^2(\mathcal{A})=1$ and $$H^2(\mathcal{A})=k\lceil t_1^2x_1^2 + t_2^2x_2^2+t_3^2x_3^2 \rceil =k\lceil t_1x_1 + t_2x_2+t_3x_3 \rceil^2.$$ In order to show $k[\lceil t_1x_1+t_2x_2+t_3x_3\rceil]$ is a subalgebra of $H(\mathcal{A})$, we need to show
                     $(t_1x_1+t_2x_2+t_3x_3)^n\not\in B^n(\mathcal{A})$ for any $n\ge 3$. If this not the case, we have
                     \begin{align*}
                     (t_1x_1+t_2x_2+t_3x_3)^n=\begin{cases}
                     \partial_{\mathcal{A}}[x_1x_2f+x_1x_3g+x_2x_3h], \,\, \text{if}\,\, n=2j+1 \,\, \text{is odd}     \\
                     \partial_{\mathcal{A}}[x_1f+x_2g+x_3h+x_1x_2x_3u], \,\, \text{if}\,\, n=2j \,\, \text{is even}
                     \end{cases}
                     \end{align*}
where $f,g,h$ and $u$ are all linear combinations of monomials with non-negative even exponents.
 When $n=2j$ is even, we have
 \begin{align*}
\quad (t_1^2x_1^2+&t_2^2x_2^2+t_3^2x_3^2)^j =(t_1x_1+t_2x_2+t_3x_3)^n\\
&=\partial_{\mathcal{A}}[x_1f+x_2g+x_3h+x_1x_2x_3u]\\
                                    &=(m_{11}x_1^2+m_{12}x_2^2+m_{13}x_3^2)f+(m_{21}x_1^2+m_{22}x_2^2+m_{23}x_3^2)g\\
                                    &+(m_{31}x_1^2+m_{32}x_2^2+m_{33}x_3^2)h+(m_{11}x_1^2+m_{12}x_2^2+m_{13}x_3^2)x_2x_3u \\
                                    &-x_1(m_{21}x_1^2+m_{22}x_2^2+m_{23}x_3^2)x_3g+x_1x_2(m_{31}x_1^2+m_{32}x_2^2+m_{33}x_3^2)u.
 \end{align*}
Considering the parity of exponents of the monomials that appear on both sides the equation above implies that
\begin{align*}
(t_1^2x_1^2+t_2^2x_2^2+t_3^2x_3^2)^j&=(m_{11}x_1^2+m_{12}x_2^2+m_{13}x_3^2)f+(m_{21}x_1^2+m_{22}x_2^2+m_{23}x_3^2)g\\
                                    &+(m_{31}x_1^2+m_{32}x_2^2+m_{33}x_3^2)h\\
                                    &=\partial_{\mathcal{A}}(x_1)f+\partial_{\mathcal{A}}(x_2)g+\partial_{\mathcal{A}}(x_3)h
\end{align*}
and
\begin{align*}
\partial_{\mathcal{A}}(x_1x_2x_3u)
&=(m_{11}x_1^2+m_{12}x_2^2+m_{13}x_3^2)x_2x_3u-x_1(m_{21}x_1^2+m_{22}x_2^2+m_{23}x_3^2)x_3g \\
&+x_1x_2(m_{31}x_1^2+m_{32}x_2^2+m_{33}x_3^2)u=0.
\end{align*}
Therefore, $(t_1^2x_1^2+t_2^2x_2^2+t_3^2x_3^2)^j$ is in the graded ideal $(\partial_{\mathcal{A}}(x_1),\partial_{\mathcal{A}}(x_2),\partial_{\mathcal{A}}(x_3))$ of $k[x_1^2,x_2^2,x_3^2]$. By Lemma \ref{prime}, $(\partial_{\mathcal{A}}(x_1),\partial_{\mathcal{A}}(x_2),\partial_{\mathcal{A}}(x_3))$ is a graded prime ideal of $k[x_1^2,x_2^2,x_3^2]$. So $t_1^2x_1^2+t_2^2x_2^2+t_3^2x_3^2\in (\partial_{\mathcal{A}}(x_1),\partial_{\mathcal{A}}(x_2),\partial_{\mathcal{A}}(x_3))$. Hence there exist $a_1,a_2$ and $a_3$ in $k$ such that \begin{align*}
t_1^2x_1^2+t_2^2x_2^2+t_3^2x_3^2&=a_1\partial_{\mathcal{A}}(x_1)+a_2\partial_{\mathcal{A}}(x_2)+a_3\partial_{\mathcal{A}}(x_3)\\
                                &=\partial_{\mathcal{A}}(a_1x_1+a_2x_2+a_3x_3).
\end{align*}
But this contradicts with the fact that $t_1^2x_1^2+t_2^2x_2^2+t_3^2x_3^2\not\in B^2(\mathcal{A})$, which we have proved above.
Thus $(t_1x_1+t_2x_2+t_3x_3)^n\not\in B^n(\mathcal{A})$ when $n$ is even.

When $n=2j+1$ is odd, we have
 \begin{align*}
 &\quad\quad(t_1x_1+t_2x_2+t_3x_3)(t_1^2x_1^2+t_2^2x_2^2+t_3^2x_3^2)^j =(t_1x_1+t_2x_2+t_3x_3)^n\\
&=\partial_{\mathcal{A}}[x_1x_2f+x_1x_3g+x_2x_3h]\\
                                    &=(m_{11}x_1^2+m_{12}x_2^2+m_{13}x_3^2)x_2f-x_1(m_{21}x_1^2+m_{22}x_2^2+m_{23}x_3^2)f\\
                                    &+(m_{11}x_1^2+m_{12}x_2^2+m_{13}x_3^2)x_3g-x_1(m_{31}x_1^2+m_{32}x_2^2+m_{33}x_3^2)g \\
                                    &+(m_{21}x_1^2+m_{22}x_2^2+m_{23}x_3^2)x_3h-x_2(m_{31}x_1^2+m_{32}x_2^2+m_{33}x_3^2)h\\
                       &=-x_1[(m_{21}x_1^2+m_{22}x_2^2+m_{23}x_3^2)f+(m_{31}x_1^2+m_{32}x_2^2+m_{33}x_3^2)g]\\
                                    &+x_2[(m_{11}x_1^2+m_{12}x_2^2+m_{13}x_3^2)f-(m_{31}x_1^2+m_{32}x_2^2+m_{33}x_3^2)h]\\
                                    &+x_3[(m_{21}x_1^2+m_{22}x_2^2+m_{23}x_3^2)h+(m_{11}x_1^2+m_{12}x_2^2+m_{13}x_3^2)g]\\
                                    &=x_1[-\partial_{\mathcal{A}}(x_2)f-\partial_{\mathcal{A}}(x_3)g]+x_2[\partial_{\mathcal{A}}(x_1)f-\partial_{\mathcal{A}}(x_3)h]+x_3[\partial_{\mathcal{A}}(x_2)h+\partial_{\mathcal{A}}(x_1)g].
 \end{align*}
This implies that
\begin{align*}
\begin{cases}
t_1(t_1^2x_1^2+t_2^2x_2^2+t_3^2x_3^2)^j=-\partial_{\mathcal{A}}(x_2)f-\partial_{\mathcal{A}}(x_3)g=\partial_{\mathcal{A}}[-x_2f-x_3g]\\
t_2(t_1^2x_1^2+t_2^2x_2^2+t_3^2x_3^2)^j=\partial_{\mathcal{A}}(x_1)f-\partial_{\mathcal{A}}(x_3)h=\partial_{\mathcal{A}}[x_1f-x_3h]\\
t_3(t_1^2x_1^2+t_2^2x_2^2+t_3^2x_3^2)^j=\partial_{\mathcal{A}}(x_2)h+\partial_{\mathcal{A}}(x_1)g=\partial_{\mathcal{A}}[x_2h+x_1g].
\end{cases}
\end{align*}
Since $(t_1,t_2,t_3)^T\neq 0$, there is at least one none-zero $t_i, i\in \{1,2,3\}$. Then we get $(t_1^2x_1^2+t_2^2x_2^2+t_3^2x_3^2)^j=(t_1x_1+t_2x_2+t_3x_3)^{2j}\in B^{2j}(\mathcal{A})$, which contradicts with proved fact that $(t_1x_1+t_2x_2+t_3x_3)^n\not\in B^n(\mathcal{A})$ when $n$ is even. Therefore, $(t_1x_1+t_2x_2+t_3x_3)^n\not\in B^n(\mathcal{A})$ when $n$ is odd.

Then we reach a conclusion that $k[\lceil t_1x_1+t_2x_2+t_3x_3\rceil]$ is a subalgebra of $H(\mathcal{A})$ when $s_1t_1^2+s_2t_2^2+s_3t_3^2\neq 0$.

 (2)When $s_1t_1^2+s_2t_2^2+s_3t_3^2=0$, we should show $t_1^2x_1^2+t_2^2x_2^2+t_3^2x_3^2\in B^2(\mathcal{A})$ and
  $s_1x_1^2+s_2x_2^2+s_3x_3^2\not\in B^2(\mathcal{A})$ first. In order to prove $t_1^2x_1^2+t_2^2x_2^2+t_3^2x_3^2\in B^2(\mathcal{A})$, we need to show the
existence of an element $q_1x_1+q_2x_2+q_3x_3\in \mathcal{A}^1$ such that
 \begin{align*}
 \partial_{\mathcal{A}}(q_1x_1+q_2x_2+q_3x_3)&=(q_1,q_2,q_3)M\left(
                                                        \begin{array}{c}
                                                          x_1^2 \\
                                                          x_2^2 \\
                                                          x_3^2 \\
                                                        \end{array}
                                                      \right) \\
                                       &= (t_1^2,t_2^2,t_3^2)\left(
                                                        \begin{array}{c}
                                                          x_1^2 \\
                                                          x_2^2 \\
                                                          x_3^2 \\
                                                        \end{array}
                                                      \right),
 \end{align*}
which is equivalent to  $$M^T\left(
                              \begin{array}{c}
                                q_1 \\
                                q_2 \\
                                q_3 \\
                              \end{array}
                            \right)=\left(
                                      \begin{array}{c}
                                        t_1^2 \\
                                        t_2^2 \\
                                        t_3^2 \\
                                      \end{array}
                                    \right).
                            $$
                            Hence it suffices to show that the nonhomogeneous linear equations
                            $$M^TX=\left(
                                      \begin{array}{c}
                                        t_1^2 \\
                                        t_2^2 \\
                                        t_3^2 \\
                                      \end{array}
                                    \right)$$
                                    has solutions. Let $M=(\beta_1,\beta_2,\beta_3)$ and $b=\left(
                                                                                              \begin{array}{c}
                                                                                                t_1^2 \\
                                                                                                t_2^2 \\
                                                                                                t_3^2 \\
                                                                                              \end{array}
                                                                                            \right)$. Since $M\left(
                                                                                                                                          \begin{array}{c}
                                                                                                                                            s_1 \\
                                                                                                                                            s_2 \\
                                                                                                                                            s_3 \\
                                                                                                                                          \end{array}
                                                                                                                                        \right)=0$,
                                                                                                                                         we have $\sum\limits_{i=1}^3s_i\beta_i=0$ and hence $\sum\limits_{i=1}^3s_i\beta_i^T=0$.
Hence \begin{align*}
r(M^T,b)=r\left(
            \begin{array}{cc}
              \beta^T_1 & t_1^2 \\
              \beta_2^T & t_2^2 \\
              \beta_3^T & t_3^2 \\
            \end{array}
          \right)&=r\left(
            \begin{array}{cc}
              \beta^T_1 & t_1^2 \\
              \beta_2^T & t_2^2 \\
           s_1\beta^T_1 + s_2 \beta_2^T+ s_3\beta_3^T &  s_1t_1^2+s_2t_2^2+s_3t_3^2 \\
            \end{array}
          \right)\\
           &=r\left(
            \begin{array}{cc}
              \beta^T_1 & t_1^2 \\
              \beta_2^T & t_2^2 \\
               0 &  0 \\
            \end{array}
          \right)\le 2.
\end{align*}
On the other hand,  we have $r(M^T,b)\ge r(M^T)=2$. So $r(M^T,b)=2=r(M^T)$ and then  the nonhomogeneous linear equations
                            $$M^TX=\left(
                                      \begin{array}{c}
                                        t_1^2 \\
                                        t_2^2 \\
                                        t_3^2 \\
                                      \end{array}
                                    \right)$$
                                    has solutions.

                                     Now, let us prove $s_1x^2+s_2y^2+s_3z^2\not\in \mathrm{im}(\partial_{\mathcal{A}})$, which is equivalent to  the nonhomogeneous linear equations $$M^TX=\left(
                                      \begin{array}{c}
                                        s_1 \\
                                        s_2 \\
                                        s_3 \\
                                      \end{array}
                                    \right)
                            $$ has no solutions. Let $s=\left(
                                      \begin{array}{c}
                                        s_1 \\
                                        s_2 \\
                                        s_3 \\
                                      \end{array}
                                    \right)$.  Then \begin{align*}
r(M^T,s)=r\left(
            \begin{array}{cc}
              \beta^T_1 & s_1 \\
              \beta_2^T & s_2 \\
              \beta_3^T & s_3 \\
            \end{array}
          \right)&=r\left(
            \begin{array}{cc}
              \beta^T_1 & s_1 \\
              \beta_2^T & s_2 \\
           s_1\beta^T_1 + s_2 \beta_2^T+ s_3\beta_3^T &  s_1^2+s_2^2+s_3^2 \\
            \end{array}
          \right)\\
           &=r\left(
            \begin{array}{cc}
              \beta^T_1 & s_1 \\
              \beta_2^T & s_2 \\
               0 &  s_1^2+s_2^2+s_3^2 \\
            \end{array}
          \right)= 3\neq r(M^T)=2.
\end{align*}
Hence $M^TX=s$ has no solutions and $H^2(\mathcal{A})=k\lceil s_1x_1^2+s_2x_2^2+s_3x_3^2 \rceil$. It remains to show that
$$(s_1x_1^2+s_2x_2^2+s_3x_3^2)^{j+1}\not\in B^{2j+2}(\mathcal{A})$$ and $$(t_1x_1+t_2x_2+t_3x_3)(s_1x_1^2+s_2x_2^2+s_3x_3^2)^j\not\in B^{2j+1}(\mathcal{A})$$ for any $j\ge 1$. We will use a proof by contradiction.

If $(s_1x_1^2+s_2x_2^2+s_3x_3^2)^{j+1}\in B^{2j+2}(\mathcal{A})$, then by Lemma \ref{coboundary}, we have  $$(s_1x_1^2+s_2x_2^2+s_3x_3^2)^{j+1}=\partial_{\mathcal{A}}[x_1f+x_2g+x_3h+x_1x_2x_3u],$$ where $f, g$, $h$ and $u$ are all linear combinations of monomials with non-negative even exponents. Considering the parity of exponents of the monomials that appear on both sides of the following equation
\begin{align*}
\quad (s_1x_1^2+&s_2x_2^2+s_3x_3^2)^{j+1} =\partial_{\mathcal{A}}[x_1f+x_2g+x_3h+x_1x_2x_3u]\\
                                    &=(m_{11}x_1^2+m_{12}x_2^2+m_{13}x_3^2)f+(m_{21}x_1^2+m_{22}x_2^2+m_{23}x_3^2)g\\
                                    &+(m_{31}x_1^2+m_{32}x_2^2+m_{33}x_3^2)h+(m_{11}x_1^2+m_{12}x_2^2+m_{13}x_3^2)x_2x_3u \\
                                    &-x_1(m_{21}x_1^2+m_{22}x_2^2+m_{23}x_3^2)x_3g+x_1x_2(m_{31}x_1^2+m_{32}x_2^2+m_{33}x_3^2)u
 \end{align*}
  implies that
\begin{align*}
(s_1x_1^2+s_2x_2^2+s_3x_3^2)^{j+1}&=(m_{11}x_1^2+m_{12}x_2^2+m_{13}x_3^2)f+(m_{21}x_1^2+m_{22}x_2^2+m_{23}x_3^2)g \\
&+(m_{31}x_1^2+m_{32}x_2^2+m_{33}x_3^2)h\\
&=\partial_{\mathcal{A}}(x_1)f+\partial_{\mathcal{A}}(x_2)g+\partial_{\mathcal{A}}(x_3)h
\end{align*}
and
\begin{align*}
\partial_{\mathcal{A}}(x_1x_2x_3u)
&=(m_{11}x_1^2+m_{12}x_2^2+m_{13}x_3^2)x_2x_3u-x_1(m_{21}x_1^2+m_{22}x_2^2+m_{23}x_3^2)x_3g \\
&+x_1x_2(m_{31}x_1^2+m_{32}x_2^2+m_{33}x_3^2)u=0.
\end{align*}
Therefore, $(s_1x_1^2+s_2x_2^2+s_3x_3^2)^{j+1}$ is in the graded ideal $(\partial_{\mathcal{A}}(x_1),\partial_{\mathcal{A}}(x_2),\partial_{\mathcal{A}}(x_3))$ of $k[x_1^2,x_2^2,x_3^2]$. By Lemma \ref{prime}, $(\partial_{\mathcal{A}}(x_1),\partial_{\mathcal{A}}(x_2),\partial_{\mathcal{A}}(x_3))$ is a graded prime ideal of $k[x_1^2,x_2^2,x_3^2]$. So $s_1x_1^2+s_2x_2^2+s_3x_3^2\in (\partial_{\mathcal{A}}(x_1),\partial_{\mathcal{A}}(x_2),\partial_{\mathcal{A}}(x_3))$. Hence there exist $b_1,b_2$ and $b_3$ in $k$ such that \begin{align*}
s_1x_1^2+s_2x_2^2+s_3x_3^2&=b_1\partial_{\mathcal{A}}(x_1)+b_2\partial_{\mathcal{A}}(x_2)+b_3\partial_{\mathcal{A}}(x_3)\\
                                &=\partial_{\mathcal{A}}(b_1x_1+b_2x_2+b_3x_3).
\end{align*}
But this contradicts with the fact that $s_1x_1^2+s_2x_2^2+s_3x_3^2\not\in B^2(\mathcal{A})$, which we have proved above.
Thus $(s_1x_1^2+s_2x_2^2+s_3x_3^2)^{j+1}\not\in B^{2j+2}(\mathcal{A})$, for any $j\ge 1$.

If $(t_1x_1+t_2x_2+t_3x_3)(s_1x_1^2+s_2x_2^2+s_3x_3^2)^j\not\in B^{2j+1}(\mathcal{A})$, then by Lemma \ref{coboundary}, we have $$(t_1x_1+t_2x_2+t_3x_3)(s_1x_1^2+s_2x_2^2+s_3x_3^2)^j=\partial_{\mathcal{A}}[x_1x_2f+x_1x_3g+x_2x_3h],$$ where $f, g$ and $h$ are all linear combinations of monomials with non-negative even exponents.
Then
\begin{align*}
&(t_1x_1+t_2x_2+t_3x_3)(s_1x_1^2+s_2x_2^2+s_3x_3^2)^j=\partial_{\mathcal{A}}[x_1x_2f+x_1x_3g+x_2x_3h]\\
&=(m_{11}x_1^2+m_{12}x_2^2+m_{13}x_3^2)x_2f-x_1(m_{21}x_1^2+m_{22}x_2^2+m_{23}x_3^2)f\\
                                    &+(m_{11}x_1^2+m_{12}x_2^2+m_{13}x_3^2)x_3g-x_1(m_{31}x_1^2+m_{32}x_2^2+m_{33}x_3^2)g \\
                                    &+(m_{21}x_1^2+m_{22}x_2^2+m_{23}x_3^2)x_3h-x_2(m_{31}x_1^2+m_{32}x_2^2+m_{33}x_3^2)h\\
                       &=-x_1[(m_{21}x_1^2+m_{22}x_2^2+m_{23}x_3^2)f+(m_{31}x_1^2+m_{32}x_2^2+m_{33}x_3^2)g]\\
                                    &+x_2[(m_{11}x_1^2+m_{12}x_2^2+m_{13}x_3^2)f-(m_{31}x_1^2+m_{32}x_2^2+m_{33}x_3^2)h]\\
                                    &+x_3[(m_{21}x_1^2+m_{22}x_2^2+m_{23}x_3^2)h+(m_{11}x_1^2+m_{12}x_2^2+m_{13}x_3^2)g]\\
                                    &=x_1[-\partial_{\mathcal{A}}(x_2)f-\partial_{\mathcal{A}}(x_3)g]+x_2[\partial_{\mathcal{A}}(x_1)f-\partial_{\mathcal{A}}(x_3)h]+x_3[\partial_{\mathcal{A}}(x_2)h+\partial_{\mathcal{A}}(x_1)g].
 \end{align*}
This implies
\begin{align*}
\begin{cases}
t_1(s_1x_1^2+s_2x_2^2+s_3x_3^2)^j=-\partial_{\mathcal{A}}(x_2)f-\partial_{\mathcal{A}}(x_3)g=\partial_{\mathcal{A}}(-x_2f-x_3g)\\
t_2(s_1x_1^2+s_2x_2^2+s_3x_3^2)^j=\partial_{\mathcal{A}}(x_1)f-\partial_{\mathcal{A}}(x_3)h=\partial_{\mathcal{A}}(x_1f-x_3h)\\
t_3(s_1x_1^2+s_2x_2^2+s_3x_3^2)^j=\partial_{\mathcal{A}}(x_2)h+\partial_{\mathcal{A}}(x_1)g=\partial_{\mathcal{A}}(x_2h+x_1g).
\end{cases}
\end{align*}
Since $(t_1,t_2,t_3)^T\neq  0$, there is at least one non-zero $t_i, i\in \{1,2,3\}.$ Then we get that $(s_1x_1^2+s_2x_2^2+s_3x_3^2)^j\in B^{2j}(\mathcal{A})$. This contradicts with the proved fact that $(s_1x_1^2+s_2x_2^2+s_3x_3^2)^j\not\in B^{2j}(\mathcal{A})$ for any $j\ge 1$.

Then we can reach a conclusion that $$k[\lceil t_1x_1 +t_2x_2+t_3x_3\rceil,\lceil s_1x_1^2+s_2x_2^2+s_3x_3^2\rceil ]/(\lceil t_1x_1 +t_2x_2+t_3x_3\rceil^2)$$ is a subalgebra of $H(\mathcal{A})$.
\end{proof}

\section{Computations of $H(\mathcal{A})$}\label{cohomology}
Let $\mathcal{A}$ be a $3$-dimensional DG Sklyanin algebra with $\mathcal{A}^{\#}=S_{a,a,0}$ and $\partial_{\mathcal{A}}$ is defined by a matrix
$M\in M_3(k)$.  Note that $\mathcal{A}$ is just the DG algebra $\mathcal{A}_{\mathcal{O}_{-1}(k^3)}(M)$, which is systematically studied in \cite{MWZ}.
In this section, we will compute $H(\mathcal{A})$ case by case. When $r(M)=3$, we have the following proposition.
\begin{prop}\label{rank3}
If $M=(m_{ij})_{3\times 3}\in \mathrm{GL}_3(k)$, then $H(\mathcal{A})=k$.
\end{prop}
\begin{proof}
It suffices to show that $H^i(\mathcal{A})=0$ when $i\neq 0$.
  If $l_1x_1+l_2x_2+l_3x_3 \in Z^1(\mathcal{A})$, then
\begin{align*}
0=\partial_{\mathcal{A}}(l_1x_1+l_2x_2+l_3x_3)=(l_1,l_2,l_3)M\left(
                                                             \begin{array}{c}
                                                               x_1^2 \\
                                                               x_2^2 \\
                                                               x_3^2 \\
                                                             \end{array}
                                                           \right),
\end{align*}
which implies that $(l_1,l_2,l_3)M=0$ and hence $M^T\left(
                                                      \begin{array}{c}
                                                        l_1 \\
                                                        l_2 \\
                                                        l_3 \\
                                                      \end{array}
                                                    \right)=0$. Then each $l_i=0$ since $r(M^T)=3$. So $Z^1(\mathcal{A})=0$ and $H^1(\mathcal{A})=0$.
Since $\partial_{\mathcal{A}}$ is a monomorphism, we have $\dim_k B^2(\mathcal{A})=3$ and $B^2(\mathcal{A})=kx_1^2\oplus kx_2^2\oplus kx_3^2$.
We claim $Z^2(\mathcal{A})=B^2(\mathcal{A})$. It suffices to show $(kx_1x_2\oplus kx_1x_3\oplus kx_2x_3) \bigcap Z^2(\mathcal{A})=0$
since $$\mathcal{A}^2=kx_1^2\oplus kx_2^2\oplus kx_3^2\oplus kx_1x_2\oplus kx_1x_3\oplus kx_2x_3.$$
For any $c_{12}x_1x_2+c_{13}x_1x_3+c_{23}x_2x_3\in Z^2(\mathcal{A})$, we have
\begin{align*}
0&=\partial_{\mathcal{A}}[c_{12}x_1x_2+c_{13}x_1x_3+c_{23}x_2x_3]\\
&=c_{12}(m_{11}x_1^2+m_{12}x_2^2+m_{13}x_3^2)x_2-c_{12}x_1(m_{21}x_1^2+m_{22}x_2^2+m_{23}x_3^2)\\
&+c_{13}(m_{11}x_1^2+m_{12}x_2^2+m_{13}x_3^2)x_3-c_{13}x_1(m_{31}x_1^2+m_{32}x_2^2+m_{33}x_3^2)\\
&+c_{23}(m_{21}x_1^2+m_{22}x_2^2+m_{23}x_3^2)x_3-c_{23}x_2(m_{31}x_1^2+m_{32}x_2^2+m_{33}x_3^2)\\
&=(-c_{12}m_{21}-c_{13}m_{31})x_1^3+(c_{12}m_{12}-c_{23}m_{32})x_2^3+(c_{13}m_{13}+c_{23}m_{23})x_3^3\\
&+(c_{12}m_{11}-c_{23}m_{31})x^2y-(c_{12}m_{22}+c_{13}m_{32} )x_1x_2^2-(c_{12}m_{23}+c_{13}m_{33})x_1x_3^2\\
&+(c_{13}m_{11}+c_{23}m_{21})x_1^2x_3+(c_{13}m_{12}+c_{23}m_{22})x_2^2x_3+(c_{12}m_{13}-c_{23}m_{33})x_2x_3^2.
\end{align*}
Then
\begin{align*}
\begin{cases}
c_{12}m_{21}+c_{13}m_{31}=0\\
c_{12}m_{12}-c_{23}m_{32}=0\\
c_{13}m_{13}+c_{23}m_{23}=0\\
c_{12}m_{11}-c_{23}m_{31}=0\\
c_{12}m_{22}+c_{13}m_{32}=0\\
c_{12}m_{23}+c_{13}m_{33}=0\\
c_{13}m_{11}+c_{23}m_{21}=0\\
c_{13}m_{12}+c_{23}m_{22}=0\\
c_{12}m_{13}-c_{23}m_{33}=0
\end{cases}\Leftrightarrow  \begin{cases}
c_{12}m_{21}+c_{13}m_{31}=0\\
c_{12}m_{22}+c_{13}m_{32}=0\\
c_{12}m_{23}+c_{13}m_{33}=0\\
c_{12}m_{11}-c_{23}m_{31}=0\\
c_{12}m_{12}-c_{23}m_{32}=0\\
c_{12}m_{13}-c_{23}m_{33}=0\\
c_{13}m_{11}+c_{23}m_{21}=0\\
c_{13}m_{12}+c_{23}m_{22}=0\\
c_{13}m_{13}+c_{23}m_{23}=0
\end{cases}\Leftrightarrow \begin{cases}
c_{12}=0\\
c_{13}=0\\
c_{23}=0
\end{cases}
\end{align*}
since $r(M)=3$.  So $(kx_1x_2\oplus kx_1x_3\oplus kx_2x_3) \bigcap Z^2(\mathcal{A})=0$. Thus $H^2(\mathcal{A})=0$.

Since $x_1^2,x_2^2$ and $x_3^2$ are central and cocycle elements in $\mathcal{A}$, they generate a DG ideal $I=(x_1^2,x_2^2,x_3^2)$ of $\mathcal{A}$.
One sees that $\mathcal{A}/I=\bigwedge(x_1,x_2,x_3)$ with $\partial_{\mathcal{A}/I}=0$.
The long exact sequence of cohomologies induced from the shot exact sequence
$$0\to I\stackrel{\iota}{\to} \mathcal{A}\stackrel{\varepsilon}{\to} \mathcal{A}/I\to 0 $$ contains $(\clubsuit)$:
\begin{align*}
& 0\to H^2(\mathcal{A}/I)=k(\lceil x_1\wedge x_2\rceil)\oplus k(\lceil x_1\wedge x_3\rceil)\oplus k(\lceil x_2\wedge x_3\rceil)\stackrel{\delta^2}{\to}
H^3(I)\stackrel{H^3(\iota)}{\to}\\ & H^3(\mathcal{A})\stackrel{H^3(\varepsilon)}{\to} H^3(\mathcal{A}/I)=k(\lceil x_1\wedge x_2\wedge x_3\rceil)
\stackrel{\delta^3}{\to} H^4(I)\stackrel{H^4(\iota)}{\to} H^4(\mathcal{A})\to
H^4(\mathcal{A}/I)=0  \\
&
\to  H^5(I) \stackrel{H^5(\iota)}{\to} H^5(\mathcal{A})\to 0\to \cdots 0\to H^i(I)\stackrel{H^i(\iota)}{\to} H^i(\mathcal{A})\to 0\to \cdots.
\end{align*}
We claim that $H^3(I)=k\lceil \omega_1\rceil \oplus k\lceil \omega_2\rceil\oplus k\lceil \omega_3\rceil$, where \begin{align*}
 \omega_1= -m_{21}x_1^3+m_{11}x_1^2x_2-m_{22}x_1x_2^2+m_{12}x_2^3-m_{23}x_1x_3^2+m_{13}x_2x_3^2                    \\
 \omega_2= -m_{31}x_1^3+m_{11}x_1^2x_3-m_{32}x_1x_2^2+m_{12}x_2^2x_3-m_{33}x_1x_3^2+m_{13}x_3^3                         \\
 \omega_3= -m_{31}x_1^2x_2+m_{21}x_1^2x_3-m_{32}x_2^3+m_{22}x_2^2x_3-m_{33}x_2x_3^2+m_{23}x_3^3.
 \end{align*}
Any cocycle element $\Omega \in Z^3(I)$ can be written as
$$\Omega=(q_1x_1+q_2x_2+q_3x_3)x_1^2+(q_4x_1+q_5x_2+q_6x_3)x_2^2+(q_7x_1+q_8x_2+q_9x_3)x_3^2,$$
where each $q_i\in k$, $1\le i\le 9$. Then
\begin{align*}
0&=\partial_{I}(z)\\
&=(q_1,q_2,q_3)M\left(
                                  \begin{array}{c}
                                    x_1^2 \\
                                    x_2^2 \\
                                    x_3^2\\
                                  \end{array}
                                \right)x_1^2+(q_4,q_5,q_6)M\left(
                                                             \begin{array}{c}
                                                               x_1^2 \\
                                                               x_2^2 \\
                                                               x_3^2 \\
                                                             \end{array}
                                                           \right)x_2^2+(q_7,q_8,q_9)M\left(
                                                                                        \begin{array}{c}
                                                                                          x_1^2 \\
                                                                                          x_2^2\\
                                                                                          x_3^2 \\
                                                                                        \end{array}
                                                                                      \right)x_3^2\\
                                                                                      &=(q_1,q_2,q_3)M\left(
                                  \begin{array}{c}
                                    x_1^4 \\
                                    x_1^2x_2^2 \\
                                    x_1^2x_3^2\\
                                  \end{array}
                                \right)+(q_4,q_5,q_6)M\left(
                                                             \begin{array}{c}
                                                               x_1^2x_2^2 \\
                                                               x_2^4 \\
                                                               x_2^2x_3^2 \\
                                                             \end{array}
                                                           \right)+(q_7,q_8,q_9)M\left(
                                                                                        \begin{array}{c}
                                                                                          x_1^2x_3^2 \\
                                                                                          x_2^2x_3^2\\
                                                                                          x_3^4 \\
                                                                                        \end{array}
                                                                                      \right)
\end{align*}
and hence
$$
\begin{cases}
(q_1,q_2,q_3)M\left(
                                  \begin{array}{c}
                                    1 \\
                                    0 \\
                                    0\\
                                  \end{array}
                                \right)=0 \\
(q_4,q_5,q_6)M\left(
                                                             \begin{array}{c}
                                                               0\\
                                                               1\\
                                                               0\\
                                                             \end{array}
                                                           \right)=0\\
(q_7,q_8,q_9)M\left(
                                                                                        \begin{array}{c}
                                                                                          0\\
                                                                                          0\\
                                                                                          1\\
                                                                                        \end{array}
                                                                                      \right)=0\\
(q_1,q_2,q_3)M\left(
                                  \begin{array}{c}
                                    0\\
                                    1 \\
                                    0\\
                                  \end{array}
                                \right)+(q_4,q_5,q_6)M\left(
                                                             \begin{array}{c}
                                                               1\\
                                                               0\\
                                                               0\\
                                                             \end{array}
                                                           \right)=0\\
(q_1,q_2,q_3)M\left(
                                  \begin{array}{c}
                                    0\\
                                    0 \\
                                    1\\
                                  \end{array}
                                \right)+(q_7,q_8,q_9)M\left(
                                                                                        \begin{array}{c}
                                                                                          1\\
                                                                                          0\\
                                                                                          0\\
                                                                                        \end{array}
                                                                                      \right)=0\\

(q_4,q_5,q_6)M\left(
                                                             \begin{array}{c}
                                                               0\\
                                                               0\\
                                                               1\\
                                                             \end{array}
                                                           \right)+(q_7,q_8,q_9)M\left(
                                                                                        \begin{array}{c}
                                                                                          0\\
                                                                                          1\\
                                                                                          0\\
                                                                                        \end{array}                                                                             \right)=0,
\end{cases}
$$
which is equivalent to
$$
\left(
  \begin{array}{ccccccccc}
    m_{11} & m_{21} & m_{31} & 0 & 0 & 0 & 0 & 0 & 0 \\
    0 & 0 & 0 & m_{12} & m_{22} & m_{32} & 0 & 0 & 0 \\
    0 & 0 & 0 & 0 & 0 & 0 & m_{13} & m_{23} & m_{33} \\
    m_{12} & m_{22} & m_{32} & m_{11} & m_{21} & m_{31} & 0 & 0 & 0 \\
    m_{13} & m_{23} & m_{33} & 0 & 0 & 0 & m_{11} & m_{21} & m_{31} \\
    0 & 0 & 0 & m_{13} & m_{23} & m_{33} & m_{12} & m_{22} & m_{32} \\
  \end{array}
\right)\left(
         \begin{array}{c}
           q_1 \\
           q_2\\
           q_3 \\
           q_4 \\
           q_5 \\
           q_6 \\
           q_7 \\
           q_8 \\
           q_9 \\
         \end{array}
       \right)=0.$$
Since $r(M)=3$, one sees that $$r\left(
  \begin{array}{ccccccccc}
    m_{11} & m_{21} & m_{31} & 0 & 0 & 0 & 0 & 0 & 0 \\
    0 & 0 & 0 & m_{12} & m_{22} & m_{32} & 0 & 0 & 0 \\
    0 & 0 & 0 & 0 & 0 & 0 & m_{13} & m_{23} & m_{33} \\
    m_{12} & m_{22} & m_{32} & m_{11} & m_{21} & m_{31} & 0 & 0 & 0 \\
    m_{13} & m_{23} & m_{33} & 0 & 0 & 0 & m_{11} & m_{21} & m_{31} \\
    0 & 0 & 0 & m_{13} & m_{23} & m_{33} & m_{12} & m_{22} & m_{32} \\
  \end{array}
\right)=6.$$
Hence $\dim_k Z^3(I)=3$. On the other hand,
$$
\begin{cases}
\partial_{\mathcal{A}}(x_1x_2)=\omega_1\\
\partial_{\mathcal{A}}(x_1x_3)=\omega_2\\
\partial_{\mathcal{A}}(x_2x_3)=\omega_3
\end{cases}
$$ implies that $\partial_I(\omega_i)=0,i=1,2,3$. Then $Z^3(I)=k\omega_1\oplus k\omega_2\oplus k\omega_3$ and hence
$H^3(I)=k\lceil \omega_1\rceil \oplus k\lceil \omega_2\rceil \oplus k\lceil \omega_3\rceil $ since $B^3(I)=0$.
The definition of connecting homomorphism implies that
\begin{align*}
\delta^2(\lceil x_1\wedge x_2\rceil)&=\lceil \omega_1 \rceil\\
\delta^2(\lceil x_1\wedge x_3\rceil)&=\lceil \omega_2 \rceil\\
\delta^2(\lceil x_2\wedge x_3\rceil)&=\lceil \omega_3 \rceil.                                                            \end{align*}
Hence $\delta^2$ is a bijection. By the long exact sequence $(\clubsuit)$, one sees that $H^3(\mathcal{A})=0$.

Since $B^2(\mathcal{A})=kx_1^2\oplus kx_2^2\oplus kx_3^2$, one sees that $$B^4(I)=kx_1^4\oplus kx_1^2x_2^2\oplus kx_1^2x_3^2\oplus kx_2^4\oplus kx_2^2x_3^2\oplus kx_3^4.$$ For any $\Omega \in Z^4(I)\bigcap (I^4/B^4(I))$, we can write it as
\begin{align*}
\Omega &=(r_1x_1x_2+r_2x_1x-3+r_3x_2x_3)x_1^2+(r_4x_1x_2+r_5x_1x_3+r_6x_2x_3)x_2^2\\
&+(r_7x_1x_2+r_8x_1x_3+r_9x_2x_3)x_3^2,
\end{align*}
where $r_i\in k, 1\le i\le 9$. Then
\begin{align*}
0=&\partial_I(\Omega)=[r_1(m_{11}x_1^2+m_{12}x_2^2+m_{13}x_3^2)x_2-r_1x_1(m_{21}x_1^2+m_{22}x_2^2+m_{23}x_3^2)]x_1^2 \\
&+[r_2(m_{11}x_1^2+m_{12}x_2^2+m_{13}x_3^2)x_3-r_2x_1(m_{31}x_1^2+m_{32}x_2^2+m_{33}x_3^2)]x_1^2\\
&+[r_3(m_{21}x_1^2+m_{22}x_2^2+m_{23}x_3^2)x_3-r_3x_2(m_{31}x_1^2+m_{32}x_2^2+m_{33}x_3^2)]x_1^2\\
&+[r_4(m_{11}x_1^2+m_{12}x_2^2+m_{13}x_3^2)x_2-r_4x_1(m_{21}x_1^2+m_{22}x_2^2+m_{23}x_3^2)]x_2^2\\
&+[r_5(m_{11}x_1^2+m_{12}x_2^2+m_{13}x_3^2)x_3-r_5x_1(m_{31}x_1^2+m_{32}x_2^2+m_{33}x_3^2)]x_2^2\\
&+[r_6(m_{21}x_1^2+m_{22}x_2^2+m_{23}x_3^2)x_3-r_6x_2(m_{31}x_1^2+m_{32}x_2^2+m_{33}x_3^2)]x_2^2\\
&+[r_7(m_{11}x_1^2+m_{12}x_2^2+m_{13}x_3^2)x_2-r_7x_1(m_{21}x_1^2+m_{22}x_2^2+m_{23}x_3^2)]x_3^2\\
&+[r_8(m_{11}x_1^2+m_{12}x_2^2+m_{13}x_3^2)x_3-r_8x_1(m_{31}x_1^2+m_{32}x_2^2+m_{33}x_3^2)]x_3^2\\
&+[r_9(m_{21}x_1^2+m_{22}x_2^2+m_{23}x_3^2)x_3-r_9x_2(m_{31}x_1^2+m_{32}x_2^2+m_{33}x_3^2)]x_3^2\\
&=-(r_1m_{21}+r_2m_{31})x_1^5+(r_4m_{12}-r_6m_{32})x_2^5+(r_8m_{13}+r_9m_{23})x_3^5\\
&+(r_1m_{11}-r_3m_{31})x_1^4x_2+(r_1m_{12}-r_3m_{32}+r_4m_{11}-r_6m_{31})x_1^2x_2^3\\
&+(r_1m_{13}-r_3m_{33}+r_7m_{11}-r_9m_{31})x_1^2x_2x_3^2 +(r_2m_{11}+r_3m_{21})x_1^4x_3\\
&-(r_1m_{22}+r_2m_{32}+r_4m_{21}+r_5m_{31})x_1^3x_2^2+(r_7m_{13}-r_9m_{33})x_2x_3^4 \\
&-(r_1m_{23}+r_2m_{33}+r_7m_{21}+r_8m_{31})x_1^3x_3^2-(r_4m_{22}+r_5m_{32})x_1x_2^4\\
&+(r_2m_{12}+r_3m_{22}+r_5m_{11}+r_6m_{21})x_1^2x_2^2x_3+(r_5m_{12}+r_6m_{22})x_2^4x_3\\
&+(r_2m_{13}+r_3m_{23}+r_8m_{11}+r_9m_{21})x_1^2x_3^3-(r_7m_{23}+r_8m_{33})x_1x_3^4 \\
&-(r_4m_{23}+r_5m_{33}+r_7m_{22}+r_8m_{32})x_1x_2^2x_3^2\\
&+(r_7m_{12}-r_9m_{32}+r_4m_{13}-r_6m_{33})x_2^3x_3^2\\
&+(r_5m_{13}+r_6m_{23}+r_8m_{12}+r_9m_{22})x_2^2x_3^3
\end{align*}
and hence
\begin{align*}
\begin{cases}
r_1m_{21}+r_2m_{31}=0\\
r_1m_{11}-r_3m_{31}=0\\
r_2m_{11}+r_3m_{21}=0\\
r_4m_{22}+r_5m_{32}=0\\
r_4m_{12}-r_6m_{32}=0\\
r_5m_{12}+r_6m_{22}=0\\
r_7m_{23}+r_8m_{33}=0\\
r_7m_{13}-r_9m_{33}=0\\
r_8m_{13}+r_9m_{23}=0\\
r_1m_{22}+r_2m_{32}+r_4m_{21}+r_5m_{31}=0\\
r_1m_{12}-r_3m_{32}+r_4m_{11}-r_6m_{31}=0\\
r_1m_{23}+r_2m_{33}+r_7m_{21}+r_8m_{31}=0\\
r_1m_{13}-r_3m_{33}+r_7m_{11}-r_9m_{31}=0\\
r_2m_{12}+r_3m_{22}+r_5m_{11}+r_6m_{21}=0\\
r_2m_{13}+r_3m_{23}+r_8m_{11}+r_9m_{21}=0\\
r_4m_{23}+r_5m_{33}+r_7m_{22}+r_8m_{32}=0\\
r_4m_{13}-r_6m_{33}+r_7m_{12}-r_9m_{32}=0\\
r_5m_{13}+r_6m_{23}+r_8m_{12}+r_9m_{22}=0.
\end{cases}
\end{align*}
Since $r(M)=3$, one sees that the rank of the coefficient matrix
$$\left(
  \begin{array}{ccccccccc}
    m_{21} & m_{31} & 0 & 0 & 0 & 0 & 0 & 0 & 0 \\
    m_{11} & 0 & -m_{31} & 0 & 0 & 0 & 0 & 0 & 0 \\
    0 & m_{11} & m_{21} & 0 & 0 & 0 & 0 & 0 & 0 \\
    0 & 0 & 0 & m_{22} & m_{32} & 0 & 0 & 0 & 0 \\
    0 & 0 & 0 & m_{12} & 0 & -m_{32} & 0 & 0 & 0 \\
    0 & 0 & 0 & 0 & m_{12} & m_{22} & 0 & 0 & 0 \\
    0 & 0 & 0 & 0 & 0 & 0 & m_{23} & m_{33} & 0 \\
    0 & 0 & 0 & 0 & 0 & 0 & m_{13} & 0 & -m_{33} \\
    0 & 0 & 0 & 0 & 0 & 0 & 0 & m_{13} & m_{23} \\
    m_{22} & m_{32} & 0 & m_{21} & m_{31} & 0 & 0 & 0 & 0 \\
    m_{12} & 0 & -m_{32} & m_{11} & 0 & -m_{31} & 0 & 0 & 0 \\
    m_{23} & m_{33} & 0 & 0 & 0 & 0 & m_{21} & m_{31} & 0 \\
    m_{13} & 0 & -m_{33} & 0 & 0 & 0 & m_{11} & 0 & -m_{31} \\
    0 & m_{12} & m_{22} & 0 & m_{11}& m_{21} & 0 & 0 & 0 \\
    0 & m_{13} & m_{23} & 0 & 0 & 0 & 0 & m_{11} & m_{21} \\
    0 & 0 & 0 & m_{23} & m_{33} & 0 & m_{22} & m_{32} & 0 \\
    0 & 0 & 0 & m_{13} & 0 & -m_{33} & m_{12} & 0 & -m_{32} \\
    0 & 0 & 0 & 0 & m_{13} & m_{23} & 0 & m_{12} & m_{22}
  \end{array}
\right)$$
is $8$. Therefore, $\dim_k [Z^4(I)\bigcap (I^4/B^4(I))]=1$.  On the other hand,
\begin{align*}
\partial_{\mathcal{A}}(x_1x_2x_3) &=(m_{11}x_1^2+m_{12}x_2^2+m_{13}x_3^2)x_2x_3-(m_{21}x_1^2+m_{22}x_2^2+m_{23}x_3^2)x_1x_3 \\
& +x_1x_2(m_{31}x_1^2+m_{32}x_2^2+m_{33}x_3^2)\\
&=x_1^2(m_{11}x_2x_3-m_{21}x_1x_3+m_{31}x_1x_2)+x_2^2(m_{12}x_2x_3-m_{22}x_1x_3+m_{32}x_1x_2)\\
                       & +z^2(m_{13}x_2x_3-m_{23}x_1x_3+m_{33}x_1x_2)
\end{align*}
We have \begin{align*}
\beta &=x_1^2(m_{11}x_2x_3-m_{21}x_1x_3+m_{31}x_1x_2)+x_2^2(m_{12}x_2x_3-m_{22}x_1x_3+m_{32}x_1x_2)\\
    &+x_3^2(m_{13}x_2x_3-m_{23}x_1x_3+m_{33}x_1x_2) \in Z^4(I)\bigcap (I^4/B^4(I))
    \end{align*}
and hence $H^4(I)=k\lceil \beta\rceil$. By the definition of connecting homomorphism, we have $\delta^3(\lceil x_1\wedge x_2\wedge x_3\rceil)=\lceil \beta\rceil \neq 0$ and hence $\delta^3$ is an isomorphism. By the cohomology long exact sequence $(\clubsuit)$, we get $H^4(\mathcal{A})=0$.
 Since $H^i(A/I)=0$ for any $i\ge 4$, we have $H^{i+1}(I)\cong H^{i+1}(\mathcal{A})$ by the cohomology long exact sequence $(\clubsuit)$.

 Since
\begin{align*}
0\neq |M|=m_{11}\left|
                          \begin{array}{cc}
                            m_{22} & m_{23} \\
                            m_{32} & m_{33} \\
                          \end{array}
                        \right|-m_{12}\left|
                                        \begin{array}{cc}
                                          m_{21} & m_{23} \\
                                          m_{31} & m_{33} \\
                                        \end{array}
                                      \right|+m_{13}\left|
                                                      \begin{array}{cc}
                                                        m_{21} & m_{22} \\
                                                        m_{31} & m_{32} \\
                                                      \end{array}
                                                    \right|,
\end{align*}
 there is at least one non-zero in
 $$\bigg\{\left|
                          \begin{array}{cc}
                            m_{22} & m_{23} \\
                            m_{32} & m_{33} \\
                          \end{array}
                        \right|, \left|
                                        \begin{array}{cc}
                                          m_{21} & m_{23} \\
                                          m_{31} & m_{33} \\
                                        \end{array}
                                      \right|, \left|
                                                      \begin{array}{cc}
                                                        m_{21} & m_{22} \\
                                                        m_{31} & m_{32} \\
                                                      \end{array}
                                                    \right|\bigg \}.$$
 Without the loss of generality, we assume that $\left|
                                                      \begin{array}{cc}
                                                        m_{22} & m_{23} \\
                                                        m_{32} & m_{33} \\
                                                      \end{array}
                                                    \right|\neq 0$ and $m_{33}\neq 0$.
 Let $Q_1=(x_1^2,x_2^2)/(x_1^2)$ and $ Q_2=I/(x_1^2,x_2^2)$. By Lemma \ref{twohg}, we have
 $$H^i(Q_1)=\begin{cases}
 k\lceil \bar{x_2}^2\rceil,\,\, \text{if}\,\, i=2\\
 k\lceil \bar{x_1}\bar{x_2}^2+\bar{x_2}^2(\frac{m_{13}m_{32}-m_{12}m_{33}}{m_{22}m_{33}-m_{23}m_{32}}\bar{x_2}+\frac{m_{12}m_{23}-m_{13}m_{22}}{m_{22}m_{33}-m_{23}m_{32}}\bar{x_3})\rceil, \,\,\text{if}\,\,i=3 \\
 0,\,\, \text{if}\,\, i\ge 4
 \end{cases}$$
and
 $$
H^i(Q_2)=\begin{cases}
 k\lceil \bar{x_1}^2\rceil,\,\, \text{if}\,\, i=2\\
 k\lceil -m_{33}\bar{x_1}\bar{x_3}^2+m_{13}\bar{x_3}^3 \rceil \oplus k\lceil -m_{33}\bar{x_2}\bar{x_3}^2+m_{23}\bar{x_3}^3\rceil,\,\, \text{if}\,\, i=3\\
 k\lceil m_{23}\bar{x_1}\bar{x_3}^3-m_{13}\bar{x_2}\bar{x_3}^3-m_{33}\bar{x_1}\bar{x_2}\bar{x_3}^2\rceil,  \,\,\text{if}\,\, i=4 \\
 0,\,\,\text{if}\,\, i\ge 5.
\end{cases}$$
The cohomology long exact sequence induced from the short exact sequence
$$0\to (x_1^2,x_2^2)\stackrel{\tau}{\to} I\stackrel{\pi}{\to} Q_2\to 0$$
contains \begin{align*}
\cdots \stackrel{H^4(\pi)}{\to} H^4(Q_2)\stackrel{\delta_4}{\to} H^5[(x_1^2,x_2^2)]\stackrel{H^5(\tau)}{\to}H^5(I)\stackrel{H^5(\pi)}{\to} H^5(Q_2)=0\stackrel{\delta^5}{\to} H^6[(x_1^2,x_2^2)]\\
\stackrel{H^6(\tau)}{\to}H^6(I)\stackrel{H^6(\pi)}{\to} H^6(Q_2)=0\to \cdots 0\to H^i[(x_1^2,x_2^2)]\stackrel{H^i(\tau)}{\to}H^i(I)\to 0\to \cdots.
\end{align*}
We have \begin{align*}
& \quad \partial_I(m_{23}x_1x_3^3-m_{13}x_2x_3^3-m_{33}x_1x_2x_3^2) \\
=&(m_{11}m_{23}-m_{13}m_{21})x_1^2x_3^3+(m_{21}m_{33}-m_{23}m_{31})x_1^3x_3^2\\
+&(m_{13}m_{31}-m_{11}m_{33})x_1^2x_2x_3^2
 +(m_{12}m_{23}-m_{13}m_{22})x_2^2x_3^3\\
+&(m_{33}m_{22}-m_{23}m_{32})x_1x_2^2x_3^2+(m_{13}m_{32}-m_{12}m_{33})x_2^3x_3^2\\
=&\Bigg[\left|
     \begin{array}{cc}
       m_{11} & m_{13} \\
       m_{21} & m_{23} \\
     \end{array}
   \right|x_3+\left|
                \begin{array}{cc}
                  m_{21} & m_{23}\\
                  m_{31} & m_{33} \\
                \end{array}
              \right|x_1-\left|
                           \begin{array}{cc}
                             m_{11} & m_{13} \\
                             m_{31} & m_{33} \\
                           \end{array}
                         \right|x_2\Bigg]x_1^2x_3^2\\
 +& \Bigg[\left|
     \begin{array}{cc}
       m_{12} & m_{13} \\
       m_{22} & m_{23} \\
     \end{array}
   \right|x_3+\left|
                \begin{array}{cc}
                  m_{22} & m_{23}\\
                  m_{32} & m_{33} \\
                \end{array}
              \right|x_1-\left|
                           \begin{array}{cc}
                             m_{12} & m_{13} \\
                             m_{32} & m_{33} \\
                           \end{array}
                         \right|x_2\Bigg]x_2^2x_3^2
\end{align*}
and
\begin{align*}
&\partial_{\mathcal{A}}\Bigg[\left|
     \begin{array}{cc}
       m_{11} & m_{13} \\
       m_{21} & m_{23} \\
     \end{array}
   \right|x_3+\left|
                \begin{array}{cc}
                  m_{21} & m_{23}\\
                  m_{31} & m_{33} \\
                \end{array}
              \right|x_1-\left|
                           \begin{array}{cc}
                             m_{11} & m_{13} \\
                             m_{31} & m_{33} \\
                           \end{array}
                         \right|x_2\Bigg]x_1^2\\
+&\partial_{\mathcal{A}}\Bigg[\left|
     \begin{array}{cc}
       m_{12} & m_{13} \\
       m_{22} & m_{23} \\
     \end{array}
   \right|x_3+\left|
                \begin{array}{cc}
                  m_{22} & m_{23}\\
                  m_{32} & m_{33} \\
                \end{array}
              \right|x_1-\left|
                           \begin{array}{cc}
                             m_{12} & m_{13} \\
                             m_{32} & m_{33} \\
                           \end{array}
                         \right|x_2\Bigg]x_2^2\\
=& -|M|x_2^2x_1^2+|M|x_1^2x_2^2=0.
\end{align*}
So \begin{align*}
\chi=&\Bigg[\left|
     \begin{array}{cc}
       m_{11} & m_{13} \\
       m_{21} & m_{23} \\
     \end{array}
   \right|x_3+\left|
                \begin{array}{cc}
                  m_{21} & m_{23}\\
                  m_{31} & m_{33} \\
                \end{array}
              \right|x_1-\left|
                           \begin{array}{cc}
                             m_{11} & m_{13} \\
                             m_{31} & m_{33} \\
                           \end{array}
                         \right|x_2\Bigg]x_1^2 \\
  + &\Bigg[\left|
     \begin{array}{cc}
       m_{12} & m_{13} \\
       m_{22} & m_{23} \\
     \end{array}
   \right|x_3+\left|
                \begin{array}{cc}
                  m_{22} & m_{23}\\
                  m_{32} & m_{33} \\
                \end{array}
              \right|x_1-\left|
                           \begin{array}{cc}
                             m_{12} & m_{13} \\
                             m_{32} & m_{33} \\
                           \end{array}
                         \right|x_2\Bigg]x_2^2\in Z^3(\mathcal{A}).
\end{align*}
Since we have proved $H^3(\mathcal{A})=0$, there exists $\omega \in \mathcal{A}$ such that $\partial_{\mathcal{A}}(\omega)=\chi$.
Then
\begin{align*}
\quad \partial_I(m_{23}x_1x_3^3-m_{13}x_2x_3^3-m_{33}x_1x_2x_3^2)=\chi x_3^2=\partial_{\mathcal{A}}(\omega)x_3^2
\end{align*}
and hence $\delta^4(\lceil m_{23}x_1x_3^3-m_{13}x_2x_3^3-m_{33}x_1x_2x_3^2 \rceil)=\lceil \partial_{\mathcal{A}}(\omega) x_3^2\rceil=0$ by the definition of connecting homomorphism. So $\delta^4=0$.  By the cohomology long exact sequence above, we have $H^i(I)\cong H^i[(x_1^2,x_2^2)], i\ge 5.$
The cohomology long exact sequence induced from the short exact sequence
$$0\to (x_1^2)\stackrel{\tau}{\to} (x_1^2,x_2^2)\stackrel{\phi}{\to} Q_1\to 0$$
contains \begin{align*}
\cdots 0 \stackrel{\delta^4}{\to} H^5((x_1^2)) \stackrel{H^5(\tau)}{\to} H^5((x_1^2,x_2^2))\stackrel{H^5(\phi)}{\to} H^5(Q_1)=0\stackrel{\delta^5}{\to} \\
\cdots 0\stackrel{\delta^{i-1}}{\to} H^i((x_1^2))\stackrel{H^i(\tau)}{\to}H^i((x_1^2,x_2^2))) \stackrel{H^i(\phi)}{\to} H^i(Q_1)=0\stackrel{\delta^{i}}{\to} \cdots.
\end{align*}
Hence $H^i((x_1^2))\cong H^i((x_1^2,x_2^2))$ for any $i\ge 5$. Then we get $$H^i((x_1^2))\cong H^i((x_1^2,x_2^2))\cong H^i(I)\cong H^i(\mathcal{A})$$
 for any $i\ge 5$.
Since $x_1^2$ is a central and cocycle element in $\mathcal{A}$, one sees that $H((x_1^2))=H(\mathcal{A})\lceil x_1^2\rceil $. We have shown that $H^i(\mathcal{A})=0$,  when $i=1,2,3,4$. Then we can inductively prove $H^i(\mathcal{A})=0$ for any $i\ge 1$.
\end{proof}

Now, let us consider the case $r(M)=2$. We have the following proposition.
\begin{prop}\label{rmtwoimc}
For $M\in M_3(k)$ with $r(M)=2$, let $k(s_1,s_2,s_3)^T$ and $k(t_1,t_2,t_3)^T$ be the solution spaces of  homogeneous linear equations $MX=0$ and $M^TX=0$, respectively.  Then $H(\mathcal{A})=k[\lceil t_1x_1 +t_2x_2+t_3x_3\rceil]$ if $s_1t_1^2+s_2t_2^2+s_3t_3^2\neq 0$; and  $H(\mathcal{A})$ equals to
$$k[\lceil t_1x_1 +t_2x_2+t_3x_3\rceil,\lceil s_1x_1^2+s_2x_2^2+s_3x_3^2\rceil ]/(\lceil t_1x_1 +t_2x_2+t_3x_3\rceil^2)$$ when $s_1t_1^2+s_2t_2^2+s_3t_3^2= 0$.
\end{prop}

\begin{proof}
First, we claim $\dim_kH^3(\mathcal{A})=1$. Indeed, For any cocycle element
 $$\xi=l_{1}x_1^3+l_{2}x_1^2x_2+l_{3}x_1^2x_3+l_{4}x_1x_2^2+l_{5}x_2^3+l_{6}x_2^2x_3+l_{7}x_1x_3^2+l_{8}x_2x_3^2+l_{9}x_3^3+l_{10}x_1x_2x_3$$ in $Z^3(\mathcal{A})$,
 we have
\begin{align*}
0&=\partial_{\mathcal{A}}(\xi)=l_{1}x_1^2(m_{11}x_1^2+m_{12}x_2^2+m_{13}x_3^2)+l_2x_1^2(m_{21}x_1^2+m_{22}x_2^2+m_{23}x_3^2)           \\
&+l_3x_1^2(m_{31}x_1^2+m_{32}x_2^2+m_{33}x_3^2)+l_4(m_{11}x_1^2+m_{12}x_2^2+m_{13}x_3^2)x_2^2\\
& +l_5(m_{21}x_1^2+m_{22}x_2^2+m_{23}x_3^2)x_2^2+l_6x_2^2(m_{31}x_1^2+m_{32}x_2^2+m_{33}x_3^2)\\
&+l_7(m_{11}x_1^2+m_{12}x_2^2+m_{13}x_3^2)x_3^2+l_8(m_{21}x_1^2+m_{22}x_2^2+m_{23}x_3^2)x_3^2\\
&+l_9(m_{31}x_1^2+m_{32}x_2^2+m_{33}x_3^2)x_3^2+l_{10}(m_{11}x_1^2+m_{12}x_2^2+m_{13}x_3^2)x_2x_3\\
&-l_{10}x_1(m_{21}x_1^2+m_{22}x_2^2+m_{23}x_3^2)x_3+l_{10}x_1x_2(m_{31}x_1^2+m_{32}x_2^2+m_{33}x_3^2).
\end{align*}
This implies that
\begin{align*}
\begin{cases}
l_1m_{11}+l_2m_{21}+l_3m_{31}=0\\
l_1m_{12}+l_2m_{22}+l_3m_{32}+l_4m_{11}+l_5m_{21}+l_6m_{31}=0\\
l_1m_{13}+l_2m_{23}+l_3m_{33}+l_7m_{11}+l_8m_{21}+l_9m_{31}=0\\
l_4m_{13}+l_5m_{23}+l_6m_{33}+l_7m_{12}+l_8m_{22}+l_9m_{32}=0\\
l_4m_{12}+l_5m_{22}+l_6m_{32}=0\\
l_7m_{13}+l_8m_{23}+l_9m_{33}=0\\
l_{10}=0.
\end{cases}
\end{align*}
Hence $$\left(
      \begin{array}{ccccccccc}
        m_{11} & m_{21} & m_{31} & 0 & 0 & 0 & 0 & 0 & 0 \\
        m_{12} & m_{22} & m_{32} & m_{11} & m_{21} & m_{31} & 0 & 0 & 0 \\
        m_{13} & m_{23} & m_{33} & 0 & 0 & 0 & m_{11} & m_{21} & m_{31} \\
        0 & 0 & 0 & m_{13} & m_{23} & m_{33} & m_{12} & m_{22} & m_{32} \\
        0 & 0 & 0 & m_{12} & m_{22} & m_{32} & 0 & 0 & 0 \\
        0 & 0 & 0 & 0 & 0 & 0 & m_{13} & m_{23} & m_{33} \\
      \end{array}
    \right)\left(
             \begin{array}{c}
               l_1 \\
               l_2 \\
               l_3 \\
               l_4 \\
               l_5\\
               l_6 \\
               l_7\\
               l_8 \\
               l_9 \\
             \end{array}
           \right)=0.
    $$
By Lemma \ref{rank5}, $$r\left(
      \begin{array}{ccccccccc}
        m_{11} & m_{21} & m_{31} & 0 & 0 & 0 & 0 & 0 & 0 \\
        m_{12} & m_{22} & m_{32} & m_{11} & m_{21} & m_{31} & 0 & 0 & 0 \\
        m_{13} & m_{23} & m_{33} & 0 & 0 & 0 & m_{11} & m_{21} & m_{31} \\
        0 & 0 & 0 & m_{13} & m_{23} & m_{33} & m_{12} & m_{22} & m_{32} \\
        0 & 0 & 0 & m_{12} & m_{22} & m_{32} & 0 & 0 & 0 \\
        0 & 0 & 0 & 0 & 0 & 0 & m_{13} & m_{23} & m_{33} \\
      \end{array}
    \right)=5.$$
    So $\dim_k Z^3(\mathcal{A})=9-5=4$. On the other hand,
\begin{align*}
\partial_{\mathcal{A}}(x_1x_2)&=(m_{11}x_1^2+m_{12}x_2^2+m_{13}x_3^2)x_2-x_1(m_{21}x_1^2+m_{22}x_2^2+m_{23}x_3^2) \\
                          &=m_{11}x_1^2x_2+m_{12}x_2^3+m_{13}x_2x_3^2-m_{21}x_1^3-m_{22}x_1x_2^2-m_{23}x_1x_3^2, \\
\partial_{\mathcal{A}}(x_1x_3)&=(m_{11}x_1^2+m_{12}x_2^2+m_{13}x_3^2)x_3-x_1(m_{31}x_1^2+m_{32}x_2^2+m_{33}x_3^2)\\
                          &=m_{11}x_1^2x_3+m_{12}x_2^2x_3+m_{13}x_3^3-m_{31}x_1^3-m_{32}x_1x_2^2-m_{33}x_1x_3^2,  \\
\partial_{\mathcal{A}}(x_2x_3)&=(m_{21}x_1^2+m_{22}x_2^2+m_{23}x_3^2)x_3-x_2(m_{31}x_1^2+m_{32}x_2^2+m_{33}x_3^2) \\
                          &=m_{21}x_1^2x_3+m_{22}x_2^2x_3+m_{23}x_3^3-m_{31}x_1^2x_2-m_{32}x_2^3-m_{33}x_2x_3^2
\end{align*}
are linearly independent, since
\begin{align*}
0&=\lambda_1\partial_{\mathcal{A}}(x_1x_2)+\lambda_2\partial_{\mathcal{A}}(x_1x_3)+\lambda_3\partial_{\mathcal{A}}(x_2x_3)\\
 &=\lambda_1(m_{11}x_1^2x_2+m_{12}x_2^3+m_{13}x_2x_3^2-m_{21}x_1^3-m_{22}x_1x_2^2-m_{23}x_1x_3^2)  \\
 &+\lambda_2(m_{11}x_1^2x_3+m_{12}x_2^2x_3+m_{13}x_3^3-m_{31}x_1^3-m_{32}x_1x_2^2-m_{33}x_1x_3^2)\\
 &+\lambda_3(m_{21}x_1^2x_3+m_{22}x_2^2x_3+m_{23}x_3^3-m_{31}x_1^2x_2-m_{32}x_2^3-m_{33}x_2x_3^2) \\
 &=(\lambda_1m_{11}-\lambda_3m_{31})x_1^2x_2 + (\lambda_1m_{12}-\lambda_3m_{32})x_2^3+(\lambda_1m_{13}-\lambda_3m_{33})x_2x_3^2\\
 &-(\lambda_1m_{21}+\lambda_2m_{31})x_1^3-(\lambda_1m_{22}+\lambda_2m_{32})x_1x_2^2-(\lambda_1m_{23}+\lambda_2m_{33})x_1x_3^2 \\
 &+(\lambda_2m_{11}+\lambda_3m_{21})x_1^2x_3 + (\lambda_2m_{12}+\lambda_3m_{22})x_2^2x_3+(\lambda_2m_{13}+\lambda_3m_{23})x_3^3
\end{align*}
implies
\begin{align*}
\begin{cases}
\lambda_1m_{11}-\lambda_3m_{31}=0\\
\lambda_1m_{12}-\lambda_3m_{32}=0\\
\lambda_1m_{13}-\lambda_3m_{33}=0\\
\lambda_1m_{21}+\lambda_2m_{31}=0\\
\lambda_1m_{22}+\lambda_2m_{32}=0\\
\lambda_1m_{23}+\lambda_2m_{33}=0\\
\lambda_2m_{11}+\lambda_3m_{21}=0\\
\lambda_2m_{12}+\lambda_3m_{22}=0\\
\lambda_2m_{13}+\lambda_3m_{23}=0
\end{cases}\Leftrightarrow
\lambda_1=\lambda_2=\lambda_3=0
\end{align*}
since $r(M)=2$. Then $\dim_k B^3(\mathcal{A})=3$ and we show the claim $\dim_kH^3(\mathcal{A})=1$.

Let $I=(r_1,r_2,r_3)$ be the DG ideal of $\mathcal{A}$ generated by the central coboundary elements $r_1=\partial_{\mathcal{A}}(x_1), r_2=\partial_{\mathcal{A}}(x_2)$ and $r_3=\partial_{\mathcal{A}}(x_3)$. Then the DG quotient ring $Q=\mathcal{A}/I$ has trivial differential.
Since each
$r_i=m_{i1}x_1^2+m_{i2}x_2^2+m_{i3}x_3^2$ and $r(M)=2$, we may assume without the loss of generality that $r_1,r_2$ are linearly independent, which is equivalent to $t_3\neq 0$.  Then $r_3=\frac{t_1}{t_3}r_1+\frac{t_2}{t_3}r_2$ and $I=(r_1,r_2)$. We have
\begin{align*}
H^i(I)=\begin{cases}
k\lceil r_1\rceil\oplus k\lceil r_2\rceil, i=2\\
\lceil r_1 \rceil H^{i-2}(\mathcal{A})\oplus \lceil r_2\rceil H^{i-2}(\mathcal{A})\oplus \lceil r_1x_2-x_1r_2\rceil H^{i-3}(\mathcal{A}),                                        i\ge 3\\
\end{cases}
\end{align*}
and
\begin{align*}
\dim_k H^i(Q)=\dim_k Q^i=\begin{cases}
0, i< 0\\
1,i=0 \\
3,i=1 \\
4, i\ge 2.
\end{cases}
\end{align*}
The short exact sequence
\begin{align*}
0\to I\stackrel{\iota}{\to} \mathcal{A}\stackrel{\pi}{\to} Q\to 0
\end{align*}
induces the cohomology long exact sequence
($\spadesuit$): \begin{align*}
 0 \to H^0(\mathcal{A}) \stackrel{H^0(\pi)}{\to} H^0(Q)\stackrel{\delta^0}{\to} H^1(I)\stackrel{H^1(\iota)}{\to} H^1(\mathcal{A})\stackrel{H^1(\pi)}{\to} H^1(Q)\stackrel{\delta^1}{\to} H^2(I) \\
\stackrel{H^2(\iota)}{\to} H^2(\mathcal{A})\stackrel{H^2(\pi)}{\to}H^2(Q) \stackrel{\delta^2}{\to} \cdots \stackrel{\delta^{i-1}}{\to} H^i(I)\stackrel{H^i(\iota)}{\to} H^i(\mathcal{A})\stackrel{H^i(\pi)}{\to}H^i(Q) \stackrel{\delta^i}{\to}\cdots.
\end{align*}
Since $r_1,r_2$ and $r_1x_2-x_1r_2$ are coboundary elements in $\mathcal{A}$,  we have $H^{i}(\iota)=0$ for any $i\ge 3$. The cohomology long exact sequence ($\spadesuit$) implies that $$\dim_k H^i(\mathcal{A})+\dim_k H^{i+1}(I)= \dim_k H^i(Q), i\ge 3.$$ By Lemma \ref{second} and $\dim_kH^3(\mathcal{A})=1$, we inductively get $\dim_kH^i(\mathcal{A})=1, i\ge 4$. Hence $\dim_k H^i(\mathcal{A})=1$ for any $i\ge 0$.

 By Lemma \ref{second},  the algebra $k[\lceil t_1x_1+t_2x_2+t_3x_3\rceil ]$ is a subalgebra of $H(\mathcal{A})$ when $\sum\limits_{i=1}^3s_it_i^2\neq 0$, and
 $$k[\lceil t_1x_1+t_2x_2+t_3x_3\lceil, \lceil s_1x_1^2+s_2x_2^2+s_3x_3^2\rceil ]/(\lceil t_1x_1 +t_2x_2+t_3x_3\rceil^2)$$ is a subalgebra of $H(\mathcal{A})$ when $\sum\limits_{i=1}^3s_it_i^2=0$. Considering the dimension of each $H^i(\mathcal{A})$ gives that
 $H(\mathcal{A})=k[\lceil t_1x_1+t_2x_2+t_3x_3\rceil ] = H(\mathcal{A})$ when $\sum\limits_{i=1}^3s_it_i^2\neq 0$, and
 $$k[\lceil t_1x_1+t_2x_2+t_3x_3\rceil, \lceil s_1x_1^2+s_2x_2^2+s_3x_3^2\rceil ]/(\lceil t_1x_1 +t_2x_2+t_3x_3\rceil^2)= H(\mathcal{A}),$$ when $\sum\limits_{i=1}^3s_it_i^2=0$.
\end{proof}

It remains to consider the case that $r(M)=1$. In this case, we might as well let $$M=\left(
                                 \begin{array}{ccc}
                                   m_{11} & m_{12} & m_{13} \\
                                   l_1m_{11} & l_1m_{12} & l_1m_{13} \\
                                   l_2m_{11} & l_2m_{12} & l_2m_{13} \\
                                 \end{array}
                               \right),\,\,\text{with}\,\, l_1,l_2\in k \,\, \text{and}\,\, (m_{11}, m_{12},m_{13})\neq 0.$$
Indeed, one can see the reason by \cite[Remark 5.4]{MWZ}.
Note that we have
                               $$\begin{cases}
                               \partial_{\mathcal{A}}(x_1)=m_{11}x_1^2+m_{12}x_2^2+m_{13}x_3^2\\
                               \partial_{\mathcal{A}}(x_2)=l_1[m_{11}x_1^2+m_{12}x_2^2+m_{13}x_3^2]\\
                               \partial_{\mathcal{A}}(x_3)=l_2[m_{11}x_1^2+m_{12}x_2^2+m_{13}x_3^2].
 \end{cases}
$$
For any $c_1x_1+c_2x_2+c_3x_3\in Z^1(\mathcal{A})$, we have
\begin{align*}
0 =\partial_{\mathcal{A}}(c_1x_1+c_2x_2+c_3x_3) &=(c_1+l_1c_2+l_2c_3)[m_{11}x_1^2+m_{12}x_2^2+m_{13}x_3^2] \\
&\Rightarrow c_1+l_1c_2+l_2c_3=0,
\end{align*}
 which admits a basic solution system
$\left(
    \begin{array}{c}
      l_1 \\
      -1 \\
      0 \\
    \end{array}
  \right),\left(
            \begin{array}{c}
              l_2 \\
              0 \\
              -1 \\
            \end{array}
          \right).$
So $$Z^1(\mathcal{A})=k(l_1x_1-x_2)\oplus k(l_2x_1-x_3)$$ and $$H^1(\mathcal{A})=k\lceil l_1x_1-x_2\rceil \oplus k\lceil l_2x_1-x_3\rceil.$$
For any $c_{11}x_1^2+c_{12}x_1x_2+c_{13}x_1x_3+c_{22}x_2^2+c_{23}x_2x_3+c_{33}x_3^2\in Z^2(\mathcal{A}),$  we have
\begin{align*}
0=&\partial_{\mathcal{A}}[c_{11}x_1^2+c_{12}x_1x_2+c_{13}x_1x_3+c_{22}x_2^2+c_{23}x_2x_3+c_{33}x_3^2]\\
=&c_{12}(m_{11}x_1^2+m_{12}x_2^2+m_{13}x_3^2)x_2-c_{12}x_1l_1(m_{11}x_1^2+m_{12}x_2^2+m_{13}x_3^2)\\
+&c_{13}(m_{11}x_1^2+m_{12}x_2^2+m_{13}x_3^2)x_3-c_{13}x_1l_2(m_{11}x_1^2+m_{12}x_2^2+m_{13}x_3^2)\\
+&c_{23}l_1(m_{11}x_1^2+m_{12}x_2^2+m_{13}x_3^2)x_3-c_{23}x_2l_2(m_{11}x_1^2+m_{12}x_2^2+m_{13}x_3^2)\\
=&-(c_{12}l_1+l_2c_{13})m_{11}x_1^3+(c_{12}-c_{23}l_2)m_{11}x_1^2x_2+(c_{13}+c_{23}l_1)m_{11}x_1^2x_3 \\
-& (c_{12}l_1+c_{13}l_2)m_{12}x_1x_2^2 -(c_{12}l_1+c_{13}l_2)m_{13}x_1x_3^2+(c_{12}-c_{23}l_2)m_{12}x_2^3\\
+&(c_{13}+c_{23}l_1)m_{12}x_2^2x_3+(c_{12}-c_{23}l_2)m_{13}x_2x_3^2+(c_{13}+c_{23}l_1)m_{13}x_3^3.
\end{align*}
Since $(m_{11},m_{12},m_{13})\neq 0$, we get
\begin{align*}
\begin{cases}
c_{12}l_1+l_2c_{13}=0\\
c_{12}-c_{23}l_2=0\\
c_{13}+c_{23}l_1=0
\end{cases}\Leftrightarrow \left(
                             \begin{array}{ccc}
                               l_1 & l_2 & 0 \\
                               1 & 0 & -l_2 \\
                               0 & 1 & l_1 \\
                             \end{array}
                           \right)\left(
                                    \begin{array}{c}
                                      c_{12} \\
                                      c_{13} \\
                                      c_{23} \\
                                    \end{array}
                                  \right)=0.
\end{align*}
We get $c_{12}=tl_2,c_{13}=-tl_1,c_{23}=t$, for some $t\in k$. Thus $$Z^2(\mathcal{A})=kx_1^2\oplus kx_2^2\oplus kx_3^2\oplus k(l_2x_1x_2-l_1x_1x_3+x_2x_3).$$
Since $B^2(\mathcal{A})=k(m_{11}x_1^2+m_{12}x_2^2+m_{13}x_3^2)$, we have
\begin{align*}
H^2(\mathcal{A})=\frac{kx_1^2\oplus kx_2^2\oplus kx_3^2\oplus k(l_2x_1x_2-l_1x_1x_3+x_2x_3)}{k(m_{11}x_1^2+m_{12}x_2^2+m_{13}x_3^2)}.
\end{align*}
Moreover,  we claim that $\dim_kH^i(\mathcal{A})=i+1$, for any $i\ge 0$. We prove this claim as follows.

Let $I=(m_{11}x_1^2+m_{12}x_2^2+m_{13}x_3^2)$ be the DG ideal of $\mathcal{A}$ generated by the central coboundary elements $\partial_{\mathcal{A}}(x_1)$. Then the DG quotient ring $Q=\mathcal{A}/I$ has trivial differential and
\begin{align*}
\dim_kH^i(Q)=\dim_k Q^i=\begin{cases}
0, i< 0\\
2i+1,i\ge 0.
\end{cases}
\end{align*}
The short exact sequence
\begin{align*}
0\to I\stackrel{\iota}{\to} \mathcal{A}\stackrel{\pi}{\to} Q\to 0
\end{align*}
induces the cohomology long exact sequence
($\heartsuit$): \begin{align*}
 0 \to H^0(\mathcal{A}) \stackrel{H^0(\pi)}{\to} H^0(Q)\stackrel{\delta^0}{\to} H^1(I)\stackrel{H^1(\iota)}{\to} H^1(\mathcal{A})\stackrel{H^1(\pi)}{\to} H^1(Q)\stackrel{\delta^1}{\to} H^2(I) \\
\stackrel{H^2(\iota)}{\to} H^2(\mathcal{A})\stackrel{H^2(\pi)}{\to}H^2(Q) \stackrel{\delta^2}{\to} \cdots \stackrel{\delta^{i-1}}{\to} H^i(I)\stackrel{H^i(\iota)}{\to} H^i(\mathcal{A})\stackrel{H^i(\pi)}{\to}H^i(Q) \stackrel{\delta^i}{\to}\cdots.
\end{align*}
Since $m_{11}x_1^2+m_{12}x_2^2+m_{13}x_3^2=\partial_{\mathcal{A}}(x_1)$ is a central coboundary elements in $\mathcal{A}$,  we have
 $H^i(I)=\lceil m_{11}x_1^2+m_{12}x_2^2+m_{13}x_3^2\rceil H^{i-2}(\mathcal{A})$ and
 $H^{i}(\iota)=0$ for any $i\ge 2$. The cohomology long exact sequence ($\heartsuit$) implies that $$\dim_k H^i(\mathcal{A})+\dim_k H^{i+1}(I)= \dim_k H^i(Q)=2i+1, i\ge 2.$$
Then $\dim_k H^i(\mathcal{A})+\dim_k H^{i-1}(\mathcal{A})=2i+1$ since
\begin{align*}
\dim_kH^{i+1}(I)&=\dim_k\{\lceil m_{11}x_1^2+m_{12}x_2^2+m_{13}x_3^2\rceil H^{i-1}(\mathcal{A})\}\\
&=\dim_kH^{i-1}(\mathcal{A}), i\ge 2.
\end{align*}
Since $\dim_kH^1(\mathcal{A})=2$,  we can inductively get $\dim_kH^i(\mathcal{A})=i+1$, for any $i\ge 0$.
In order to accomplish the computation of $H(\mathcal{A})$,  we make a classification chart as follows:
$$
\begin{cases}
m_{12}l_1^2+m_{13}l_2^2\neq m_{11}, \begin{cases}
l_1l_2\neq 0;\\
l_1l_2=0;\\
\end{cases}\\
m_{12}l_1^2+m_{13}l_2^2=m_{11}, \begin{cases}
 l_1l_2\neq 0;\\
 l_1\neq 0, l_2= 0;\\
 l_2\neq 0, l_1= 0;\\
 l_1= l_2= 0. \\
 \end{cases}\\
\end{cases}
$$
We will compute $H(\mathcal{A})$ case by case according to this classification chart. We have the following proposition.
\begin{prop}\label{rankone}
$(a)$ If
$m_{12}l_1^2+m_{13}l_2^2\neq m_{11}$ and $l_1l_2\neq 0$, then $H(\mathcal{A})$ is
 $$\frac{k\langle \lceil l_1x_1-x_2\rceil, \lceil l_2x_1-x_3\rceil \rangle}{(m_{12}\lceil l_1x_1-x_2\rceil^2+m_{13}\lceil l_2x_1-x_3\rceil^2-\frac{\lceil l_1x_1-x_2\rceil \lceil l_2x_1-x_3\rceil+\lceil l_2x_1-x_3\rceil \lceil l_1x_1-x_2\rceil}{\frac{2l_1l_2}{m_{12}l_1^2+m_{13}l_2^2}})};$$

$(b)$ If $m_{12}l_1^2+m_{13}l_2^2\neq m_{11}$ and $l_1l_2=0$, then
                               $$H(\mathcal{A})=\frac{k\langle \lceil l_1x_1-x_2\rceil, \lceil l_2x_1-x_3\rceil \rangle}{(\lceil l_1x_1-x_2\rceil\lceil l_2x_1-x_3\rceil+\lceil l_2x_1-x_3\rceil \lceil l_1x_1-x_2\rceil)};$$

$(c)$ If $m_{12}l_1^2+m_{13}l_2^2=m_{11}$ and $l_1l_2\neq 0$, then
                               $$H(\mathcal{A})=\frac{k\langle \lceil l_1x_1-x_2\rceil, \lceil l_2x_1-x_3\rceil \rangle}{(m_{12}\lceil l_1x_1-x_2\rceil^2+m_{13}\lceil l_2x_1-x_3\rceil^2)};$$

$(d)$ If $m_{12}l_1^2+m_{13}l_2^2=m_{11}$, $l_1\neq 0$ and $l_2= 0$,
                           then $$H(\mathcal{A})=\frac{k\langle \lceil l_1x_1-x_2\rceil, \lceil x_3\rceil,\lceil x_1^2\rceil \rangle}{\left(
                                                                                                       \begin{array}{c}
                                                                                                         m_{12}\lceil l_1x_1-x_2\rceil^2+m_{13}\lceil x_3\rceil^2 \\
                                                                                                         \lceil x_1^2\rceil \lceil l_1x_1-x_2\rceil- \lceil l_1x_1-x_2\rceil \lceil x_1^2\rceil \\
                                                                                                       \lceil x_1^2\rceil \lceil x_3\rceil-\lceil x_3\rceil \lceil x_1^2\rceil   \\
                                                                                                       \lceil l_1x_1-x_2\rceil \lceil x_3\rceil +\lceil x_3\rceil \lceil l_1x_1-x_2\rceil
                                                                                                       \end{array}
                                                                                                     \right)};$$

$(e)$ If $m_{12}l_1^2+m_{13}l_2^2=m_{11}$, $l_2\neq 0$ and $l_1= 0$,
                           then $$H(\mathcal{A})=\frac{k\langle \lceil l_2x_1-x_3\rceil, \lceil x_2\rceil,\lceil x_1^2\rceil \rangle}{\left(
                                                                                                       \begin{array}{c}
                                                                                                         m_{13}\lceil l_2x_1-x_3\rceil^2+m_{12}\lceil x_2\rceil^2 \\
                                                                                                         \lceil x_1^2\rceil \lceil l_2x_1-x_3\rceil- \lceil l_2x_1-x_3\rceil \lceil x_1^2\rceil \\
                                                                                                       \lceil x_1^2\rceil \lceil x_2\rceil-\lceil x_2\rceil \lceil x_1^2\rceil   \\
                                                                                                       \lceil l_2x_1-x_3\rceil \lceil x_2\rceil +\lceil x_2\rceil \lceil l_2x_1-x_3\rceil
                                                                                                       \end{array}
                                                                                                     \right)};
$$

$(f)$ If $m_{12}l_1^2+m_{13}l_2^2=m_{11}$, $l_1= 0$ and $l_2= 0$, then
$$H(\mathcal{A})=\frac{k\langle \lceil x_3\rceil, \lceil x_2\rceil,\lceil x_1^2\rceil \rangle}{\left(
                                                                                                       \begin{array}{c}
                                                                                                         m_{12}\lceil x_2\rceil^2+m_{13}\lceil x_3\rceil^2 \\
                                                                                                         \lceil x_1^2\rceil \lceil x_3\rceil- \lceil x_3\rceil \lceil x_1^2\rceil \\
                                                                                                       \lceil x_1^2\rceil \lceil x_2\rceil-\lceil x_2\rceil \lceil x_1^2\rceil   \\
                                                                                                       \lceil x_3\rceil \lceil x_2\rceil +\lceil x_2\rceil \lceil x_3\rceil
                                                                                                       \end{array}
                                                                                                     \right)}.
$$

\end{prop}
\begin{proof}
(a)Note that $x_1x_2+x_2x_1=0,x_1x_3+x_3x_1=0$ and $x_2x_3+x_3x_2=0$
in $\mathcal{A}$. We have  \begin{align*}
\begin{cases}
(l_1x_1-x_2)^2=l_1^2x_1^2+x_2^2,\\
(l_2x_1-x_3)^2=l_2^2x_1^2+x_3^2,\\
(l_1x_1-x_2)(l_2x_1-x_3)+(l_2x_1-x_3)(l_1x_1-x_2)=2l_1l_2x_1^2.
\end{cases}
\end{align*}
It is straight forward to check that $$Z^2(\mathcal{A})=kx_1^2\oplus k(l_1x_1-x_2)^2\oplus k(l_2x_1-x_3)^2\oplus k(l_1x_1-x_2)(l_2x_1-x_3).$$
Since
\begin{align*}
&\quad m_{12}(l_1x_1-x_2)^2+m_{13}(l_2x_1-x_3)^2-(m_{12}l_1^2+m_{13}l_2^2-m_{11})x_1^2\\
&=m_{12}x_2^2+m_{13}x_3^2+m_{11}x_1^2\in B^2(\mathcal{A}),
 \end{align*} we have
\begin{align}\label{notcob}
H^2(\mathcal{A})=k\lceil l_1x_1-x_2\rceil^2\oplus k\lceil l_2x_1-x_3\rceil^2\oplus k\lceil (l_1x_1-x_2)(l_2x_1-x_3)\rceil.
\end{align}
We claim that $$\frac{k\langle \lceil l_1x_1-x_2\rceil, \lceil l_2x_1-x_3\rceil \rangle}{(m_{12}\lceil l_1x_1-x_2\rceil^2+m_{13}\lceil l_2x_1-x_3\rceil^2-\frac{\lceil l_1x_1-x_2\rceil \lceil l_2x_1-x_3\rceil+\lceil l_2x_1-x_3\rceil \lceil l_1x_1-x_2\rceil}{\frac{2l_1l_2}{m_{12}l_1^2+m_{13}l_2^2}}}$$
is a subalgebra of $H(\mathcal{A})$. It suffices to show that \begin{align*}
\begin{cases}
(l_1x_1-x_2)^n\not\in B^{n}(\mathcal{A})\\
(l_2x_1-x_3)^n\not\in B^n(\mathcal{A}) \\
(l_1x_1-x_2)^i(l_2x_1-x_3)^j\not\in  B^{i+j}(\mathcal{A})
\end{cases}
\end{align*}
for any $n\ge 2$ and $i,j\ge 1$. Indeed, if $(l_1x_1-x_2)^n\in B^{n}(\mathcal{A})$ then we have
$$(l_1x_1-x_2)^n=\begin{cases}
\partial_{\mathcal{A}}[x_1x_2f+x_1x_3g+x_2x_3h],\,\, \text{if}\,\,n=2j+1\,\, \text{is odd}\\
\partial_{\mathcal{A}}[x_1f+x_2g+x_3h+x_1x_2x_3u], \,\, \text{if}\,\,n=2j\,\, \text{is even},
\end{cases}
$$
where $f,g,h$ and $u$ are all linear combinations of monomials with non-negative even exponents.
 When $n=2j$ is even, we have
 \begin{align*}
\quad (l_1^2x_1^2+&x_2^2)^j =(l_1x_1-x_2)^n\\
&=\partial_{\mathcal{A}}[x_1f+x_2g+x_3h+x_1x_2x_3u]\\
                                    &=(m_{11}x_1^2+m_{12}x_2^2+m_{13}x_3^2)f+l_1(m_{11}x_1^2+m_{12}x_2^2+m_{13}x_3^2)g\\
                                    &+l_2(m_{11}x_1^2+m_{12}x_2^2+m_{13}x_3^2)h+(m_{11}x_1^2+m_{12}x_2^2+m_{13}x_2^2)x_2x_3u \\
                                    &-x_1l_1(m_{11}x_1^2+m_{12}x_2^2+m_{13}x_3^2)x_3u+x_1x_2l_2(m_{11}x_1^2+m_{12}x_2^2+m_{13}x_3^2)u.
 \end{align*}
Considering the parity of exponents of the monomials that appear on both sides of the equation above implies that
\begin{align*}
(l_1^2x_1^2+x_2^2)^j&=(m_{11}x_1^2+m_{12}x_2^2+m_{13}x_3^2)f+l_1(m_{11}x_1^2+m_{12}x_2^2+m_{13}x_3^2)g\\
                                    &+l_2(m_{11}x_1^2+m_{12}x_2^2+m_{13}x_3^2)h\\
                                    &=\partial_{\mathcal{A}}(x_1)[f+l_1g+l_2h]
\end{align*}
and
\begin{align*}
\partial_{\mathcal{A}}(xyzu)
&=(m_{11}x_1^2+m_{12}x_2^2+m_{13}x_3^2)x_2x_3u-l_1x_1(m_{11}x_1^2+m_{12}x_2^2+m_{13}x_3^2)x_3u \\
&+x_1x_2l_2(m_{11}x_1^2+m_{12}x_2^2+m_{13}x_3^2)u=0.
\end{align*}
 Hence $(l_1^2x_1^2+x_2^2)^j$ is in the graded ideal $(\partial_{\mathcal{A}}(x_1))$ of $k[x_1^2,x_2^2,x_3^2]$. By Lemma \ref{prime}, $(\partial_{\mathcal{A}}(x_1),\partial_{\mathcal{A}}(x_2),\partial_{\mathcal{A}}(x_3))=(\partial_{\mathcal{A}}(x_1))$ is a graded prime ideal of $k[x_1^2,x_2^2,x_3^2]$. So $l_1^2x_1^2+x_2^2\in (\partial_{\mathcal{A}}(x_1))$. Hence there exist $a_1\in k$ such that \begin{align*}
l_1^2x_1^2+x_2^2&=a_1\partial_{\mathcal{A}}(x_1)=\partial_{\mathcal{A}}(a_1x_1).
\end{align*}
But this contradicts with the fact that $l_1^2x_1^2+x_2^2\not\in B^2(\mathcal{A})$, which we have proved above.
Thus $(l_1x_1-x_2)^n\not\in B^n(\mathcal{A})$ when $n$ is even.

When $n=2j+1$ is odd, we have
 \begin{align*}
 &(l_1x_1-x_2)(l_1^2x_1^2+x_2^2)^j =(l_1x_1-x_2)^n=\partial_{\mathcal{A}}[x_1x_2f+x_1x_3g+x_2x_3h]\\
                                    &=(m_{11}x_1^2+m_{12}x_2^2+m_{13}x_3^2)x_2f-l_1x_1(m_{11}x_1^2+m_{12}x_2^2+m_{13}x_3^2)f\\
                                    &+(m_{11}x_1^2+m_{12}x_2^2+m_{13}x_3^2)x_3g-l_2x_1(m_{11}x_1^2+m_{12}x_2^2+m_{13}x_3^2)g \\
                                    &+l_1(m_{11}x_1^2+m_{12}x_2^2+m_{13}x_3^2)x_3h-l_2x_2(m_{11}x_1^2+m_{12}x_2^2+m_{13}x_3^2)h\\
                       &=x_2(m_{11}x_1^2+m_{12}x_2^2+m_{13}x_3^2)(f-l_2h)-x_1(m_{11}x_1^2+m_{12}x_2^2+m_{13}x_3^2)(l_1f+l_2g)\\
                                    &+x_3(m_{11}x_1^2+m_{12}x_2^2+m_{13}x_3^2)(g+l_1h)\\
                                    &=(m_{11}x_1^2+m_{12}x_2^2+m_{13}x_3^2)[x_2(f-l_2h)-x_1(l_1f+l_2g)+x_3(g+l_1h)]\\
                                    &=x_1[-\partial_{\mathcal{A}}(x_2)f-\partial_{\mathcal{A}}(x_3)g]+x_2[\partial_{\mathcal{A}}(x_1)f-\partial_{\mathcal{A}}(x_3)h]+x_3[\partial_{\mathcal{A}}(x_2)h+\partial_{\mathcal{A}}(x_1)g].
 \end{align*}
This implies that
\begin{align*}
\begin{cases}
l_1(l_1^2x_1^2+x_2^2)^j=-(l_1f+l_2g)(m_{11}x_1^2+m_{12}x_2^2+m_{13}x_3^2)\\
(l_1^2x_1^2+x_2^2)^j =(m_{11}x_1^2+m_{12}x_2^2+m_{13}x_3^2)(l_2h-f)\\
0=g+l_1h.
\end{cases}
\end{align*}
 Then $(l_1^2x_1^2+x_2^2)^j=(l_1x_1-x_2)^{2j}\in B^{2j}(\mathcal{A})$, which contradicts with the proved fact that $(l_1x_1-x_2)^n\not\in B^n(\mathcal{A})$ when $n$ is even. Therefore, $(l_1x_1-x_2)^n\not\in B^n(\mathcal{A})$ when $n$ is odd.  Then $(l_1x_1-x_2)^n\not\in B^n(\mathcal{A})$ for any $n\ge 3$. Similarly,  we can show that
  \begin{align*}
  \begin{cases}
  (l_2x_1-x_3)^n\not\in B^n(\mathcal{A}), \,\, \text{for any}\,\, n\ge 3\\
  (l_1x_1-x_2)^{2i+1}(l_2x_1-x_3)^{2j}\not\in  B^{2i+2j+1}(\mathcal{A}),  \,\, \text{for any}\,\, i,j\ge 1  \\
  (l_1x_1-x_2)^{2i}(l_2x_1-x_3)^{2j+1}\not\in  B^{2i+2j+1}(\mathcal{A}),  \,\, \text{for any}\,\, i,j\ge 1  \\
  (l_1x_1-x_2)^{2i}(l_2x_1-x_3)^{2j}\not\in B^{2i+2j}(\mathcal{A}), \,\, \text{for any}\,\, i,j\ge 1.
  \end{cases}
  \end{align*}
  It remains to prove $(l_1x_1-x_2)^{2i+1}(l_2x_1-x_3)^{2j+1}\not\in  B^{2i+2j+2}(\mathcal{A})$ for any $i,j\ge 1.$
If $(l_1x_1-x_2)^{2i+1}(l_2x_1-x_3)^{2j+1}\in B^{2i+2j+2}(\mathcal{A})$, then
\begin{align*}
&\quad (l_1l_2x_1^2-l_1x_1x_3+l_2x_1x_2+x_2x_3)(l_1^2x_1^2+x_2^2)^i(l_2^2x_1^2+x_3^2)^j \\
&=(l_1x_1-x_2)^{2i+1}(l_2x_1-x_3)^{2j+1}=\partial_{\mathcal{A}}[x_1f+x_2g+x_3h+x_1x_2x_3u]\\
                                    &=(m_{11}x_1^2+m_{12}x_2^2+m_{13}x_3^2)f+l_1(m_{11}x_1^2+m_{12}x_2^2+m_{13}x_3^2)g\\
                                    &+l_2(m_{11}x_1^2+m_{12}x_2^2+m_{13}x_3^2)h+(m_{11}x_1^2+m_{12}x_2^2+m_{13}x_3^2)x_2x_3u \\
                                    &-x_1l_1(m_{11}x_1^2+m_{12}x_2^2+m_{13}x_3^2)x_3u+x_1x_2l_2(m_{11}x_1^2+m_{12}x_2^2+m_{13}x_3^2)u.
\end{align*}
where $f,g,h$ and $u$ are all linear combinations of monomials with non-negative even exponents.
Hence
\begin{align*}
l_1l_2x_1^2(l_1^2x_1^2+x_2^2)^i&(l_2^2x_1^2+x_3^2)^j=(m_{11}x_1^2+m_{12}x_2^2+m_{13}x_3^2)f\\
&+l_1(m_{11}x_1^2+m_{12}x_2^2+m_{13}x_3^2)g+l_2(m_{11}x_1^2+m_{12}x_2^2+m_{13}x_3^2)h
\end{align*}
and
$$(l_1^2x_1^2+x_2^2)^i(l_2^2x_1^2+x_3^2)^j=(m_{11}x_1^2+m_{12}x_2^2+m_{13}x_3^2)u\in (\partial_{\mathcal{A}}(x_1)).$$
Since $(\partial_{\mathcal{A}}(x_1))$ is a prime ideal in $k[x_1^2,x_2^2,x_3^2]$, we conclude that $(l_1^2x_1^2+x_2^2)\in (\partial_{\mathcal{A}}(x_1))$ or
$l_2^2x_1^2+x_3^2 \in  (\partial_{\mathcal{A}}(x_1))$. This contradicts with (\ref{notcob}). By the discussion above,
$$\frac{k\langle \lceil l_1x_1-x_2\rceil, \lceil l_2x_1-x_3\rceil \rangle}{(m_{12}\lceil l_1x_1-x_2\rceil^2+m_{13}\lceil l_2x_1-x_3\rceil^2-\frac{\lceil l_1x_1-x_2\rceil \lceil l_2x_1-x_3\rceil+\lceil l_2x_1-x_3\rceil \lceil l_1x_1-x_2\rceil}{\frac{2l_1l_2}{m_{12}l_1^2+m_{13}l_2^2}})}$$
is a subalgebra of $H(\mathcal{A})$. On the other hand, we have $\dim_kH^i(\mathcal{A})=i+1$.
Then we can conclude that $H(\mathcal{A})$ is $$\frac{k\langle \lceil l_1x_1-x_2\rceil, \lceil l_2x_1-x_3\rceil \rangle}{(m_{12}\lceil l_1x_1-x_2\rceil^2+m_{13}\lceil l_2x_1-x_3\rceil^2-\frac{\lceil l_1x_1-x_2\rceil \lceil l_2x_1-x_3\rceil+\lceil l_2x_1-x_3\rceil \lceil l_1x_1-x_2\rceil}{\frac{2l_1l_2}{m_{12}l_1^2+m_{13}l_2^2}})}.$$

(b) In this case, $m_{12}l_1^2+m_{13}l_2^2\neq m_{11}$ and $l_1l_2=0$. One sees that
$$\begin{cases}
(l_1x_1-x_2)^2=l_1^2x_1^2+x_2^2,\\
(l_2x_1-x_3)^2=l_2^2x_1^2+x_3^2,\\
(l_1x_1-x_2)(l_2x_1-x_3)+(l_2x_1-x_3)(l_1x_1-x_2)=2l_1l_2x_1^2=0.
\end{cases}$$

It is straight forward to check that $$Z^2(\mathcal{A})=kx_1^2\oplus k(l_1x_1-x_2)^2\oplus k(l_2x_1-x_3)^2\oplus k(l_1x_1-x_2)(l_2x_1-x_3).$$
Since
\begin{align*}
&\quad m_{12}(l_1x_1-x_2)^2+m_{13}(l_2x_1-x_3)^2-(m_{12}l_1^2+m_{13}l_2^2-m_{11})x_1^2\\
&=m_{12}x_2^2+m_{13}x_3^2+m_{11}x_1^2\in B^2(\mathcal{A}),
 \end{align*} we have
$$H^2(\mathcal{A})=k\lceil l_1x_1-x_2\rceil^2\oplus k\lceil l_2x_1-x_3\rceil^2\oplus k\lceil (l_1x_1-x_2)(l_2x_1-x_3)\rceil.$$
 Just as the proof of (a), we can show that
$$\frac{k\langle \lceil l_1x_1-x_2\rceil, \lceil l_2x_1-x_3\rceil \rangle}{(\lceil l_1x_1-x_2\rceil\lceil l_2x_1-x_3\rceil+\lceil l_2x_1-x_3\rceil \lceil l_1x_1-x_2\rceil)}$$
is a subalgebra of $H(\mathcal{A})$. On the other hand, we have $\dim_kH^i(\mathcal{A})=i+1$.
Then we can conclude that
$$H(\mathcal{A})=\frac{k\langle \lceil l_1x_1-x_2\rceil, \lceil l_2x_1-x_3\rceil \rangle}{(\lceil l_1x_1-x_2\rceil\lceil l_2x_1-x_3\rceil+\lceil l_2x_1-x_3\rceil \lceil l_1x_1-x_2\rceil)}.$$

(c)
In this case, $m_{12}l_1^2+m_{13}l_2^2=m_{11}$ and $l_1l_2\neq 0$. So we have
\begin{align*}
m_{12}(l_1x_1-x_2)^2+m_{13}(l_2x_1-x_3)^2 &=(m_{12}l_1^2+m_{13}l_2^2)x_1^2+m_{12}x_2^2+m_{13}x_3^2\\
                                          &=m_{11}x_1^2+m_{12}x_2^2+m_{13}x_3^2 =\partial_{\mathcal{A}}(x_1)
\end{align*}
and
$$\begin{cases}
(l_1x_1-x_2)(l_2x_1-x_3)+(l_2x_1-x_3)(l_1x_1-x_2)=2l_1l_2x_1^2\\
(l_1x_1-x_2)(l_2x_1-x_3)-(l_2x_1-x_3)(l_1x_1-x_2)=2[x_2x_3-l_1x_1x_3+l_2x_1x_2].
\end{cases}
$$
Hence $H^2(\mathcal{A})$ is
$$\frac{k(l_1x_1-x_2)(l_2x_1-x_3)\oplus k(l_2x_1-x_3)(l_1x_1-x_2)\oplus k(l_1x_1-x_2)^2\oplus k(l_2x_1-x_3)^2}{k[m_{12}(l_1^2x_1^2+x_2^2)+m_{13}(l_2^2x_1^2+x_3^2)]}.
$$
 Just as the proof of (a), we can show that $$\frac{k\langle \lceil l_1x_1-x_2\rceil, \lceil l_2x_1-x_3\rceil \rangle}{(m_{12}\lceil l_1x_1-x_2\rceil^2+m_{13}\lceil l_2x_1-x_3\rceil^2)}$$
is a subalgebra of $H(\mathcal{A})$. Since  $\dim_kH^i(\mathcal{A})=i+1$, we can conclude that
 $$H(\mathcal{A})=\frac{k\langle \lceil l_1x_1-x_2\rceil, \lceil l_2x_1-x_3\rceil \rangle}{(m_{12}\lceil l_1x_1-x_2\rceil^2+m_{13}\lceil l_2x_1-x_3\rceil^2)}.$$

(d)
Since $m_{12}l_1^2+m_{13}l_2^2=m_{11}$, $l_1\neq 0$ and $l_2= 0$, we have $m_{12}l_1^2=m_{11}$,
\begin{align*}
m_{12}(l_1x_1-x_2)^2+m_{13}x_3^2 &=m_{12}l_1^2x_1^2+m_{12}x_2^2+m_{13}x_3^2\\
                                          &=m_{11}x_1^2+m_{12}x_2^2+m_{13}x_3^2 =\partial_{\mathcal{A}}(x_1)
\end{align*}
and $(l_1x_1-x_2)x_3+z(l_1x_1-x_2)=l_1(x_1x_3+x_3x_1)-(x_2x_3+x_3x_2)=0$. Thus
$$H^2(\mathcal{A})=\frac{kx_3^2\oplus k(l_1^2x_1^2+x_2^2)\oplus k(l_1x_1-x_2)x_3\oplus kx_1^2}{k[m_{12}(l_1x_1-x_2)^2+m_{13}x_3^2]}.$$
 Just as the proof of (a), we can show that
 $$\frac{k\langle \lceil l_1x_1-x_2\rceil, \lceil x_3\rceil,\lceil x_1^2\rceil \rangle}{\left(
                                                                                                       \begin{array}{c}
                                                                                                         m_{12}\lceil l_1x_1-x_2\rceil^2+m_{13}\lceil x_3\rceil^2 \\
                                                                                                         \lceil x_1^2\rceil \lceil l_1x_1-x_2\rceil- \lceil l_1x_1-x_2\rceil \lceil x_1^2\rceil \\
                                                                                                       \lceil x_1^2\rceil \lceil x_3\rceil-\lceil x_3\rceil \lceil x_1^2\rceil   \\
                                                                                                       \lceil l_1x_1-x_2\rceil \lceil x_3\rceil +\lceil x_3\rceil \lceil l_1x_1-x_2\rceil
                                                                                                       \end{array}
                                                                                                     \right)}
$$
is a subalgebra of $H(\mathcal{A})$. Since $\dim_kH^i(\mathcal{A})=i+1$, we get
 $$H(\mathcal{A})=\frac{k\langle \lceil l_1x_1-x_2\rceil, \lceil x_3\rceil,\lceil x_1^2\rceil \rangle}{\left(
                                                                                                       \begin{array}{c}
                                                                                                         m_{12}\lceil l_1x_1-x_2\rceil^2+m_{13}\lceil x_3\rceil^2 \\
                                                                                                         \lceil x_1^2\rceil \lceil l_1x_1-x_2\rceil- \lceil l_1x_1-x_2\rceil \lceil x_1^2\rceil \\
                                                                                                       \lceil x_1^2\rceil \lceil x_3\rceil-\lceil x_3\rceil \lceil x_1^2\rceil   \\
                                                                                                       \lceil l_1x_1-x_2\rceil \lceil x_3\rceil +\lceil x_3\rceil \lceil l_1x_1-x_2\rceil
                                                                                                       \end{array}
                                                                                                     \right)}.$$

(e) In this case, we have $m_{12}l_1^2+m_{13}l_2^2=m_{11}$, $l_2\neq 0$ and $l_1= 0$.  So $m_{13}l_2^2=m_{11}$,
\begin{align*}
m_{13}(l_2x_1-x_3)^2+m_{12}x_2^2 &=m_{13}l_2^2x_1^2+m_{12}x_2^2+m_{13}x_3^2\\
                                          &=m_{11}x_1^2+m_{12}x_2^2+m_{13}x_3^2 =\partial_{\mathcal{A}}(x_1)
\end{align*}
and $(l_2x_1-x_3)x_2+x_2(l_2x_1-x_3)=l_2(x_1x_2+x_2x_1)-(x_2x_3+x_3x_2)=0$. Thus
$$H^2(\mathcal{A})=\frac{kx_2^2\oplus k(l_2^2x_1^2+x_3^2)\oplus k(l_2x_1-x_3)x_2\oplus kx_1^2}{k[m_{13}(l_2x_1-x_3)^2+m_{12}x_2^2]}.$$
 Just as the proof of (1), we can show that
 $$\frac{k\langle \lceil l_2x_1-x_3\rceil, \lceil x_2\rceil,\lceil x_1^2\rceil \rangle}{\left(
                                                                                                       \begin{array}{c}
                                                                                                         m_{13}\lceil l_2x_1-x_3\rceil^2+m_{12}\lceil x_2\rceil^2 \\
                                                                                                         \lceil x_1^2\rceil \lceil l_2x_1-x_3\rceil- \lceil l_2x_1-x_3\rceil \lceil x_1^2\rceil \\
                                                                                                       \lceil x_1^2\rceil \lceil x_2\rceil-\lceil x_2\rceil \lceil x_1^2\rceil   \\
                                                                                                       \lceil l_2x_1-x_3\rceil \lceil x_2\rceil +\lceil x_2\rceil \lceil l_2x_1-x_3\rceil
                                                                                                       \end{array}
                                                                                                     \right)}
$$
is a subalgebra of $H(\mathcal{A})$.  Since $\dim_kH^i(\mathcal{A})=i+1$, we have $$H(\mathcal{A})=\frac{k\langle \lceil l_2x_1-x_3\rceil, \lceil x_2\rceil,\lceil x_1^2\rceil \rangle}{\left(
                                                                                                       \begin{array}{c}
                                                                                                         m_{13}\lceil l_2x_1-x_3\rceil^2+m_{12}\lceil x_2\rceil^2 \\
                                                                                                         \lceil x_1^2\rceil \lceil l_2x_1-x_3\rceil- \lceil l_2x_1-x_3\rceil \lceil x_1^2\rceil \\
                                                                                                       \lceil x_1^2\rceil \lceil x_2\rceil-\lceil x_3\rceil \lceil x_1^2\rceil   \\
                                                                                                       \lceil l_2x_1-x_3\rceil \lceil x_2\rceil +\lceil x_2\rceil \lceil l_2x_1-x_3\rceil
                                                                                                       \end{array}
                                                                                                     \right)}.
$$

(f) In this case $m_{11}=0$,  and hence $\begin{cases}
\partial_{\mathcal{A}}(x_1)=m_{12}x_2^2+m_{13}x_3^2\\
\partial_{\mathcal{A}}(x_2)=0\\
\partial_{\mathcal{A}}(x_3)=0.
\end{cases}
$
So $$H^2(\mathcal{A})=\frac{kx_1^2\oplus kx_2^2\oplus kx_3^2\oplus kx_2x_3}{k(m_{12}x_2^2+m_{13}x_3^2)}.$$
 Just as the proof of (a), we can show that  $$\frac{k\langle \lceil x_3\rceil, \lceil x_2\rceil,\lceil x_1^2\rceil \rangle}{\left(
                                                                                                       \begin{array}{c}
                                                                                                         m_{12}\lceil x_2\rceil^2+m_{13}\lceil x_3\rceil^2 \\
                                                                                                         \lceil x_1^2\rceil \lceil x_3\rceil- \lceil x_3\rceil \lceil x_1^2\rceil \\
                                                                                                       \lceil x_1^2\rceil \lceil x_2\rceil-\lceil x_2\rceil \lceil x_1^2\rceil   \\
                                                                                                       \lceil x_3\rceil \lceil x_2\rceil +\lceil x_2\rceil \lceil x_3\rceil
                                                                                                       \end{array}
                                                                                                     \right)}
$$
is a subalgebra of $H(\mathcal{A})$. Since $\dim_kH^i(\mathcal{A})=i+1$, we conclude
$$H(\mathcal{A})=\frac{k\langle \lceil x_3\rceil, \lceil x_2\rceil,\lceil x_1^2\rceil \rangle}{\left(
                                                                                                       \begin{array}{c}
                                                                                                         m_{12}\lceil x_2\rceil^2+m_{13}\lceil x_3\rceil^2 \\
                                                                                                         \lceil x_1^2\rceil \lceil x_3\rceil- \lceil x_3\rceil \lceil x_1^2\rceil \\
                                                                                                       \lceil x_1^2\rceil \lceil x_2\rceil-\lceil x_2\rceil \lceil x_1^2\rceil   \\
                                                                                                       \lceil x_3\rceil \lceil x_2\rceil +\lceil x_2\rceil \lceil x_3\rceil
                                                                                                       \end{array}
                                                                                                     \right)}.
$$
\end{proof}

\section{some applications}
Let $\mathcal{A}$ be a connected cochain DG algebra such that its underlying graded algebra $\mathcal{A}^{\#}$ is the graded skew polynomial algebra
$$k\langle x_1,x_2, x_3\rangle/\left(\begin{array}{ccc}
x_1x_2+x_2x_1\\
x_2x_3+x_3x_2\\
x_3x_1+x_1x_3\\
                                 \end {array}\right), |x_1|=|x_2|=|x_3|=1.$$
Then
 $\partial_{\mathcal{A}}$ is determined by a matrix $M\in M_3(k)$ such that
\begin{align*}
\left(
                         \begin{array}{c}
                           \partial_{\mathcal{A}}(x_1)\\
                           \partial_{\mathcal{A}}(x_2)\\
                           \partial_{\mathcal{A}}(x_3)
                         \end{array}
                       \right)=M\left(
                         \begin{array}{c}
                           x_1^2\\
                           x_2^2\\
                           x_3^2
                         \end{array}
                       \right),\text{for some}\, M\in M_3(k).
\end{align*}
By the computations in Section \ref{cohomology},  we reach the following conclusion.
\begin{prop}\label{rank1}
$H(\mathcal{A})$ is an AS-Gorenstein graded algebra when $r(M)\neq 1$.
\end{prop}
\begin{proof}

If $r(M)=0$, then $H(\mathcal{A})=\mathcal{A}^{\#}$ is obviously an AS-Gorenstein graded algebra since $\mathcal{A}^{\#}$ is an AS-regular algebra of dimension $3$.
By Proposition \ref{rank3}, we have $H(\mathcal{A})=k$ if $r(M)=3$. So the statement of the proposition is also right when $r(M)=3$.

For the case $r(M)=2$, let $k(s_1,s_2,s_3)^T$ and $k(t_1,t_2,t_3)^T$ be the solution spaces of  homogeneous linear equations $MX=0$ and $M^TX=0$, respectively.
By Proposition \ref{rmtwoimc},  $H(\mathcal{A})=k[\lceil t_1x_1 +t_2x_2+t_3x_3\rceil]$ if $s_1t_1^2+s_2t_2^2+s_3t_3^2\neq 0$; and  $H(\mathcal{A})$ equals to
$$k[\lceil t_1x_1 +t_2x_2+t_3x_3\rceil,\lceil s_1x_1^2+s_2x_2^2+s_3x_3^2\rceil ]/(\lceil t_1x_1 +t_2x_2+t_3x_3\rceil^2)$$ when $s_1t_1^2+s_2t_2^2+s_3t_3^2= 0$.
Since \begin{align*}
&\quad k[\lceil t_1x_1 +t_2x_2+t_3x_3\rceil,\lceil s_1x_1^2+s_2x_2^2+s_3x_3^2\rceil ]/(\lceil t_1x_1 +t_2x_2+t_3x_3\rceil^2)\\
&\cong \frac{k[\lceil t_1x_1 +t_2x_2+t_3x_3\rceil]}{(\lceil t_1x_1 +t_2x_2+t_3x_3\rceil^2)}[\lceil s_1x_1^2+s_2x_2^2+s_3x_3^2\rceil],
\end{align*}
it is AS-Gorenstein by Lemma \ref{extaslem}. Thus $H(\mathcal{A})$ is an AS-Gorenstein graded algebra when $r(M)=2$.
\end{proof}

Now, it remains to consider the case that $r(M)=1$. We may assume that $$M=\left(
                                 \begin{array}{ccc}
                                   m_{11} & m_{12} & m_{13} \\
                                   l_1m_{11} & l_1m_{12} & l_1m_{13} \\
                                   l_2m_{11} & l_2m_{12} & l_2m_{13} \\
                                 \end{array}
                               \right),\,\,\text{with}\,\, l_1,l_2\in k \,\, \text{and}\,\, (m_{11}, m_{12},m_{13})\neq 0.$$
We have the following proposition.
\begin{prop}\label{ASgor}
The graded algebra $H(\mathcal{A})$ is AS-Gorenstein if we have any one of the following conditions:
\begin{enumerate}
\item $m_{12}l_1^2+m_{13}l_2^2\neq m_{11}$ and $l_1l_2=0$;
\item  $m_{12}l_1^2+m_{13}l_2^2=m_{11}$, $l_1\neq 0$ and $l_2= 0$;
\item   $m_{12}l_1^2+m_{13}l_2^2=m_{11}$, $l_2\neq 0$ and $l_1= 0$;
\item  $m_{12}l_1^2+m_{13}l_2^2=m_{11}$, $l_1= 0$ and $l_2= 0$;
\item  $m_{12}l_1^2+m_{13}l_2^2=m_{11},l_1l_2\neq 0$ and  $m_{12}m_{13}\neq 0$;
\item $m_{12}l_1^2+m_{13}l_2^2\neq m_{11},l_1l_2\neq 0$ and $4m_{12}m_{13}l_1^2l_2^2\neq (m_{12}l_1^2+m_{13}l_2^2-m_{11})^2$.
\end{enumerate}
\end{prop}
\begin{proof}
By Proposition \ref{rankone}(b), we have $$H(\mathcal{A})=\frac{k\langle \lceil l_1x_1-x_2\rceil, \lceil l_2x_1-x_3\rceil \rangle}{(\lceil l_1x_1-x_2\rceil\lceil l_2x_1-x_3\rceil+\lceil l_2x_1-x_3\rceil \lceil l_1x_1-x_2\rceil)},$$
when $m_{12}l_1^2+m_{13}l_2^2\neq m_{11}$ and $l_1l_2=0$. In this case, $H(\mathcal{A})$ is an AS-regular graded algebra of dimension $2$.

By Proposition \ref{rankone}(d),   $$H(\mathcal{A})=\frac{k\langle \lceil l_1x_1-x_2\rceil, \lceil x_3\rceil,\lceil x_1^2\rceil \rangle}{\left(
                                                                                                       \begin{array}{c}
                                                                                                         m_{12}\lceil l_1x_1-x_2\rceil^2+m_{13}\lceil x_3\rceil^2 \\
                                                                                                         \lceil x_1^2\rceil \lceil l_1x_1-x_2\rceil- \lceil l_1x_1-x_2\rceil \lceil x_1^2\rceil \\
                                                                                                       \lceil x_1^2\rceil \lceil x_3\rceil-\lceil x_3\rceil \lceil x_1^2\rceil   \\
                                                                                                       \lceil l_1x_1-x_2\rceil \lceil x_3\rceil +\lceil x_3\rceil \lceil l_1x_1-x_2\rceil
                                                                                                       \end{array}
                                                                                                     \right)}$$
when  $m_{12}l_1^2+m_{13}l_2^2=m_{11}$, $l_1\neq 0$ and $l_2= 0$. We have
\begin{align*}
H(\mathcal{A})&=\frac{k\langle \lceil l_1x_1-x_2\rceil, \lceil x_3\rceil,\lceil x_1^2\rceil \rangle}{\left(
                                                                                                       \begin{array}{c}
                                                                                                         m_{12}\lceil l_1x_1-x_2\rceil^2+m_{13}\lceil x_3\rceil^2 \\
                                                                                                         \lceil x_1^2\rceil \lceil l_1x_1-x_2\rceil- \lceil l_1x_1-x_2\rceil \lceil x_1^2\rceil \\
                                                                                                       \lceil x_1^2\rceil \lceil x_3\rceil-\lceil x_3\rceil \lceil x_1^2\rceil   \\
                                                                                                       \lceil l_1x_1-x_2\rceil \lceil x_3\rceil +\lceil x_3\rceil \lceil l_1x_1-x_2\rceil
                                                                                                       \end{array}
                                                                                                     \right)}\\
&\cong \frac{k\langle \lceil l_1x_1-x_2\rceil, \lceil x_3\rceil \rangle}{\left(
                                                                                                       \begin{array}{c}
                                                                                                         m_{12}\lceil l_1x_1-x_2\rceil^2+m_{13}\lceil x_3\rceil^2 \\
                                                                                                       \lceil l_1x_1-x_2\rceil \lceil x_3\rceil +\lceil x_3\rceil \lceil l_1x_1-x_2\rceil
                                                                                                       \end{array}
                                                                                                     \right)}[\lceil x_1^2\rceil].
\end{align*}
By Rees Lemma, one sees that $$\frac{k\langle \lceil l_1x_1-x_2\rceil, \lceil x_3\rceil \rangle}{\left(
                                                                                                       \begin{array}{c}
                                                                                                         m_{12}\lceil l_1x_1-x_2\rceil^2+m_{13}\lceil x_3\rceil^2 \\
                                                                                                       \lceil l_1x_1-x_2\rceil \lceil x_3\rceil +\lceil x_3\rceil \lceil l_1x_1-x_2\rceil
                                                                                                       \end{array}
                                                                                                     \right)}$$ is AS-Gorenstein. Applying Lemma \ref{extaslem}, we get that $H(\mathcal{A})$ is AS-Gorenstein.
By Proposition \ref{rankone}(e) and (f),  we can similarly show that $H(\mathcal{A})$ is  AS-Gorenstein if we have either
$$m_{12}l_1^2+m_{13}l_2^2=m_{11}, l_2\neq 0, l_1=0$$ or
$$m_{12}l_1^2+m_{13}l_2^2=m_{11}, l_1= 0, l_2= 0.$$

When $m_{12}l_1^2+m_{13}l_2^2=m_{11},l_1l_2\neq 0$ and  $m_{12}m_{13}\neq 0$, we have
$$H(\mathcal{A})=\frac{k\langle \lceil l_1x_1-x_2\rceil, \lceil l_2x_1-x_3\rceil \rangle}{(m_{12}\lceil l_1x_1-x_2\rceil^2+m_{13}\lceil l_2x_1-x_3\rceil^2)}$$
by Proposition \ref{rankone}(c). Since $m_{12}m_{13}\neq 0$, the graded algebra $H(\mathcal{A})$ is AS-regular by \cite[Proposition 1.1]{Zhang}.

When $m_{12}l_1^2+m_{13}l_2^2\neq m_{11},l_1l_2\neq 0$ and $4m_{12}m_{13}l_1^2l_2^2\neq (m_{12}l_1^2+m_{13}l_2^2-m_{11})^2$, the graded algebra $H(\mathcal{A})$ is
$$\frac{k\langle \lceil l_1x_1-x_2\rceil, \lceil l_2x_1-x_3\rceil \rangle}{(m_{12}\lceil l_1x_1-x_2\rceil^2+m_{13}\lceil l_2x_1-x_3\rceil^2-\frac{\lceil l_1x_1-x_2\rceil \lceil l_2x_1-x_3\rceil+\lceil l_2x_1-x_3\rceil \lceil l_1x_1-x_2\rceil}{\frac{2l_1l_2}{m_{12}l_1^2+m_{13}l_2^2}})}$$
by Proposition \ref{rankone}(a). Since $4m_{12}m_{13}l_1^2l_2^2\neq (m_{12}l_1^2+m_{13}l_2^2-m_{11})^2$, one sees that $H(\mathcal{A})$ is AS-regular by \cite[Proposition 1.1]{Zhang}.
\end{proof}

\begin{thm}\label{nonasgoren}
Let $\mathcal{A}$ be a connected cochain DG algebra such that
$$\mathcal{A}^{\#}=k\langle x_1,x_2, x_3\rangle/\left(\begin{array}{ccc}
x_1x_2+x_2x_1\\
x_2x_3+x_3x_2\\
x_3x_1+x_1x_3\\
                                 \end {array}\right), |x_1|=|x_2|=|x_3|=1,$$
and $\partial_A$ is determined by
\begin{align*}
\left(
                         \begin{array}{c}
                           \partial_{\mathcal{A}}(x_1)\\
                           \partial_{\mathcal{A}}(x_2)\\
                           \partial_{\mathcal{A}}(x_3)
                         \end{array}
                       \right)=N\left(
                         \begin{array}{c}
                           x_1^2\\
                           x_2^2\\
                           x_3^2
                         \end{array}
                       \right).
\end{align*}
Then the graded algebra $H(\mathcal{A})$ is not left (right) Gorenstein if and only if
there exists some  $C=(c_{ij})_{3\times 3}\in \mathrm{QPL}_3(k)$ satisfying $N=C^{-1}M(c_{ij}^2)_{3\times 3}$,
where
$$M=\left(
                                 \begin{array}{ccc}
                                   1 & 1 & 0 \\
                                   1 & 1 & 0 \\
                                   1 & 1 & 0 \\
                                 \end{array}
                               \right)
\,\,\text{or}\,\,M=\left(
                                 \begin{array}{ccc}
                                   m_{11} & m_{12} & m_{13} \\
                                   l_1m_{11} & l_1m_{12} & l_1m_{13} \\
                                   l_2m_{11} & l_2m_{12} & l_2m_{13} \\
                                 \end{array}
                               \right)$$ with $m_{12}l_1^2+m_{13}l_2^2\neq m_{11}, l_1l_2\neq 0$ and $4m_{12}m_{13}l_1^2l_2^2= (m_{12}l_1^2+m_{13}l_2^2-m_{11})^2$.
\end{thm}
\begin{proof}
First, let us prove the `if' part.  Suppose that there exists some  $C=(c_{ij})_{3\times 3}\in \mathrm{QPL}_3(k)$ satisfying $N=C^{-1}M(c_{ij}^2)_{3\times 3}$,
where $$M=\left(
                                 \begin{array}{ccc}
                                   1 & 1 & 0 \\
                                   1 & 1 & 0 \\
                                   1 & 1 & 0 \\
                                 \end{array}
                               \right)
\,\,\text{or}\,\,M=\left(
                                 \begin{array}{ccc}
                                   m_{11} & m_{12} & m_{13} \\
                                   l_1m_{11} & l_1m_{12} & l_1m_{13} \\
                                   l_2m_{11} & l_2m_{12} & l_2m_{13} \\
                                 \end{array}
                               \right)$$ with $m_{12}l_1^2+m_{13}l_2^2\neq m_{11}, l_1l_2\neq 0$ and $4m_{12}m_{13}l_1^2l_2^2= (m_{12}l_1^2+m_{13}l_2^2-m_{11})^2$. Note that $\mathcal{A}=\mathcal{A}_{\mathcal{O}_{-1}(k^3)}(N)$.
                               In both cases, $\mathcal{A}_{\mathcal{O}_{-1}(k^3)}(M)\cong \mathcal{A}_{\mathcal{O}_{-1}(k^3)}(N)$ by \cite[Theorem B]{MWZ}.
                               When $M=\left(
                                 \begin{array}{ccc}
                                   1 & 1 & 0 \\
                                   1 & 1 & 0 \\
                                   1 & 1 & 0 \\
                                 \end{array}
                               \right)$, we have $$H(\mathcal{A}_{\mathcal{O}_{-1}(k^3)}(M))=\frac{k\langle \lceil x_1-x_2\rceil,\lceil x_1-x_3\rceil\rangle}{(\lceil x_1-x_2\rceil^2)}$$ by Proposition \ref{rankone}(c). By Lemma \ref{nongorsec}, $H(\mathcal{A}_{\mathcal{O}_{-1}(k^3)}(M))$ is not left (right) Gorenstein.
                               If $$M=\left(
                                 \begin{array}{ccc}
                                   m_{11} & m_{12} & m_{13} \\
                                   l_1m_{11} & l_1m_{12} & l_1m_{13} \\
                                   l_2m_{11} & l_2m_{12} & l_2m_{13} \\
                                 \end{array}
                               \right), m_{12}l_1^2+m_{13}l_2^2\neq m_{11}, l_1l_2\neq 0$$ and $4m_{12}m_{13}l_1^2l_2^2= (m_{12}l_1^2+m_{13}l_2^2-m_{11})^2$, then $H(\mathcal{A}_{\mathcal{O}_{-1}(k^3)}(M))$ is  $$\frac{k\langle \lceil l_1x_1-x_2\rceil, \lceil l_2x_1-x_3\rceil \rangle}{(m_{12}\lceil l_1x_1-x_2\rceil^2+m_{13}\lceil l_2x_1-x_3\rceil^2-\frac{\lceil l_1x_1-x_2\rceil \lceil l_2x_1-x_3\rceil+\lceil l_2x_1-x_3\rceil \lceil l_1x_1-x_2\rceil}{\frac{2l_1l_2}{m_{12}l_1^2+m_{13}l_2^2}})}$$ by Proposition \ref{rankone}(a). Since $4m_{12}m_{13}l_1^2l_2^2= (m_{12}l_1^2+m_{13}l_2^2-m_{11})^2$, the graded algebra $H(\mathcal{A}_{\mathcal{O}_{-1}(k^3)}(M))$ is not left (right) graded Gorenstein by Lemma \ref{nongorone}. Thus $H(\mathcal{A})$ is not left (right) graded Gorenstein in both cases.

                                It remains to show the `only if' part.
If $H(\mathcal{A}_{\mathcal{O}_{-1}(k^3)}(N))$ is not left (right) Gorenstein, then $r(N)=1$ by Proposition \ref{rank1}. By \cite[Remark 5.4]{MWZ}, we have $\mathcal{A}_{\mathcal{O}_{-1}(k^3)}(N)\cong \mathcal{A}_{\mathcal{O}_{-1}(k^3)}(M)$, where
$$M=\left(
                                 \begin{array}{ccc}
                                   m_{11} & m_{12} & m_{13} \\
                                   l_1m_{11} & l_1m_{12} & l_1m_{13} \\
                                   l_2m_{11} & l_2m_{12} & l_2m_{13} \\
                                 \end{array}
                               \right),$$
 $(0,0,0)\neq (m_{11},m_{12},m_{13})\in k^3$ and  $l_1,l_2\in k$. By Proposition \ref{rankone}(d-f) and Proposition \ref{ASgor},  we have either
  $$l_1l_2\neq 0, m_{12}m_{13}=0\,\,\text{and} \,\, m_{12}l_1^2+m_{13}l_2^2= m_{11}$$
 or $$l_1l_2\neq 0, m_{12}l_1^2+m_{13}l_2^2\neq m_{11}, 4m_{12}m_{13}l_1^2l_2^2=(m_{12}l_1^2+m_{13}l_2^2-m_{11})^2.$$
By \cite[Proposition 5.8]{MWZ}, there exists $B=(b_{ij})_{3\times 3}\in \mathrm{QPL}_3(k)$ such that
$$ B^{-1}M(b_{ij}^2)_{3\times 3}=\left(
                                 \begin{array}{ccc}
                                   1 & 1 & 0 \\
                                   1 & 1 & 0 \\
                                   1 & 1 & 0 \\
                                 \end{array}
                               \right),$$
if
$l_1l_2\neq 0, m_{12}m_{13}=0\,\,\text{and} \,\, m_{12}l_1^2+m_{13}l_2^2= m_{11}$. In this case,
$$\mathcal{A}_{\mathcal{O}_{-1}(k^3)}(N)\cong \mathcal{A}_{\mathcal{O}_{-1}(k^3)}(M)\cong \mathcal{A}_{\mathcal{O}_{-1}(k^3)}(Q)$$ by \cite[Theorem B]{MWZ}, where $$Q=\left(
                                 \begin{array}{ccc}
                                   1 & 1 & 0 \\
                                   1 & 1 & 0 \\
                                   1 & 1 & 0 \\
                                 \end{array}
                               \right).$$
\end{proof}

Now, we get the following concrete counter examples to disprove Conjecture \ref{biproduct}.
\begin{ex}\label{countex}
Let $\mathcal{A}$ be a connected cochain DG algebra such that
$$\mathcal{A}^{\#}=k\langle x_1,x_2, x_3\rangle/\left(\begin{array}{ccc}
x_1x_2+x_2x_1\\
x_2x_3+x_3x_2\\
x_3x_1+x_1x_3\\
                                 \end {array}\right), |x_1|=|x_2|=|x_3|=1,$$
and $\partial_A$ is determined by
\begin{align*}
\left(
                         \begin{array}{c}
                           \partial_{\mathcal{A}}(x_1)\\
                           \partial_{\mathcal{A}}(x_2)\\
                           \partial_{\mathcal{A}}(x_3)
                         \end{array}
                       \right)=M\left(
                         \begin{array}{c}
                           x_1^2\\
                           x_2^2\\
                           x_3^2
                         \end{array}
                       \right).
\end{align*}
Then by Proposition \ref{nonasgoren},  $H(\mathcal{A})$ is not left (right) Gorenstein when $M$ is one of the following three matrixes:
                                $$ \left(\begin{array}{ccc}
                                   1 & 1 & 0 \\
                                   1 & 1 & 0 \\
                                   1 & 1 & 0\\
                                 \end{array}
                               \right), \left(
                           \begin{array}{ccc}
                             0 & 1 & 1 \\
                             0 & 1 & 1 \\
                             0 & 1 & 1 \\
                           \end{array}
                         \right),\left(
                           \begin{array}{ccc}
                             1 & 1 & 1 \\
                             1 & 1 & 1 \\
                             2 & 2 & 2 \\
                           \end{array}
                         \right).$$
\end{ex}

\section*{Acknowledgments}
The first author is supported by the National Natural Science Foundation of
China (No. 11871326) and the Scientific Research Program of the Higher Education Institution of XinJiang (No. XJEDU2017M032).


\end{document}